\documentclass[10pt,twoside]{article}
\usepackage{amssymb}
\usepackage{amsmath,amscd}
\usepackage{graphicx}
\usepackage{epstopdf}

\usepackage[all,cmtip]{xy}

\newtheorem{thm}{Theorem} 
\newtheorem{cor}[thm]{Corollary} 
\newtheorem{lem}[thm]{Lemma} 
\newtheorem{conj}[thm]{Conjecture} 
\newtheorem{prop}[thm]{Proposition} 
\newtheorem{defn}[thm]{Definition} 
\newtheorem{rem}[thm]{Remark} 

\renewcommand{\int}{\operatorname{int}}
\newcommand{\sign}{\operatorname{sign}}

\newcommand{\Z}{\mathbb{Z}}

\newcommand{\CP}{\mathbb{CP}}

\newcommand{\cT}{\mathcal{T}}
\newcommand{\cW}{\mathcal{W}}
\newcommand{\cV}{\mathcal{V}}
\newcommand{\cM}{\mathcal{M}}

\newcommand{\sL}{\text{\sf{L}}}
\newcommand{\sRL}{\text{\sf{RL}}}

\newcommand{\imra}{\looparrowright}

\def\YEAR{\year}\newcount\VOL\VOL=\YEAR\advance\VOL by-1995
\def\firstpage{1}\def\lastpage{1000}
\def\received{}\def\revised{}
\def\communicated{}

\makeatletter
\def\magnification{\afterassignment\m@g\count@}
\def\m@g{\mag=\count@\hsize6.5truein\vsize8.9truein\dimen\footins8truein}
\makeatother

\font\eightrm=cmr8
\font\caps=cmcsc10                    
\font\Caps=cmcsc10 scaled \magstep1   

\makeatletter
\@specialpagefalse\headheight=8.5pt
\def\DocMath{}
\renewcommand{\@evenhead}{%
    \ifnum\thepage>\lastpage\rlap{\thepage}\hfill%
    \else\rlap{\thepage}\slshape\leftmark\hfill{\caps\SAuthor}\hfill\fi}%
\renewcommand{\@oddhead}{%
    \ifnum\thepage=\firstpage{\DocMath\hfill\llap{\thepage}}%
    \else{\slshape\rightmark}\hfill{\caps\STitle}\hfill\llap{\thepage}\fi}%
\makeatother

\def\TSkip{\bigskip}
\newbox\TheTitle{\obeylines\gdef\GetTitle #1
\ShortTitle  #2
\SubTitle    #3
\Author      #4
\ShortAuthor #5
\EndTitle
{\setbox\TheTitle=\vbox{\baselineskip=20pt\let\par=\cr\obeylines%
\halign{\centerline{\Caps##}\cr\noalign{\medskip}\cr#1\cr}}%
	\copy\TheTitle\TSkip\TSkip%
\def\next{#2}\ifx\next\empty\gdef\STitle{#1}\else\gdef\STitle{#2}\fi%
\def\next{#3}\ifx\next\empty%
    \else\setbox\TheTitle=\vbox{\baselineskip=20pt\let\par=\cr\obeylines%
    \halign{\centerline{\caps##} #3\cr}}\copy\TheTitle\TSkip\TSkip\fi%
\centerline{\caps #4}\TSkip\TSkip%
\def\next{#5}\ifx\next\empty\gdef\SAuthor{#4}\else\gdef\SAuthor{#5}\fi%
\ifx\received\empty\relax
    \else\centerline{\eightrm Received: \received}\fi%
\ifx\revised\empty\TSkip%
    \else\centerline{\eightrm Revised: \revised}\TSkip\fi%
\ifx\communicated\empty\relax
    \else\centerline{\eightrm Communicated by \communicated}\fi\TSkip\TSkip%
\catcode'015=5}}\def\Title{\obeylines\GetTitle}
\def\Abstract{\begingroup\narrower
    \parskip=\medskipamount\parindent=0pt{\caps Abstract. }}
\def\EndAbstract{\par\endgroup\TSkip}

\long\def\MSC#1\EndMSC{\def\arg{#1}\ifx\arg\empty\relax\else
     {\par\narrower\noindent%
     2000 Mathematics Subject Classification: #1\par}\fi}

\long\def\KEY#1\EndKEY{\def\arg{#1}\ifx\arg\empty\relax\else
	{\par\narrower\noindent Keywords and Phrases: #1\par}\fi\TSkip}

\newbox\TheAdd\def\Addresses{\vfill\copy\TheAdd\vfill
    \ifodd\number\lastpage\vfill\eject\phantom{.}\vfill\eject\fi}
{\obeylines\gdef\GetAddress #1
\Address #2 
\Address #3
\Address #4
\EndAddress
{\def\xs{4.3truecm}\parindent=0pt
\setbox0=\vtop{{\obeylines\hsize=\xs#1\par}}\def\next{#2}
\ifx\next\empty 
     \setbox\TheAdd=\hbox to\hsize{\hfill\copy0\hfill}
\else\setbox1=\vtop{{\obeylines\hsize=\xs#2\par}}\def\next{#3}
\ifx\next\empty 
     \setbox\TheAdd=\hbox to\hsize{\hfill\copy0\hfill\copy1\hfill}
\else\setbox2=\vtop{{\obeylines\hsize=\xs#3\par}}\def\next{#4}
\ifx\next\empty\ 
     \setbox\TheAdd=\vtop{\hbox to\hsize{\hfill\copy0\hfill\copy1\hfill}
                \vskip20pt\hbox to\hsize{\hfill\copy2\hfill}}
\else\setbox3=\vtop{{\obeylines\hsize=\xs#4\par}}
     \setbox\TheAdd=\vtop{\hbox to\hsize{\hfill\copy0\hfill\copy1\hfill}
	        \vskip20pt\hbox to\hsize{\hfill\copy2\hfill\copy3\hfill}}
\fi\fi\fi\catcode'015=5}}\gdef\Address{\obeylines\GetAddress}

\begin{document}
\Title
Pulling Apart 2--spheres in 4--manifolds
\ShortTitle 
Pulling Apart 2--spheres in 4--manifolds
\SubTitle   
\Author 
Rob Schneiderman and Peter Teichner
\ShortAuthor 
Schneiderman and Teichner
\EndTitle
\Abstract 
An obstruction theory for representing homotopy classes of surfaces
in 4--manifolds by immersions with pairwise disjoint images is
developed, using the theory of \emph{non-repeating} Whitney towers.
The accompanying higher-order intersection invariants provide a geometric generalization of Milnor's  link-homotopy invariants, and can give the complete obstruction to pulling apart 2--spheres in certain families
of 4--manifolds. It is also shown that in an arbitrary simply connected 4--manifold any number of parallel copies of an immersed $2$--sphere with vanishing self-intersection number can be pulled apart, and that this is not always possible in the non-simply connected setting. The order 1 intersection invariant is shown to be the complete obstruction to pulling apart 2--spheres in any 4--manifold after taking connected
sums with finitely many copies of $S^2\times S^2$; and
the order 2 intersection indeterminacies for quadruples of immersed 2--spheres in a simply-connected 4--manifold are shown to lead to interesting number theoretic questions.
\EndAbstract
\MSC 
Primary 57M99; Secondary 57M25.
\EndMSC
\KEY 
2--sphere, 4--manifold, disjoint immersion, homotopy invariant, non-repeating Whitney tower.
\EndKEY
\Address 
Rob Schneiderman 
Department of Mathematics 
\ \ and Computer Science,
Lehman College, 
\ \ City University of New York, 
New York NY 10468 USA
robert.schneiderman@
\ \ lehman.cuny.edu

\Address
Peter Teichner
Max-Planck Institut 
\ \ f{\"u}r Mathematik
Vivatsgasse 7, 53111 Bonn
Germany
teichner@mac.com
\Address
\Address
\EndAddress

\section{Introduction}
We study the question of whether a map $A:\Sigma \to X$ is homotopic to a map $A'$ such that $A'(\Sigma_i)$ are pairwise {\em disjoint} subsets of $X$, where $\Sigma = \amalg_i \Sigma_i$ is the decomposition into connected components.
In this case, we will say that $A$ {\em can be pulled apart}. 

This question arises as a precursor to the embedding problem -- whether or not $A$ is homotopic to an embedding. 
It also arises in the study of configuration spaces $X^{(m)}$ of $m$ {\em distinct ordered} points in $X$, where elements of $\pi_nX^{(m)}$ 
are represented by $m$ disjoint maps of $n$--spheres
to $X$, and one might ask whether or not a given element of the $m$-fold product $\prod^m\pi_nX$ 
lies in the image of the map $\pi_nX^{(m)}\to\prod^m\pi_nX$ induced by the canonical projections 
$p_1,\dots,p_m: X^{(m)} \to X$.

For example, let $\Sigma=\amalg_i S^n$ be a disjoint union of $n$--spheres and $X$ be a connected $2n$--manifold. For $n\geq 2$, there are Wall's well known intersection numbers $\lambda(A_i,A_j)\in\Z[\pi_1X]$, where $A_i:S^n \to X$ are the components of $A$ \cite{W}. These are obstructions for representing $A$ by an embedding, and the main geometric reason for the success of surgery theory is that, for $n\geq 3$, they are (almost) complete obstructions: The only missing ingredient is Wall's self-intersection invariant $\mu$, a quadratic refinement of $\lambda$. However, for the question of making the $A_i(S^n)$ disjoint, it is necessary and sufficient that $\lambda(A_i,A_j)=0$ for $i\neq j$. We abbreviate this condition on intersection numbers by writing $\lambda_0(A)=0$. 
 
As expected, the condition $\lambda_0(A)=0$ is not sufficient for pulling apart $A$ if $n=2$, but this failure is surprisingly subtle: Given only {\em two} maps $A_1,A_2:S^2 \to X^4$ with $\lambda(A_1,A_2)=0$, one can pull them apart by a clever sequence of finger moves and Whitney moves, see \cite{Ko} and Section~\ref{sec:Wall} below. 
However, this is {\em not} true any more for three (or more) 2--spheres in a 4--manifold.  In \cite{ST1} we defined an additional invariant $\lambda_1(A)$ which takes values in a quotient of $\Z[\pi_1X\times\pi_1X]$ and was shown to be the complete obstruction to pulling apart a {\em triple} $A=A_1,A_2,A_3: S^2 \to X$ of $2$--spheres mapped into an arbitrary $4$--manifold $X$ with  vanishing $\lambda_0(A)$. For trivial
$\pi_1X$  the analogous obstruction was defined earlier in \cite{Ma,Y}. 

In this paper, we extend this work to an arbitrary number of 2--spheres 
(and other surfaces -- see Remark~\ref{rem:pi1-null}) 
in $4$--manifolds. 
The idea is to apply a variation of the theory of {\em Whitney towers} as developed in
\cite{CST,CST0,CST1,CST2,S3,ST1,ST2} to address the problem. Before we introduce the relevant material on Whitney towers, we mention a couple of new results that can be stated without prerequisites. 

Throughout this paper the letter $m$ will usually denote the number of surface components to be pulled apart, and from now on the letters $\Sigma$ and  $X$ will be used to denote
surfaces and $4$--manifolds, respectively. The distinction between a map of a surface and its image in $X$ will frequently be disregarded in the interest of brevity.

\subsection*{Pulling apart parallel 2--spheres}
The following theorem is discussed and proven in Section~\ref{sec:parallel-thm-proof}:
\begin{thm}\label{thm:parallel}
If $X$ is a simply connected $4$--manifold and $A:\amalg^m S^2\to X$ consists of $m$ copies of {\em the same} map $A_0:S^2 \to X$ of a $2$--sphere with trivial normal Euler number, then $A$ can be pulled apart if and only if $[A_0]\in H_2(X;\Z)$ has vanishing homological self-intersection number $[A_0]\cdot [A_0]=0\in\Z$.
\end{thm}
Note that each transverse self-intersection of $A_0$ gives rise to $m^2-m$ intersections among the $m$ parallel copies $A$, not counting self-intersections, see Figure~\ref{fig:parallel-8-intersections-with-dots}. As a consequence, there cannot be a simple argument to pull $A$ apart. In fact, the analogous statement fails for non-simply connected 4--manifolds, see Example~\ref{example:parallel-counter-example}. 
\begin{figure}[ht!]
         \centerline{\includegraphics[scale=.5]{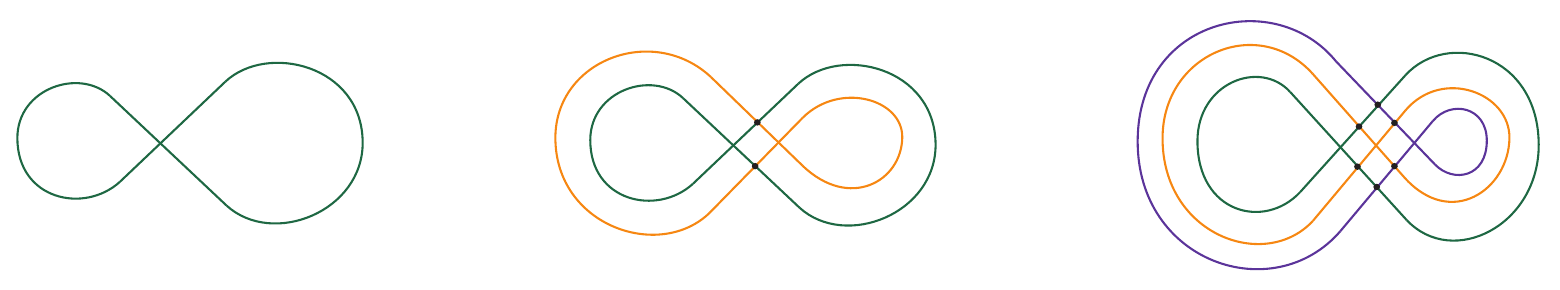}}
         \caption{One self-intersection leads to $m^2-m$ intersections among $m$ copies.}
         \label{fig:parallel-8-intersections-with-dots}
\end{figure}

\subsection*{Stably pulling apart 2--spheres}
We say that surfaces $A:\Sigma\to X$ can be \emph{stably} pulled apart if $A$ can be pulled apart after taking the connected sum of $X$ with finitely many copies of $S^2\times S^2$. The invariants $\lambda_0$ and $\lambda_1$ are unchanged by this stabilization, and in this setting they give the complete obstruction
to pulling apart $m$ maps of $2$--spheres: 

\begin{thm}\label{thm:stable-separation}
$A:\amalg^m S^2\to X$ can be stably pulled apart if and only if 
$\lambda_0(A)=0=\lambda_1(A)$. 
\end{thm}
This result also holds when the stabilizing factors $S^2\times S^2$ are replaced by any simply-connected closed 4-manifolds (other than $S^4$). It also holds when components of $A$ are maps of disks. The invariant 
$\lambda_1$ is described precisely in sections~\ref{sec:whitney-towers}
 and \ref{sec:order2-INT}; and the proof of Theorem~\ref{thm:stable-separation} is given in section~\ref{sec:stable-thm-proof}. (Note that $X$ is not required to be simply connected in Theorem~\ref{thm:stable-separation}.)
We remark that the stronger invariant $\tau_1(A)$ of \cite{ST1}, together with Wall's self-intersection invariant $\tau_0(A)$, is the complete obstruction to \emph{stably embedding} $A$, see \cite[Cor.1]{S3}. 

\begin{rem} \label{rem:clean-up}
The question of pulling apart surfaces in 4--manifolds is independent of category. More precisely, any connected 4--manifold can be given a smooth structure away from one point \cite{FQ} and any continuous map can be approximated arbitrarily closely by a smooth map. As a consequence, we can work in the smooth category and as a first step, we can always turn a map $A:\Sigma^2 \to X^4$ into a generic immersion. We will also assume that surfaces are \emph{properly} immersed, i.e.~$A(\partial \Sigma) \subset \partial X$ with the interior of $\Sigma$ mapping to the interior of $X$, and that homotopies fix the boundary. 
\end{rem}

\subsection{Pulling apart {\em two} 2--spheres in a 4--manifold} \label{sec:Wall}
To motivate the introduction of Whitney towers into the problem, it is important to understand the basic case of pulling apart {\em two} maps of 2--spheres $A_1,A_2 :S^2 \to X$. 
 Wall's intersection pairing associates a sign and an element of $\pi_1X$ to each
transverse intersection point between the surfaces, and the
vanishing of $\lambda(A_1,A_2)$ implies that all of these
intersections can be paired by {\em Whitney disks}. As illustrated in Figures~\ref{Push-downA1andA2-fig}
and~\ref{Push-downB1andB2-fig}, these Whitney
disks can be used to pull apart $A_1$ and $A_2$ by first pushing
any intersection points between $A_2$ and the interior of a
Whitney disk $W_{(1,2)}$ down into $A_2$, and then using the Whitney disks to
guide Whitney moves on $A_1$ to eliminate all intersections
between $A_1$ and $A_2$ (details in \cite{Ko}).
\begin{figure}[ht!]
         \centerline{\includegraphics[scale=.5]{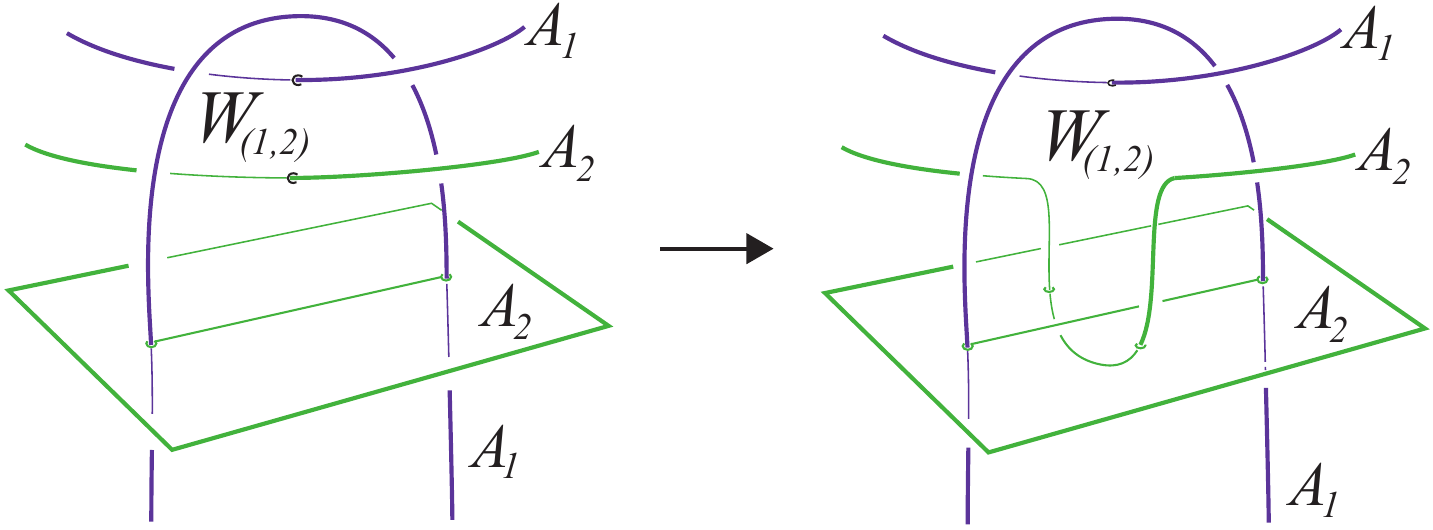}}
         \caption{Pushing an intersection between 
         $A_2$ and the interior of a Whitney disk $W_{(1,2)}$ down into $A_2$ only creates (two) self-intersections in $A_2$.}
         \label{Push-downA1andA2-fig}
\end{figure}
\begin{figure}[ht!]
         \centerline{\includegraphics[scale=.5]{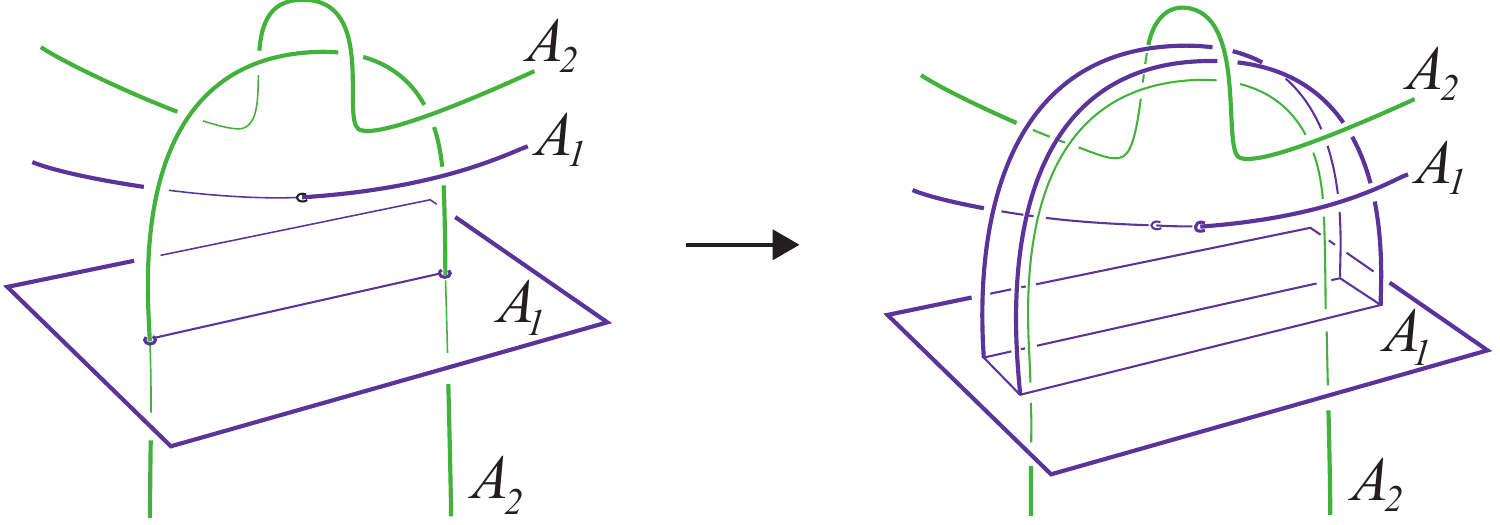}}
         \caption{A Whitney move guided by the Whitney disk of Figure~\ref{Push-downA1andA2-fig}. The intersection between $A_1$ and the interior of the Whitney disk becomes a pair of self-intersections of $A_1$ after the Whitney move.}
         \label{Push-downB1andB2-fig}
\end{figure}

\subsection{Pulling apart {\em three or more} 2--spheres} \label{subsec:3-or-more-intro}
Note that for a triple of spheres one cannot use the 
method of figures~\ref{Push-downA1andA2-fig} and \ref{Push-downB1andB2-fig} to eliminate an
intersection point between one sphere and a Whitney disk that
pairs intersections between the \emph{other} two spheres. Such
``higher-order'' intersections were used in \cite{ST1} to define the invariant
$\lambda_1(A)$ discussed above. In this case, the procedure for
separating the surfaces involves constructing ``second order''
Whitney disks which pair the intersections between surfaces and
Whitney disks. The existence of these second order Whitney disks
allows for an analogous pushing-down procedure which only creates
self-intersections and cleans up the Whitney disks enough to pull
apart the surfaces by an ambient homotopy.

Building on these ideas, we will describe an obstruction theory in
terms of {\em non-repeating Whitney towers} $\cW$ built on 
properly immersed surfaces in $X$, and {\em non-repeating
intersection invariants} $\lambda_n(\cW)$ taking values in quotients of the group ring
of $(n+1)$ products of $\pi_1X$. The {\em
order} $n$ of the non-repeating Whitney tower $\cW$ determines how many of the
underlying surfaces at the bottom of the tower can be made
pairwise disjoint by a homotopy, and the vanishing of $\lambda_n(\cW)$ is sufficient to find an order $n+1$
non-repeating Whitney tower. 

Non-repeating Whitney towers
are special cases of the Whitney towers defined in \cite{ST2} (see
also \cite{CST,CST1,CST2,CST3,S1,S2,ST1}). An introduction to these
notions is sketched here with details given in Section~\ref{sec:whitney-towers}.
We work in the smooth oriented category, with orientations usually suppressed.

\subsection{Whitney towers and non-repeating Whitney towers}\label{sec:intro-w-towers}
Consider $A:\Sigma=\amalg_i\Sigma_i^2 \to X^4$ where the surface components $\Sigma_i$ are spheres or disks (and see Remark~\ref{rem:clean-up} for initial clean-ups on $A$).  To begin our obstruction theory, we say that $A$ forms a {\em Whitney tower of order 0}, and define the \emph{order} of each properly immersed connected surface $A_i:\Sigma_i\to
 X$ to be zero. 

If all the singularities (transverse intersections) of $A$ can be paired by Whitney disks then
we get a {\em Whitney tower of order 1} which is the union of these
{\em order $1$ Whitney disks} and the order $0$ Whitney tower. 

If we only have Whitney disks pairing the intersections between \emph{distinct} order $0$ surfaces $A_i:\Sigma_i \to
 X$ of $A$, then we get an {\em order $1$ non-repeating
Whitney tower}. 

If it exists, an {\em order $2$ Whitney tower} 
also includes Whitney disks (of {\em order $2$}) pairing all the
intersections between the order $1$ Whitney disks and the order $0$
surfaces. An {\em order $2$ non-repeating Whitney tower} only
requires second order Whitney disks for intersections between an
$A_i$ and Whitney disks pairing $A_j$ and $A_k$, where $i$, $j$
and $k$ are distinct. As explained in Section~\ref{sec:whitney-towers}, all of this generalizes to higher order,
including the distinction between non-repeating and repeating
intersection points, however things get more subtle as different
``types'' of intersections of the same order can appear (parametrized by isomorphism classes of unitrivalent trees). 

An \emph{order $n$ Whitney tower} has Whitney disks pairing up all
intersections of order less than $n$, and an \emph{order $n$ non-repeating Whitney tower}
is only required to have Whitney disks pairing all non-repeating intersections of order less than $n$ (sections~\ref{sec:w-tower-def} and \ref{sec:non-repeating-w-tower-def}).
So ``order $n$ non-repeating'' is a weaker condition than ``order $n$''.

The underlying order $0$ surfaces $A$ in a Whitney tower $\cW$ are said to \emph{support} $\cW$, and we say that $A$ \emph{admits} an order $n$ Whitney tower if $A$ is homotopic (rel boundary) to $A'$ supporting $\cW$ of order $n$.

\subsection{Pulling apart surfaces in $4$--manifolds}\label{sec:intro-pulling-apart}
As a first step towards determining whether or not $A:\Sigma\to X$ can
be pulled apart, we have the following translation of the problem
into the language of Whitney towers. This is the main tool in our theory:
\begin{thm}\label{thm:tower-separates} Let $m$ be the number of components of $\Sigma$. Then
$A:\Sigma\to X$ can be pulled apart if and only if $A$ admits a non-repeating Whitney tower of order $m-1$.
\end{thm}
The existence of a non-repeating Whitney tower of sufficient order
encodes ``pushing down'' homotopies and Whitney moves which lead
to disjointness, as will be seen in the proof of
Theorem~\ref{thm:tower-separates}
given in Section~\ref{sec:thm-tower-separates-proof}. 
It will be clear from the proof of Theorem~\ref{thm:tower-separates} that 
for $1< n<m$ the existence of a non-repeating Whitney tower of order $n$
implies that any $n+1$ of the order $0$ surfaces $A_i:\Sigma_i\to
 X$ can be pulled apart.

\subsection{Higher-order intersection invariants}\label{sec:intro-higher-order-ints}
An immediate advantage of this point of view is that the higher-order
intersection theory of \cite{ST2} can be applied inductively to
increase the order of a Whitney tower or, in some cases, detect
obstructions to doing so. 
The main idea is that to each unpaired
intersection point $p$ in a Whitney tower $\cW$ on $A$ one can
associate a decorated unitrivalent tree $t_p$ which bifurcates
down from $p$ through the Whitney disks to the order 0 surfaces
$A_i$ (Figure~\ref{fig:W-disks-inner-prod-oriented}, also Figure~\ref{bing-hopf-example-tower-and-grope-with-tree}). 
The \emph{order} of $p$ is the number of trivalent vertices
in $t_p$. The univalent vertices of $t_p$ are labeled by the elements $i\in \{1,\dots,m\}$ from the set indexing the $A_i$. 
The edges of $t_p$ are
decorated with elements of the fundamental group $\pi:=\pi_1X$ of
the ambient $4$--manifold $X$. Orientations of $A$ and $X$ determine
vertex-orientations and a sign $\sign(p)\in\{\pm\}$ for $t_p$,
and the {\em order $n$ intersection invariant} $\tau_n(\cW)$ of an order $n$ Whitney tower $\cW$ is defined
as the sum
$$
\tau_n(\cW):=\sum \sign(p) \cdot t_p \in\cT_n(\pi,m)
$$
over all order $n$ intersection points $p$ in $\cW$. Here
$\cT_n(\pi,m)$ is a free abelian group generated by order $n$
decorated trees modulo relations which include the usual
antisymmetry (AS) and Jacobi (IHX) relations of finite type theory
(Figure~\ref{fig:Relations-fig}). 
Restricting to non-repeating intersection points in an order $n$
non-repeating Whitney tower $\cW$, yields the analogous order $n$
{\em non-repeating intersection invariant} $\lambda_n(\cW)$:
$$
\lambda_n(\cW):=\sum \sign(p) \cdot t_p \in\Lambda_n(\pi,m)
$$
which takes values in the subgroup $\Lambda_n(\pi,m)<\cT_n(\pi,m)$
generated by order $n$ trees whose univalent vertices have \emph{distinct}
labels. 
We refer to Definition~\ref{def:tau-lambda} for more precise statements. 
In the following 
we shall sometimes suppress the number $m$ of components of $\Sigma$ and just write $\Lambda_n(\pi)$. 

\begin{rem}\label{rem:intro-target-tree-groups}
We will show in Lemma~\ref{lem:Lambda-computation} that $\Lambda_n(\pi,m)$ is isomorphic to the direct sum of $\binom{m}{n+2}n!$-many copies of the integral group ring $\Z[\pi^{(n+1)}]$ of the $(n+1)$-fold cartesian product $\pi^{(n+1)}=\pi\times\pi\times\cdots\times\pi$.
Note that $\Lambda_n(\pi,m)$ is trivial for $n\geq m-1$ since an order $n$ unitrivalent tree has $n+2$ univalent vertices.
For $\pi$ left-orderable, $\cT_1(\pi,m)$ is computed in \cite[Sec.2.3.1]{S3}. For $\pi$ trivial, $\cT_n(m):=\cT_n(1,m)$ is computed in \cite{CST3} for all $n$, and in \cite{CST1} the torsion subgroup of $\cT_n(m)$ (which is only $2$-torsion) is shown to correspond to obstructions to ``untwisting'' Whitney disks in \emph{twisted Whitney towers} in the $4$--ball. The absence of torsion in
$\Lambda_n(\pi,m)$ corresponds to the fact that such obstructions are not relevant in the non-repeating setting since
a \emph{boundary-twisting} operation \cite[Sec.1.3]{FQ}
can be used to eliminate non-trivially twisted Whitney disks
at the cost of only creating repeating intersections.
\end{rem}

In the case $n=0$, our notation
$\lambda_0(A)\in\Lambda_0(\pi)$ just describes
Wall's Hermitian intersection pairing $\lambda(A_i,A_j)\in\Z[\pi]$
(see section~\ref{sec:order-zero-int-trees}).

For $n=1$, we showed in \cite{ST1} that if
$\lambda_0(A)=0$ then taking
$\lambda_1(A):=\lambda_1(\cW)$ in an appropriate quotient of $\Lambda_1(\pi)$ defines a homotopy invariant of  $A$ (independent of the
choice of non-repeating Whitney tower $\cW$). See sections~\ref{sec:order-one-int-trees} and \ref{sec:order-1-INT}.

The main open problem in this intersection theory is to determine for $n\geq 2$ the largest quotient of $\Lambda_n(\pi)$ for which $\lambda_n(\cW)$ only depends on the homotopy class of $A:\Sigma\to X$. Even for $n=1$, this quotient will generally depend on $A$, unlike Wall's invariants $\lambda_0$.

\subsection{The geometric obstruction theory}
In Theorem~2 of \cite{ST2} it was shown that the vanishing of
$\tau_n(\cW)\in\cT_n(\pi)$ implies that $A$ admits an
order $n+1$ Whitney tower. The proof of this result uses controlled geometric realizations of the relations in 
$\cT_n(\pi)$, and the exact same constructions (which are all homogeneous in the univalent labels -- see Section~4 of \cite{ST2}) give the analogous result in the non-repeating setting:

\begin{thm}\label{thm:non-rep-obstruction-theory}
If $A:\Sigma\to X$ admits a
non-repeating Whitney tower $\cW$ of order~$n$
with $\lambda_n(\cW)=0\in\Lambda_n(\pi)$, then $A$ admits an
order $(n+1)$ non-repeating Whitney tower.$\hfill\square$
\end{thm} 

Combining Theorem~\ref{thm:non-rep-obstruction-theory} with Theorem~\ref{thm:tower-separates} above yields the following result, which was announced in \cite[Thm.3]{ST2}:
\begin{cor}
\label{cor:tower=disjoint} 
If $\Sigma$ has $m$ components and $A:\Sigma\to X$ admits a
non-repeating Whitney tower $\cW$ of order~$m-2$ such that
$\lambda_{m-2}(\cW)=0\in\Lambda_{m-2}(\pi)$, then 
$A$ can be pulled apart.$\hfill\square$
\end{cor}
Thus, the problem of deciding whether or not any given $A$ can be pulled apart 
can be attacked inductively by determining the extent to which $\lambda_n(\cW)$ only depends on the homotopy class of $A$.  


The next two subsections describe settings where $\lambda_n(\cW)\in\Lambda_n(\pi)$
does indeed tell the whole story. For Whitney towers in simply connected $4$--manifolds, we drop $\pi$ from the notation, writing $\Lambda_n(m)$, or just $\Lambda_n$ if the number of order $0$ surfaces is understood.

\subsection{Pulling apart disks in the 4--ball}
A \emph{link-homotopy} of an $m$-component link $L=L_1\cup
L_2\cup\dots\cup L_m$ in the $3$--sphere is a homotopy of
$L$ which preserves disjointness of the link components, i.e.~during 
the homotopy only self-intersections of the $L_i$ are
allowed. In order to study ``linking modulo knotting'',  Milnor
\cite{M1} introduced the equivalence relation of link-homotopy
  and defined his (non-repeating) $\mu$-invariants, showing in particular that a
link is link-homotopically trivial if and only if it has all
vanishing $\mu$-invariants. In the setting of link-homotopy,
Milnor's algebraically defined $\mu$-invariants are intimately
connected to non-repeating intersection invariants as implied by the following result, proved in Section~\ref{sec:link-htpy-thm-proof}
using a new notion of \emph{Whitney tower-grope duality} (Proposition~\ref{prop:w-tower-duality}).

\begin{thm}\label{thm:lambda-link-htpy}
Let $L$ be an $m$-component link in $S^3$ bounding
$D: \amalg^m D^2 \to B^4$. If $D$ admits an order $n$ non-repeating Whitney tower $\cW$ then $\lambda_n(\cW)\in\Lambda_n(m)$ does not depend on the choice of $\cW$. In fact, $\lambda_n(\cW)$ contains the same information as all non-repeating Milnor invariants of length $n+2$ and it is therefore a link-homotopy invariant of $L$.
\end{thm}
We refer to Theorem~\ref{thm:mu} for a precise statement on how Milnor's invariants are related to
$\lambda_n(L):=\lambda_n(\cW)\in\Lambda_n(m)$. 
 Together with
Corollary~\ref{cor:tower=disjoint} we get the following result:
\begin{cor}\label{cor:lambda(L)}An $m$-component link $L$ is
link-homotopically trivial if and only if $\lambda_n(L)$ vanishes
for all $n=0,1,2,\ldots,m-1$. $\hfill\square$
\end{cor}
This recovers Milnor's characterization of links which are link-homotopically trivial \cite{M1}, and uses the fact that $L$
bounding disjointly immersed disks into $B^4$ is equivalent to $L$ being
link-homotopically trivial \cite{Gi,Go}.

\begin{rem} A precise description of the relationship
between general (repeating) Whitney towers on $D$ and
Milnor's $\mu$-invariants (with repeating indices
\cite{M2}) for $L$ is given in \cite{CST2}. Our current discussion is both easier and harder at the same time: We only make a statement about non-repeating Milnor invariants, a subset of all Milnor invariants, but as an input we only use a non-repeating Whitney tower, an object containing less information then a Whitney tower.
\end{rem}

\subsection{Pulling apart $2$--spheres in special $4$--manifolds}\label{sec:intro-simply-connected-$4$--manifolds}
The relationship between Whitney towers and Milnor's link invariants can be used to describe some more general settings where the non-repeating intersection invariant $\lambda_n(\cW)\in\Lambda_n$ of a non-repeating Whitney tower gives homotopy invariants of the underlying order $0$ surfaces.
Denote by $X_L$ the $4$--manifold which is gotten by attaching $0$-framed $2$--handles to the $4$--ball along 
a link $L$ in the $3$--sphere. The following theorem is proved in Section~\ref{sec:proof-thm-in-X-L}:
\begin{thm}\label{thm:in-X-L}
If a link $L$ bounds an order $n$ Whitney tower on disks in the $4$--ball, then:
\begin{enumerate}
\item Any map $A:\amalg^m S^2\to X_L$ of $2$--spheres into $X_L$ admits an order $n$
Whitney tower.
\item For any order $n$ non-repeating Whitney tower $\cW$ supported by $A:\amalg^m S^2\to X_L$, the non-repeating intersection invariant $\lambda_n(A):=\lambda_n(\cW)\in\Lambda_n(m)$ is independent of the choice of $\cW$. 
\end{enumerate}
\end{thm}
Note that the number $m$ of $2$--spheres need not be equal to the number of components of the link $L$. 
Using the realization techniques for Whitney towers in the $4$--ball described in \cite[Sec.3]{CST1}, examples of such $A$ realizing any value in $\Lambda_n(m)$ can be constructed.
\begin{cor}\label{cor:X-L}
If a link $L$ bounds an order $n$ Whitney tower on disks in the $4$--ball, then:
\begin{enumerate}
\item
$A:\amalg^m S^2\to X_L$ admits an order $n+1$ non-repeating Whitney tower if and only if $\lambda_n(A)=0\in\Lambda_n(m)$. 

\item
In the case $m=n+2$, we have that $A:\amalg^m S^2\to X_L$ can be pulled apart if and only if $\lambda_{m-2}(A)=0\in\Lambda_{m-2}(m)$. $\hfill\square$

\end{enumerate}
\end{cor}
The ``if'' parts of the statements in Corollary~\ref{cor:X-L} follow from Theorem~\ref{thm:non-rep-obstruction-theory} and Corollary~\ref{cor:tower=disjoint} above. 
 The ``only if'' statements follow from the fact that the second statement of Theorem~\ref{thm:in-X-L} implies that $\lambda_n(A):=\lambda_n(\cW)\in\Lambda_n(m)$ only depends on the homotopy class of $A$, as explained in Proposition~\ref{prop:fixed-immersion} below.
In this setting, Kojima \cite{Koj} had identified (via Massey products) the first non-vanishing Milnor invariant $\mu_L(123\cdots m)$ of an $m$-component link $L$ as an obstruction to pulling apart the collection of $m$ $2$--spheres determined by $L$ in $X_L$.

\subsection{Indeterminacies from lower-order intersections.}\label{sec:intro-INT-indeterminacies}
The sufficiency results of Theorem~\ref{thm:non-rep-obstruction-theory}
and Corollary~\ref{cor:tower=disjoint} show that the groups $\Lambda_n(\pi)$
provide upper bounds on the invariants needed for a complete obstruction
theoretic answer to the question of whether or not $A:\Sigma\to X$ can be pulled apart. 
And as illustrated by
Theorem~\ref{thm:lambda-link-htpy} and Theorem~\ref{thm:in-X-L} above, 
there are settings in which $\lambda_n(\cW)\in\Lambda_n$ 
only depends on the homotopy class (rel boundary) of $A$, sometimes giving the complete obstruction 
to pulling $A$ apart.

In general however, more relations are needed in the target group
to account for indeterminacies in the choices of possible Whitney
towers on a given $A$. In particular, for Whitney
towers in a $4$--manifold $X$ with non-trivial second homotopy group $\pi_2X$, there can be
indeterminacies which correspond to tubing the interiors of
Whitney disks into immersed $2$--spheres. 
Such INT \emph{intersection relations} are, in principle, inductively
manageable in the sense that they are determined by strictly lower-order 
intersection invariants on generators of $\pi_2X$. 
For instance, the INT$_1$ relations in the target groups of the order $1$
invariants $\tau_1$ and $\lambda_1$ of \cite{Ma,ST1}
are determined by the order $0$ intersection  
form 
on $\pi_2X$.
However, as we
describe in Section~\ref{sec:order2-INT}, higher-order INT relations can be non-linear, and if one wants the
resulting target to carry exactly the obstruction to the
existence of a higher-order tower then interesting subtleties already arise in the order $2$ setting.

It is interesting to note that these INT indeterminacies are generalizations of
the Milnor-invariant indeterminacies in that they may involve intersections between $2$--spheres
\emph{other} than the $A_i$. The Milnor link-homotopy invariant indeterminacies come from sub-links
because there are no other essential 2-spheres in $X_L$. 
For instance, the proof of Theorem~\ref{thm:stable-separation} exploits the hyperbolic
summands of the stabilized intersection form on $\pi_2$. We pause here to note another positive consequence of the intersection indeterminacies before returning to further
discussion of the well-definedness of the invariants.

\subsubsection{Casson's separation lemma}\label{subsubsec:Casson-sep-lemma} 
The next theorem shows that in the presence of
algebraic duals for the order $0$ surfaces $A_i$, all our higher-order obstructions vanish.
This recovers the following result of Casson (proved algebraically in the
simply-connected setting \cite{Casson}) and Quinn (proved using
transverse spheres \cite{Edwards,Quinn}):
\begin{thm}\label{thm:duals}
If $\lambda(A_i,A_j)=0$ for all $i\neq j$, and there exist 
2--spheres $B_i:S^2\to X$ such that
$\lambda(A_i,B_j)=\delta_{ij}$ for all $i$, then 
$A_i$ can be pulled apart.
\end{thm}
Here $\lambda$ denotes Wall's intersection pairing with values in
$\Z[\pi]$, and $\delta_{ij}\in\{0,1\}$ is the Kronecker delta. Note that there
are no restrictions on  intersections among the dual spheres $B_i$.
Theorem~\ref{thm:duals} is proved in section~\ref{sec:casson-lemma-proof}.

\subsubsection{Homotopy invariance of higher-order intersection invariants}\label{subsubsec:htpy-invariance}
Our proposed program for pulling apart $2$--spheres in $4$--manifolds involves refining 
Theorem~\ref{thm:non-rep-obstruction-theory} by formulating (and computing)
the relations INT$_n(A)\subset\Lambda_n(\pi)$ so that $\lambda_n(A):=\lambda_n(\cW)\in\Lambda_n(\pi)/\mathrm{INT}_n(A)$ is a homotopy invariant of $A$ (independent of the choice of order $n$ non-repeating Whitney tower $\cW$) which represents
the complete obstruction to the existence
of an order $n+1$ non-repeating tower supported by $A$. Via Theorem~\ref{thm:tower-separates} this would provide a
procedure to determine whether or not $A$ can be pulled apart. The following observation clarifies what needs to be shown:
\begin{prop}\label{prop:fixed-immersion}
If for a \emph{fixed} immersion $A$ the value of $\lambda_n(\cW)\in\Lambda_n(\pi)/\mathrm{INT}_n(A)$ does not depend the choice of order $n$ non-repeating Whitney tower $\cW$ supported by $A$, then 
$\lambda_n(A):=\lambda_n(\cW)\in\Lambda_n(\pi)/\mathrm{INT}_n(A)$ only depends on the
homotopy class of $A$.$\hfill\square$
\end{prop}
To see why this is true, observe that, up to isotopy, any generic regular homotopy from $A$ to $A'$ can be realized as a sequence of finitely many finger moves followed by finitely many Whitney moves. Since any Whitney move has a finger move as an ``inverse'', there exists $A''$ which differs from each of $A$ and $A'$ by only finger moves (up to isotopy). But a finger move is supported near an arc, which can be assumed to be disjoint from the Whitney disks in a Whitney tower, and the pair of intersections created by a finger move admit a local Whitney disk; so any Whitney tower on $A$ or $A'$ gives rise to a Whitney tower on $A''$ with the same intersection invariant.

Thus, the problem is to find INT$_n(A)$ relations which give independence of the choice of $\cW$ for a fixed immersion $A$, and can be realized geometrically so that  
$\lambda_n(\cW)\in\mathrm{INT}_n(A)$ implies that $A$ bounds an order $n+1$ non-repeating Whitney tower.
We conjecture that all these needed relations do indeed 
correspond to lower-order intersections involving $2$--spheres, and hence deserve to be called ``intersection'' relations.
Although such INT$_n(A)$ relations are completely understood for $n=1$ (see \ref{sec:order-1-INT} below),
a precise formulation for the $n=2$ case already presents interesting subtleties. We remark that for maps of higher genus surfaces there can also be indeterminacies (due to choices of boundary arcs of Whitney disks) which do not come from $2$--spheres; see \cite{S3} for the order $1$ invariants of immersed annuli.

Useful necessary and sufficient conditions for pulling apart four or more $2$--spheres in an arbitrary $4$--manifold are not currently known. In Section~\ref{sec:order2-INT} we examine the intersection indeterminacies for the relevant order 2 non-repeating
intersection invariant $\lambda_2$ in the simply connected setting, and show how they can be computed as the image 
in $\Lambda_2(4)\cong \Z^2$ of a map whose non-linear part
is determined by certain Diophantine quadratic equations
which are coupled by the intersection form on $\pi_2X$ (see section~\ref{subsubsec:quadratic-INT2-maps}).
Carrying out this computation in general raises interesting number theoretic questions, and has motivated work
of Konyagin and Nathanson in \cite{Kon}.

We'd like to pose the following {\em challenge}: Formulate the $\mathrm{INT}_n(A)$ relations for $n\geq 2$ which make the following conjecture precise and true:
\begin{conj}\label{conj:INT}
$A:\amalg^m S^2\to X$ can be pulled apart if and only if $\lambda_n(A):=\lambda_n(\cW)$
vanishes in $\Lambda_n(\pi)/\mathrm{INT}_n(A)$ for $n=0,1,2,3,\ldots,m-2$.
\end{conj}


\section{Whitney towers}\label{sec:whitney-towers}

This section contains a summary of relevant Whitney tower notions and
notations as described in more detail in \cite{CST,CST1,CST2,S1,S2,S3,ST1,ST2}.
Recall our blurring of the distinction between a map $A:\Sigma\to X$ and its image, which leads us to speak of
$A$ as a ``collection'' of immersed connected surfaces in $X$.

\begin{rem}\label{rem:pi1-null}
Although this paper focuses on pulling apart $A$ in the case where the components $\Sigma_i$ of $\Sigma$ are spheres and/or disks, much of the
discussion is also relevant to the \emph{$\pi_1$-null} setting; i.e.~the $\Sigma_i$ are compact connected surfaces of arbitrary genus and the component maps
$A_i:\Sigma_i\to X$ induce trivial maps 
$\pi_1\Sigma_i\to\pi_1X$
on fundamental groups.
\end{rem}

\subsection{Whitney towers}\label{sec:w-tower-def}
The following formalizes the discussion from the introduction by inductively defining Whitney towers
of order $n$ for each non-negative integer $n$.
\begin{defn}\label{w-tower-defn}\mbox{}
\begin{itemize}
\item A {\em surface of order 0} in a $4$--manifold $X$
is a properly immersed connected compact surface (boundary embedded in the boundary
of $X$ and interior immersed in the interior of $X$). A {\em
Whitney tower of order 0} in $X$ is a collection of order 0
surfaces.
\item The {\em order of a (transverse) intersection point} between a surface of order $n_1$ and a
surface of order $n_2$ is $n_1+n_2$.
\item The {\em order of a Whitney disk} is $n+1$ if it pairs intersection points of order $n$.
\item For $n\geq 0$,
a {\em Whitney tower of order $n+1$} is a Whitney tower $\cW$ of
order $n$ together with Whitney disks pairing all order $n$
intersection points of $\cW$. These order $n+1$ Whitney disks are allowed to
intersect each other as well as lower-order surfaces.
\end{itemize}

The Whitney disks in a Whitney tower are required to be {\em
framed} \cite{CST1,FQ,ST1} and have disjointly embedded boundaries.
Each order $0$ surface in a Whitney tower is also required to be framed, in the sense that its normal bundle in $X$ has trivial
(relative) Euler number.
Interior intersections are assumed to be transverse. A
Whitney tower is {\em oriented} if all its surfaces (order $0$
surfaces {\em and} Whitney disks) are oriented. Orientations and framings on any 
boundary components of order $0$ surfaces are required to be compatible
with those of the order $0$ surfaces.  
A {\em based}
Whitney tower includes a chosen basepoint on each surface
(including Whitney disks) together with a {\em whisker} (arc) for
each surface connecting the chosen basepoints to the basepoint of
$X$. 
\end{defn}
We will assume our Whitney towers are based and oriented, although
whiskers and orientations will usually be suppressed from
notation. The collection $A$ of order $0$ surfaces in a Whitney tower $\cW$ is said to \emph{support} $\cW$, and we also say that $\cW$ is a Whitney tower \emph{on} $A$.
A collection $A$ of order $0$ surfaces is said to
\emph{admit} 
an order $n$ Whitney tower if $A$ is homotopic (rel boundary)
to $A'$ supporting an order $n$ Whitney tower.

\subsection{Trees for Whitney disks and intersection points.}
In this paper, a {\em tree} will always refer to a finite oriented unitrivalent tree, where the {\em (vertex) orientation} of a tree is given by cyclic orderings of the adjacent edges around each trivalent vertex. The \emph{order} of a tree is the number of trivalent vertices.
Univalent vertices will usually be labeled from the set $\{1,2,3,\ldots,m\}$ indexing the order $0$ surfaces, and we consider trees up to isomorphisms preserving these labelings.
A tree is \emph{non-repeating} if its
univalent labels are distinct. When $X$ is not simply connected, edges will be oriented
and labeled with elements of $\pi_1X$. A \emph{root} of a tree
is a chosen univalent vertex (usually left un-labeled).

We start by considering the case where $X$ is simply connected:

Formal non-associative bracketings of elements from the index set are used as subscripts to
index surfaces in a Whitney tower $\cW\subset X$, writing $A_i$ for an
order~$0$ surface (dropping the brackets around the singleton
$i$), $W_{(i,j)}$ for an order~$1$ Whitney disk that pairs
intersections between $A_i$ and $A_j$, and $W_{((i,j),k)}$ for an order~$2$
Whitney disk pairing intersections between
$W_{(i,j)}$ and $A_k$, and so on, with the ordering of the bracket components determined 
by an orientation convention described below (\ref{sec:orientations}). When writing $W_{(I,J)}$ for a
Whitney disk pairing intersections between $W_I$ and $W_J$, the
understanding is that if a bracket $I$ is just a singleton $i$
then the surface $W_I=W_{i}$ is just the order~zero surface
$A_i$. Note that both Whitney disks and order $0$ surfaces are referred to as ``surfaces in $\cW$''.

Via the usual correspondence 
between non-associative brackets and rooted trees, this indexing gives a correspondence between surfaces in $\cW$
 and rooted trees: To a Whitney disk $W_{(I,J)}$
we associate the rooted tree corresponding to the bracket $(I,J)$. We use the same notation for rooted trees and brackets,
so the bracket operation
corresponds to the \emph{rooted
product} of trees which glues together the root vertices of $I$ and $J$ to a single vertex and sprouts a new rooted edge from this vertex.
With this notation the order of a Whitney disk $W_K$ is equal to the order of (the rooted tree) $K$.

The rooted tree
$(I,J)$ associated to $W_{(I,J)}$ can be considered to be a subset of $\cW$, with its root edge
(including the root edge's trivalent vertex) sitting in the interior of
$W_{(I,J)}$, and its other edges bifurcating down through lower-order
Whitney disks. The unrooted tree $t_p$ associated to any intersection point $p\in W_{(I,J)}\cap W_K$ is the \emph{inner product}
$t_p= \langle \,(I,J),K\,\rangle$ gotten by identifying the roots of the trees $(I,J)$ and $K$ to a single non-vertex point.
Note that $t_p$ also can be considered as a subset of $\cW$, with the edge of $t_p$ containing $p$ a sheet-changing
path connecting the basepoints of $W_{(I,J)}$ and $W_K$ (see Figure~\ref{fig:W-disks-inner-prod-oriented}). 

If $X$ is not simply connected, then the edges of the just-described trees are decorated by elements of $\pi_1X$ as follows:
Considering the trees as subsets of $\cW$, each edge of a tree is a sheet-changing path connecting basepoints of
adjacent surfaces of $\cW$. Choosing orientations of these sheet-changing paths determines elements of $\pi_1X$ (using the whiskers on the surfaces) which are attached as labels on
the correspondingly oriented tree edges.

Note that the notation for trees is slightly different in the older papers \cite{S1,ST2}, where the rooted tree 
associated to a bracket $I$ is denoted $t(I)$, and the rooted and inner products are denoted by $*$ and $\cdot$ respectively. The notation of this paper agrees with the more recent papers \cite{CST0, CST1,CST2,CST3, CST4,S3}.

\begin{figure}[ht!]
         \centerline{\includegraphics[scale=.45]{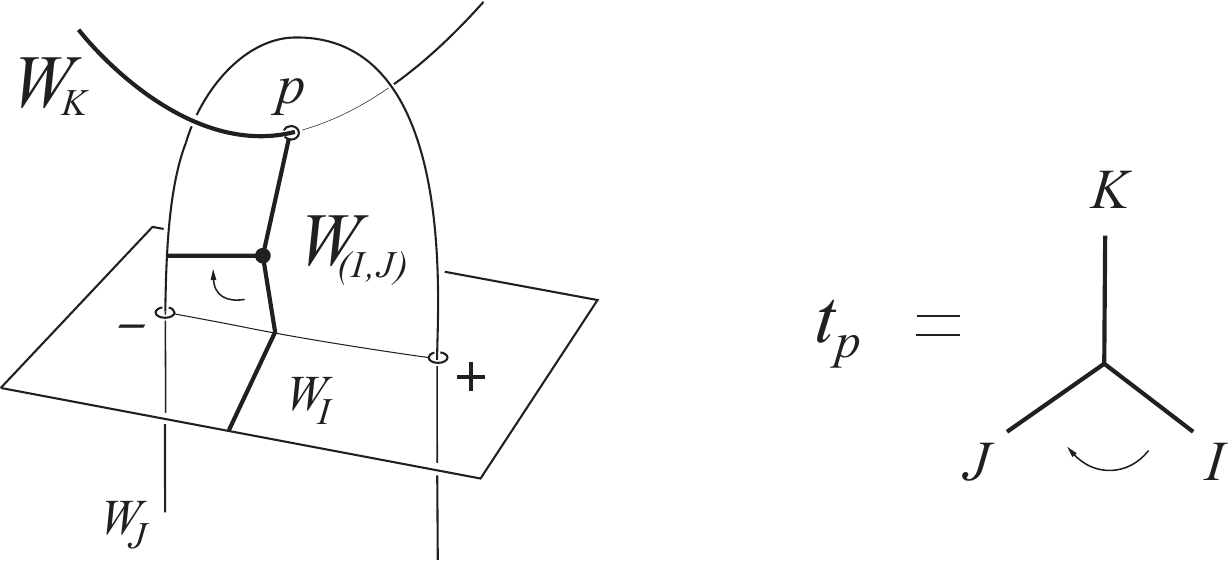}}
         \caption{A local picture of the tree $t_p=\langle (I,J),K\rangle$ associated to $p\in W_{(I,J)}\cap W_K$ near 
         a trivalent vertex adjacent to the edge of $t_p$ passing through an unpaired intersection point $p$
         in a Whitney tower $\cW$. On the left $t_p$ is pictured as a subset of $\cW$, and on the right as an abstract 
         labeled vertex-oriented tree. In a non-simply connected $4$--manifold $X$ the edges of $t_p$ would also be oriented and labeled by elements of $\pi_1X$ (as in Figure~\ref{fig:Relations-fig} below).}
         \label{fig:W-disks-inner-prod-oriented}
\end{figure}

\subsection{Orientation conventions}\label{sec:orientations}
Thinking of the tree $I$ associated to
a Whitney disk
$W_I$ as a subset of $\cW$, it can be arranged that the
trivalent orientations of $I$ are induced by the orientations
of the corresponding Whitney disks: Note that the pair of edges
which pass from a trivalent vertex down into the lower-order
surfaces paired by a Whitney disk determine a ``corner'' of the
Whitney disk which does not contain the other edge of the
trivalent vertex. If this corner contains the {\em negative}
intersection point paired by the Whitney disk, then the vertex
orientation and the Whitney disk orientation agree. Our figures
are drawn to satisfy this convention. 

This ``negative corner'' convention (also used in \cite{CST1,CST2}), which differs from the
positive corner convention used in \cite{CST,ST2}, turns out to be
compatible with the usual commutator conventions, for instance in the setting of Milnor invariants
(see Figure~\ref{fig:Wtower-to-grope1and2}).

\subsection{Non-repeating Whitney towers}\label{sec:non-repeating-w-tower-def}
Whitney disks and intersection points are called
\emph{non-repeating} if their associated trees are non-repeating. This means that the univalent vertices are labeled by {\em distinct} indices 
(corresponding to distinct order $0$ surfaces, i.e.~distinct connected components of $A$). 
A Whitney tower $\cW$ is an order $n$ \emph{non-repeating Whitney tower} if all non-repeating intersections of order (strictly) less than $n$ are paired by Whitney disks. In particular, if $\cW$ is an order $n$ Whitney tower then $\cW$ is also an order $n$ non-repeating Whitney tower. In a non-repeating Whitney tower repeating
intersections of any order are not required to be paired by Whitney disks.

\subsection{Intersection invariants}\label{sec:int-trees}
For a group $\pi$, denote by $\cT_n(m,\pi)$ the abelian group
generated by order $n$ (decorated) trees modulo the relations
illustrated in Figure~\ref{fig:Relations-fig}. 
\begin{figure}[ht!]
         \centerline{\includegraphics[scale=.75]{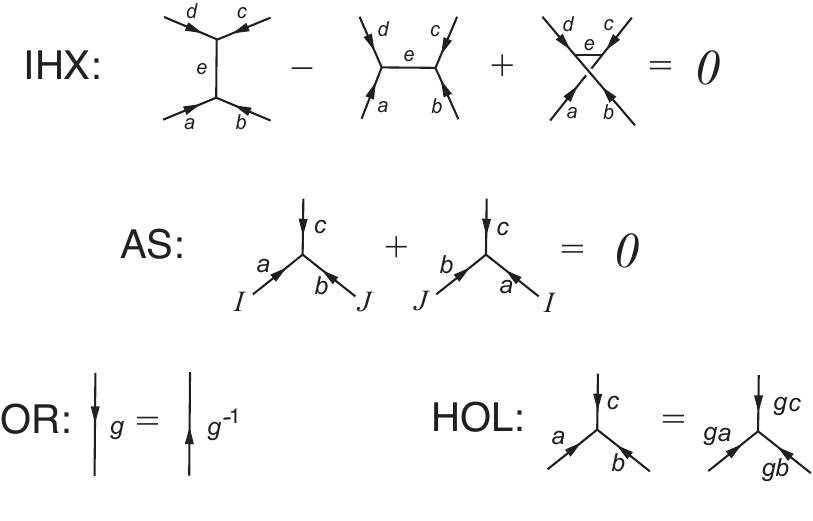}}
         \caption{The relations in $\cT_n(\pi,m)$: IHX (Jacobi), AS (antisymmetry), OR (orientation), HOL (holonomy). 
         These are `local' pictures, meaning that the unlabeled univalent vertices extend to fixed decorated subtrees in each equation.
         For instance, in the right-hand term of the HOL relation only the three visible edge decorations are multiplied by the element $g$, 
         corresponding to a change of whisker on a Whitney disk at the indicated trivalent vertex. 
         All vertex-orientations are induced from a fixed orientation of the plane; in particular, the two terms in the AS relation 
         only differ by the orientation at the indicated trivalent vertex, where the two edges 
         extending to the subtrees $I$ and $J$ have been interchanged.}
         \label{fig:Relations-fig}
\end{figure}

Note that when $\pi$ is the trivial group, the edge
decorations (orientations and $\pi$-labels) disappear, and the relations reduce to the usual AS
antisymmetry and IHX Jacobi relations of finite type theory (compare also the decorated graphs of \cite{GL}). All the relations are
homogeneous in the univalent labels, and restricting the
generating trees to be non-repeating order $n$ trees defines the
subgroup $\Lambda_n(m,\pi)<\cT_n(m,\pi)$.
(See sections 2.1 and 3 of \cite{ST2} for explanations of these relations.)

\begin{defn}\label{def:tau-lambda}
For an order $n$ (oriented) Whitney tower $\cW$ in $X$, the \emph{order n intersection
invariant} $\tau_n (\cW)$ is defined by summing the signed trees $\pm
t_p$ over all order $n$ intersections $p\in\cW$:
$$
\tau_n(\cW):=\sum \sign(p)\cdot t_p\in\cT_n(\pi).
$$
Here $\pi=\pi_1X$; and $\sign(p)=\pm$, for $p\in W_I\cap W_J$, is the usual
sign of an intersection between the oriented Whitney disks $W_I$ and
$W_J$.

If $\cW$ is an order $n$ non-repeating Whitney tower, the order $n$
\emph{non-repeating intersection invariant}
$\lambda_n(\cW)$ is analogously defined by
$$
\lambda_n(\cW):=\sum \sign(p)\cdot t_p\in\Lambda_n(\pi)
$$
where the sum is over all order $n$ non-repeating intersections
$p\in\cW$.
\end{defn}

\subsection{Order $0$ intersection
invariants}\label{sec:order-zero-int-trees} The order $0$
intersection invariants $\tau_0$ and $\lambda_0$ for $A:\amalg^m S^2 \to X$ carry the same information as Wall's \cite{W}
Hermitian intersection form $\mu,\lambda$: The generators in $\tau_0(A)\in\cT_0(\pi,m)$ with
both vertices labeled by the same index $i$ correspond to Wall's
self-intersection invariant $\mu(A_i)$. For $\mu(A_i)$ to be a \emph{homotopy}
(not just regular homotopy) invariant, one must also mod out by a
\emph{framing relation} which kills order $0$ trees labeled by the trivial
element in $\pi$ (see \cite{CST1} for higher-order framing
relations). Wall's homotopy invariant Hermitian intersection
pairing $\lambda(A_i,A_j)\in\Z[\pi]$ for $i\neq j$ corresponds
to $\lambda_0(A)\in\Lambda_0(\pi,m)$. 

The vanishing of these invariants corresponds to the order $0$ intersections coming in canceling pairs (after perhaps a homotopy of $A$), so $A$ admits an order 1 Whitney tower if and
only if $\tau_0(A)=0\in\cT_0(\pi,m)$, and admits
an order 1 non-repeating Whitney tower if and only if
$\lambda_0(A)=0\in\Lambda_0(\pi,m)$.

\subsection{Order $1$ intersection
invariants}\label{sec:order-one-int-trees} It was shown in
\cite{ST1}, and for $\pi_1X=1$ and $m=3$ in \cite{Ma,Y}, that for
$A:\amalg^m S^2 \to X$ admitting an order 1 Whitney tower 
(resp.~non-repeating Whitney tower) $\cW$, the order 1 intersection invariant
$\tau_1(A):=\tau_1(\cW)$ (resp.~order $1$
non-repeating intersection invariant
$\lambda_1(A):=\lambda_1(\cW)$) is a homotopy
invariant of $A$, if taken in an appropriate
quotient of $\cT_1(\pi,m)$ (resp.~$\Lambda_1(\pi,m)$). The relations defining this quotient 
are
determined by order $0$ intersections between the $A_i$ and
immersed $2$--spheres in $X$. These are the order 1
intersection relations $\textrm{INT}_1$ which are described in \cite{ST1} (in
slightly different notation) and below in
Section~\ref{sec:order2-INT} (for $\lambda_1$). As remarked in the introduction, for $\tau_1$ there are also framing
relations, but there are no framing relations for $\lambda_n$ (for
all $n$) because Whitney disks can always be framed
by the boundary-twisting operation \cite[Sec.1.3]{FQ} which creates only repeating
intersections.

From \cite{ST1}, we have that $A$ admits an order 2 Whitney tower (resp.~order 2 non-repeating Whitney tower) if and only if $\tau_1(A)$ 
(resp.~$\lambda_1(A)$) vanishes. In particular,
$\lambda_1(A_1,A_2,A_3)\in\Lambda_1(\pi,3)/\mathrm{INT}_1$ is the complete
obstruction to pulling apart three order $0$ surfaces with vanishing $\lambda_0(A_1,A_2,A_3)$.

\subsection{Order $n$ intersection
invariants}\label{sec:order-n-obstruction} As was shown in Theorem~2 of \cite{ST2}, for $A$
admitting a Whitney tower $\cW$ of order $n$, if
$\tau_n(\cW)=0\in\cT_n(\pi)$ then $A$ admits a Whitney tower of
order $n+1$. The proof of this result proceeds by geometrically realizing
the relations in the target group of the intersection invariant in a
controlled manner, so that one can convert ``algebraic cancellation'' of
pairs of trees to ``geometric cancellation'' of pairs of points (paired by
next-order Whitney disks).
The exact same arguments work
restricting to the non-repeating case to prove Theorem~\ref{thm:non-rep-obstruction-theory}
of the introduction: For $A$
admitting a non-repeating Whitney tower $\cW$ of order $n$, if
$\lambda_n(\cW)=0\in\Lambda_n(\pi)$ then $A$ admits a non-repeating Whitney tower of
order $n+1$. Beyond this ``sufficiency'' result, it is not known for $n\geq 2$ what additional relations
 $\mathrm{INT}_n\subset\Lambda_n(\pi)$ would also make the vanishing of $\lambda_n(\cW)$ in the quotient a necessary
 condition
 for $A$ to admit a non-repeating Whitney tower of
order $n+1$, as discussed in \ref{subsubsec:htpy-invariance} of the introduction.

\subsection{The groups $\Lambda_n$}\label{subsec:lambda-n-groups}
The groups $\Lambda_n(\pi,m)$ provide upper bounds for the order $n$ non-repeating obstruction theory, and hence by Corollary~\ref{cor:tower=disjoint} also for the obstructions to pulling apart surfaces.
The following result describes the structure of  $\Lambda_n(\pi,m)$:
\begin{lem}\label{lem:Lambda-computation}
$\Lambda_n(\pi,m)$ is isomorphic (as an additive abelian group) to the $\binom{m}{n+2}n!$-fold direct sum of the integral group ring $\Z[\pi^{n+1}]$ of the $(n+1)$-fold cartesian product $\pi^{n+1}=\pi\times\pi\times\cdots\times\pi$.
\end{lem}

\emph{Proof:}  
First consider the case where $\pi$ is trivial. Since the relations in $\Lambda_n(m)$ are all homogenous in the univalent labels, $\Lambda_n(m)$ is the direct sum of 
subgroups $\Lambda_n(n+2)$ over the $\binom{m}{n+2}$ choices of $n+2$ of the $m$ labels. 
(As noted in the introduction, $\Lambda_n(\pi,m)$ is trivial for $n\geq m-1$ since an order $n$ unitrivalent tree has $n+2$ univalent vertices.) We will show that each of these subgroups
has a basis given by the  $n!$ distinct
\emph{simple} non-repeating trees shown in Figure~\ref{fig:simple-tree-decorated} (ignoring the edge decorations for the moment), where an order $n$ tree is simple if it contains a geodesic of edge-length $n+1$.

For a given choice of $n+2$ labels, placing a root at the minimal-labeled vertex of each order $n$ tree gives an isomorphism from $\Lambda_n(n+2)$ to
the subgroup of non-repeating length $n+1$ brackets in the free Lie algebra (over $\Z$) on the other labels (with AS and IHX relations going to skew-symmetry relations and Jacobi identities).
This ``reduced'' free Lie algebra (see also \ref{sec:mu-invts} below) is known to have rank $n!$, as explicitly described in \cite[Thm.5.11]{MKS} (also implicitly contained in \cite[Sec.4--5]{M1}), so the
trees in Figure~\ref{fig:simple-tree-decorated} are linearly independent if they span.

To see that the
trees in Figure~\ref{fig:simple-tree-decorated} form a spanning set, first observe that for a given choice of $n+2$ labels, 
each order $n$ non-repeating tree
$t$ has a distinguished geodesic edge path $T_t$ from the minimal-label univalent vertex to the maximal-label univalent vertex.
For an orientation-inducing embedding of $t$ in the plane, it can be arranged that
all the sub-trees of $t$ emanating from $T_t$ lie on a preferred side of $T_t$ by applying AS relations at the trivalent vertices of $T_t$ as needed. Then, by repeatedly applying IHX relations (replacing the left-most I-tree by the difference of the H-tree and X-tree in the IHX relation of Figure~\ref{fig:Relations-fig}) at trivalent vertices of distinguished geodesics to reduce the order of the emanating sub-trees one eventually gets a linear combination of simple non-repeating trees as in Figure~\ref{fig:simple-tree-decorated} which is uniquely determined by $t$. (To see how the IHX relation reduces the order of subtrees emanating from a distinguished geodesic, 
observe that if the central edge of the I-tree in an IHX relation is the first edge of such a subtree, then the corresponding emanating subtrees in the H-tree and X-tree both have order decreased by one.)

In the case of non-trivial $\pi$, the group elements decorating the edges of the simple trees can always be (uniquely) normalized
to the trivial element on all but $n+1$ of the edges as shown in Figure~\ref{fig:simple-tree-decorated} (by applying HOL relations from Figure~\ref{fig:Relations-fig} and working from the minimal towards the maximal vertex label).
$\hfill\square$

\begin{figure}[ht!]
         \centerline{\includegraphics[scale=.35]{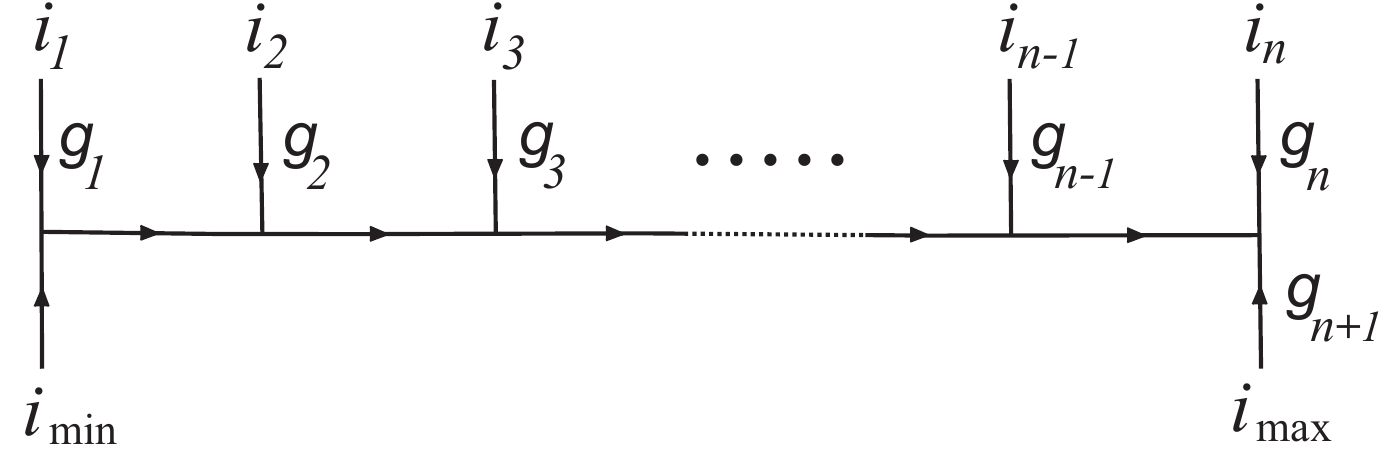}}
         \caption{A \emph{simple} order $n$ tree with minimal- and maximal-labeled vertices 
         connected by a length $n+1$ geodesic; 
         $1\leq i_{\textrm{min}}<i_{\textrm{max}}\leq m$, and $i_{\textrm{min}}<i_k<i_{\textrm{max}}$ for $1\leq k\leq n$. 
         Vertex orientations are induced by the planar embedding. By the HOL relations, all but $n+1$ of 
         the edge decorations can be set to the trivial element in $\pi$ (indicated by the `empty-labeled' edges in the figure).}
         \label{fig:simple-tree-decorated}
\end{figure}

\subsection{Some properties of Whitney towers}\label{sec:properties-of-w-towers}
For future reference, we note here some elementary properties of Whitney towers and their intersection invariants.

Let $A:\Sigma\to X$ support an order $n$ Whitney tower $\cW\subset X$, where $\Sigma$ has $m$ connected components $\Sigma_i$. We will consider the effects on $\tau_n(\cW)$ of changing the order $0$ surfaces $A_i:\Sigma_i\to
 X$ of $A$ by the operations of re-indexing, including parallel copies, taking internal sums, switching orientations, and deletions; all of which preserve the property that $A$ supports an order $n$ Whitney tower. We will focus on the case where $X$ is simply connected, which will be used in Section~\ref{sec:proof-thm-in-X-L}. (Analogous properties hold in the non-simply connected setting,
although when taking internal sums (\ref{subsec:connected-sums}) some care would be needed in keeping track of the effect on the edge decorations due to choices of arcs guiding the sums.)

\subsubsection{Re-indexing order $0$ surfaces}\label{subsec:re-indexing-order-zero-surfaces}
For $A:\Sigma\to X$ the natural indexing of the order $0$ surfaces of $\cW$ is by $\pi_0\Sigma$. In practice, we fix an identification
of $\pi_0\Sigma$ with the label set $\{1,2,\ldots,m\}$, and the effect of changing this identification is given by the corresponding permutation
of the univalent labels on all the trees representing $\tau_n(\cW)$. 

\subsubsection{Parallel Whitney towers}\label{subsec:parallel-w-towers}
Suppose $A$ is extended to $A'$ by including a parallel copy $A_{m+1}$ of the last order $0$ surface $A_m$ of $A$. 
Recall from Definition~\ref{w-tower-defn} that order $0$ surfaces have trivial (relative) normal Euler numbers, so each self-intersection of $A_m$ will give rise to a single self-intersection of $A_{m+1}$ and a pair of intersections between $A_{m+1}$ and $A_m$; and each intersection between $A_m$ and any $A_i$, for $i\neq m$, will give rise to a single intersection between $A_{m+1}$ and $A_i$; and no other intersections in $A'$ will be created. By the splitting procedure of \cite[Lem.13]{ST2}
(also \cite[Lem.3.5]{S1}) it can be arranged that all Whitney disks in $\cW$ are embedded and contained in standard $4$--ball thickenings of their trees. Since the Whitney disks are all framed, $\cW$ can be extended to an order $n$ Whitney tower $\cW'$ on $A'$ by including parallel copies of the Whitney disks in
$\cW$ as illustrated by Figure~\ref{fig:parallel-w-disks}.  This new Whitney tower $\cW'$ can be constructed in an arbitrarily small neighborhood of $\cW$, and the intersection invariants are related in the following way. 

\begin{figure}[ht!]
         \centerline{\includegraphics[scale=.275]{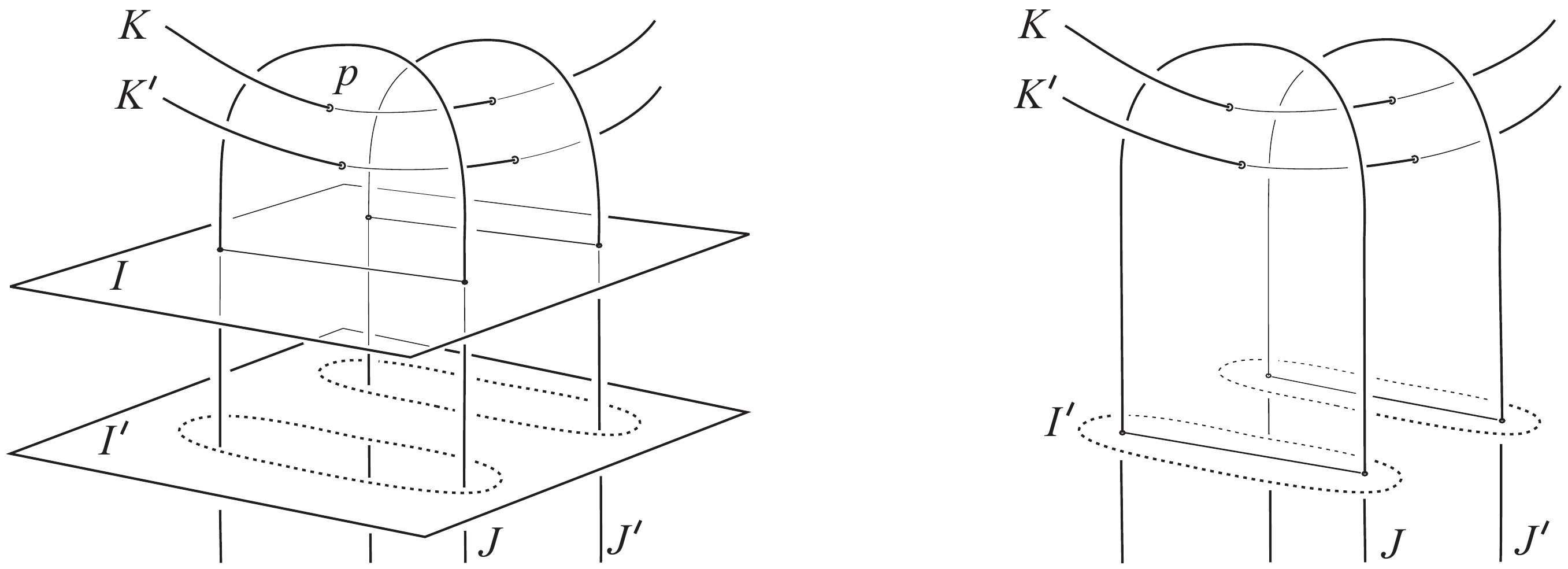}}
         \caption{Extending a Whitney tower using parallel Whitney disks: An unpaired intersection point 
         $p\in W_{(I,J)}\cap W_K$ in a Whitney tower on $A$, where each of $I$, $J$ 
         and $K$ contain exactly one occurrence of the index $m$, 
         gives rise to eight unpaired intersections in four Whitney disks in the 
         new Whitney tower on $A'$, formed from $A$ by including a parallel copy of $A_m$.
         The dotted oval loops (left) bound neighborhoods of the Whitney arcs in the opaque $I'$ sheet which have been perturbed into a nearby time coordinate along with the two corresponding translucent Whitney disks (right).}
         \label{fig:parallel-w-disks}
\end{figure}

Define $\delta\colon\cT_n(m)\to\cT_n(m+1)$ to 
be the homomorphism induced by the map which sends a generator $t$ having $r$-many $m$-labeled univalent vertices to the $2^{r}$-term sum over all choices of replacing the label $m$ by the label $(m+1)$.
Then $\tau_n(\cW')=\delta(\tau_n(\cW))$. (In the non-simply connected setting, group elements decorating the edges would be preserved by taking parallel whiskers.)

Via re-indexing, the effect of including a parallel copy of any $i$th order $0$ surface can be described by analogous relabeling maps $\delta_i$, and 
iterating this procedure constructs an order $n$ Whitney tower near $\cW$ on any number
of parallel copies of any order $0$ surfaces of $A$, with the resulting change in $\tau_n(\cW)$
described by compositions of the $\delta_i$ maps. 

\subsubsection{Internal sums}\label{subsec:connected-sums}
Suppose $A'$ is formed from $A$ by taking the ambient connected sum of $A_{m-1}$ with $A_m$ in $X$ 
(or by joining $\partial A_{m-1}$ to $\partial A_m$ with a band in $\partial X$), 
so that $A'$ has $m-1$ components.
Since it may be assumed that the interior of the arc guiding the sum is disjoint from $\cW$, 
it is clear that $A'$ bounds an order $n$ Whitney tower $\cW'$ all of whose Whitney disks and singularities are identical to $\cW$. 
Then $\tau_n(\cW')=\sigma(\tau_n(\cW))\in\cT_n(m-1)$, where the map 
$\sigma:\cT_n(m)\to\cT_n(m-1)$ is induced by the relabeling map on generators which changes all $m$-labeled 
univalent vertices to $(m-1)$-labeled univalent vertices. (In the non-simply connected setting, group elements decorating all edges would be preserved if the guiding arc together with the whiskers on $A_{m-1}$ and $A_m$ formed a null-homotopic loop.)

Via re-indexing, the effect of summing any $A_i$ with any $A_j$ ($j\neq i$) is described
by the analogous map $\sigma_{ij}$,
and for iterated internal sums the resulting intersection invariant
is described by compositions of the $\sigma_{ij}$ maps.

\subsubsection{Switching order $0$ surface orientations}\label{subsec:reorient-order-zero}
As explained in \cite[Sec.3]{ST2}, the orientation of $A$ determines the vertex-orientations of the
trees representing $\tau_n(\cW)$ up to AS relations, via our above convention (\ref{sec:orientations}). The effect 
on $\tau_n(\cW)$ of switching the orientation
of an order $0$ surface $A_i$ of $A$ is described as follows. 

Define $s_i\colon\cT_n(m)\to\cT_n(m)$ to 
be the automorphism induced by the map which sends a generator $t$
to $(-1)^{i(t)}t$, where again $i(t)$ denotes the multiplicity of the univalent label $i$ in $t$.
Then if $\cW'$ is a reorientation of $\cW$ which is compatible with a reversal of orientation of
$A_i$, then we have $\tau_n(\cW')=s_i(\tau_n(\cW))$. 

The effect on the intersection invariant of reorienting any number of order $0$ surfaces of $A$
is described by compositions of the $s_i$ maps.

\subsubsection{Deleting order $0$ surfaces}\label{subsec:deleting-order-zero}
The result $A'$ of deleting the last order $0$ surface $A_m$ of $A$ supports an order $n$ Whitney tower $\cW'$
formed by deleting those Whitney disks from $\cW$ which involve $A_m$; that is, deleting any Whitney disk whose tree has at least one univalent vertex labeled by $m$. We have $\tau_n(\cW')=e(\tau_n(\cW))$, where the homomorphism
$e\colon\cT_n(m)\to\cT_n(m-1)$ is induced by the map which sends a generator $t$
to zero if $m$ appears as a label in $t$, and is the identity otherwise. Via re-indexing, the effect of deleting any $A_i$ can be described by analogous maps $e_i$, and the change in $\tau_n(\cW)$ due to multiple deletions of order $0$ surfaces is
described by compositions of the $e_i$. 

\subsubsection{Canceling parallels}\label{subsubsec:canceling-parallel-lemma}
We note here the following easily-checked lemma, which will be used in Section~\ref{sec:proof-thm-in-X-L}:
\begin{lem}\label{lem:canceling-parallels}
The composition $\sigma_{ji'}\circ\sigma_{i'i''}\circ s_{i''}\circ\delta_{i'}\circ\delta_{i}$ is the identity map
on $\cT_n(m)$.\hfill$\square$
\end{lem}
Lemma~\ref{lem:canceling-parallels} describes the effect on the intersection invariant that corresponds to including two parallel copies $A'_i$ and $A''_i$ of $A_i$, switching the orientation on $A''_i$, then recombining $A'_i$ and $A''_i$ by an internal sum into a single $i'$th component, and then internal summing this combined $i'$th component into any $j$th component of $A$.
(Note that applying the analogous sequence of operations to a link obviously preserves the isotopy class of the link.)

\section{Proof of Theorem~\ref{thm:tower-separates}}\label{sec:thm-tower-separates-proof}
We want to show that $m$ connected surfaces $A_i:\Sigma_i\to
 X$ can be pulled apart if and only if they admit an order $m-1$ non-repeating Whitney tower.
 
\emph{Proof:}  
The ``only if'' direction follows by definition, since disjoint order $0$ surfaces form a non-repeating Whitney tower of any order.
So let $\cW$ be a non-repeating Whitney tower of order $m-1$ on $A_1,
A_2,\ldots,A_m$. If $\cW$ contains no Whitney disks, then the
$A_i$ are pairwise disjoint. In case $\cW$ does contain Whitney
disks, we will describe how to use finger moves and Whitney moves
to eliminate the Whitney disks of $\cW$ while preserving the
non-repeating order $m-1$. 

First note that $\cW$ contains no unpaired non-repeating intersections: All non-repeating intersections of order $<m-1$ are paired 
by definition; and since trees of order $\geq m-1$ have $\geq m+1$ univalent vertices, all intersections of
order greater than or equal to $m-1$ in any Whitney tower on $m$ order $0$ surfaces 
must be repeating intersections.

Now consider a Whitney disk $W_{(I,J)}$ in $\cW$ of maximal order. If $W_{(I,J)}$ is \emph{clean}
(the interior of $W_{(I,J)}$ contains no singularities) then do the
$W_{(I,J)}$-Whitney move on either $W_I$ or $W_J$. This eliminates
$W_{(I,J)}$ (and the corresponding canceling pair of intersections between
$W_I$ and $W_J$) while creating no new intersections, hence preserves the order of the resulting non-repeating Whitney tower
which we continue to denote by $\cW$.

If any maximal order Whitney disk $W_{(I,J)}$ in $\cW$ is not
clean, then the singularities in the interior of $W_{(I,J)}$ are
exactly a finite number of unpaired intersection points, all of
which are repeating. (Since $W_{(I,J)}$ is of maximal order, the
interior of $W_{(I,J)}$ contains no Whitney arcs; and $\cW$ contains no unpaired non-repeating intersections, as noted above.)
So, for any $p\in W_{(I,J)}\cap W_K$, at least one of $(I,K)$ or
$(J,K)$ is a repeating bracket. Assuming that $(I,K)$, say, is
repeating, push $p$ off of $W_{(I,J)}$ down into $W_I$ by a finger
move (Figure~\ref{fig:PushDownIJK}). This creates only a pair of
repeating intersections between $W_I$ and $W_K$. After pushing
down all intersections in the interior of $W_{(I,J)}$ by
finger moves in this way, do the clean $W_{(I,J)}$-Whitney move on
either $W_I$ or $W_J$. Repeating this procedure on all maximal
order Whitney disks eventually yields the desired order $m-1$
non-repeating Whitney tower (with no Whitney disks) on order $0$ surfaces
$A_i'$. The $A_i'$ are regularly homotopic to the $A_i$; the pushing-down finger moves
will have created pairs of self-intersections in the pairwise disjointly immersed $A_i'$.
$\hfill\square$
\begin{figure}[ht!]
         \centerline{\includegraphics[scale=.45]{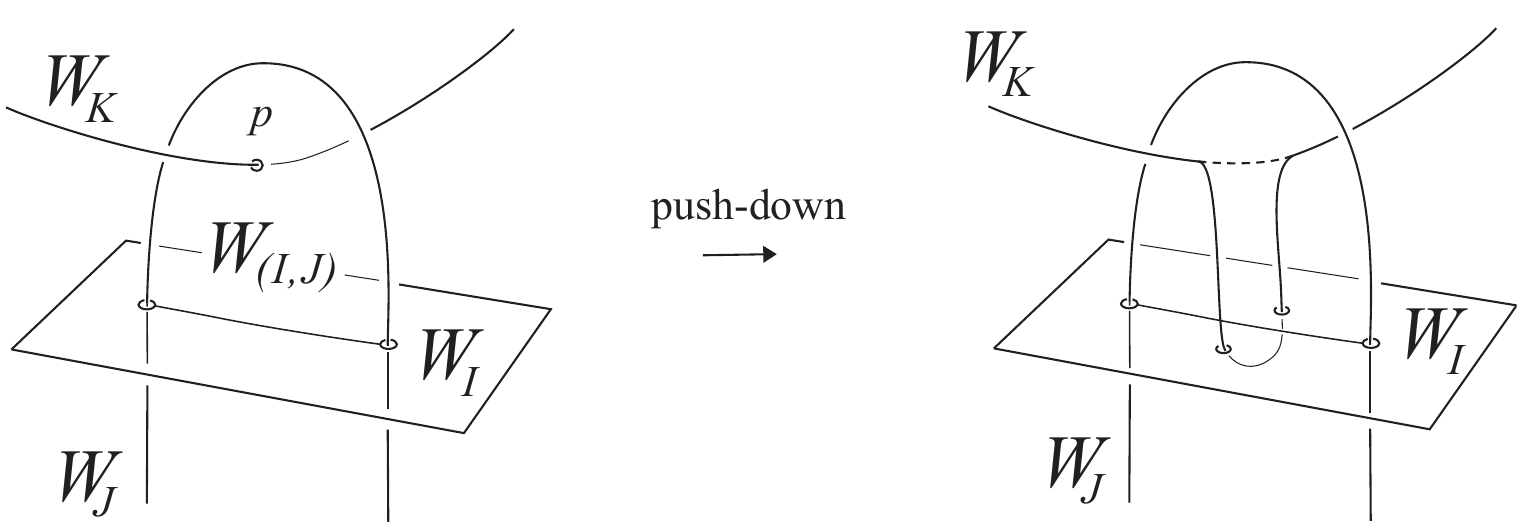}}
         \caption{}
         \label{fig:PushDownIJK}
\end{figure}

\section{Proof of Theorem~\ref{thm:lambda-link-htpy}}\label{sec:link-htpy-thm-proof}
Consider a link $L=L_1\cup L_2\cup\cdots\cup L_m\subset S^3$ that bounds an order $n$ non-repeating Whitney tower $\cW$ on immersed disks in the $4$--ball. We will prove Theorem~\ref{thm:lambda-link-htpy} by relating $\lambda_n(\cW)$ to Milnor's length $n+2$ link-homotopy $\mu$-invariants of $L$ in Theorem~\ref{thm:mu}, showing in particular that $\lambda_n(L):=\lambda_n(\cW)\in\Lambda_n(m)$ only depends on the link-homotopy class of $L$ (and not on the Whitney tower $\cW$). 

The essential idea is that $\cW$ can be used to compute the link longitudes as iterated commutators in
Milnor's nilpotent quotients of the fundamental group of the link
complement.
The proof uses a new
result, \emph{Whitney tower-grope duality}, which describes certain class $n+2$ gropes
that live in the complement of an order $n$ Whitney tower in any $4$--manifold (Proposition~\ref{prop:w-tower-duality}).
After fixing notation for the first-non-vanishing Milnor invariants of $L$ in section~\ref{sec:mu-invts}, we give the explicit
identification of them with $\lambda_n(\cW)$ in Theorem~\ref{thm:mu} of section~\ref{subsec:trees-to-brackets-main-Milnor-thm}. 

\subsection{Milnor's link-homotopy $\mu$-invariants}\label{sec:mu-invts}
This subsection briefly reviews and fixes notation for the first non-vanishing non-repeating $\mu$-invariants of a link.
See any of
\cite{B2,HL,HM,M1} for details.  

For a group $G$ {\em normally} generated by elements $g_1,g_2,\ldots,g_m$, the
\emph{Milnor group} of $G$ (with respect to the $g_i$) is
the quotient of $G$ by the subgroup normally generated by all commutators between $g_i$ and $g_i^h:=hg_ih^{-1}$, so we kill the elements
$$
[g_i,g_i^h]=g_ig_i^hg_i^{-1}g_i^{-h}
$$
for $1\leq i\leq m$, and all $h\in G$.
One can prove (e.g.~\cite[Lem.1.3]{HL}) by induction on $m$ that this quotient is nilpotent and (therefore) generated by $g_1,\dots,g_m$. 

The \emph{Milnor group} $\cM (L)$ of an $m$-component link $L$ is
the Milnor group of the fundamental group of the link complement $\pi_1(S^3\smallsetminus L)$ with respect to a
generating set of meridional elements. Specifically, $\cM(L)$ has
a presentation
$$
\cM(L)=\langle x_1,x_2,\ldots,x_m\,|\,[\ell_i,x_i], \,
[x_j,x_j^h]\rangle
$$
where each $x_i$ is represented by a meridian (one for each component), 
and the $\ell_i$ are words in the 
$x_i$ determined by the link longitudes. The Milnor group
$\cM(L)$ is the largest common quotient of the fundamental groups
of all links which are link homotopic to $L$. Since $\cM(L)$ only depends on the
conjugacy classes of the meridional generators $x_i$, it only depends on the link $L$ (and no base-points are necessary). 

A presentation for the Milnor group of the unlink (or any
link-homotopically trivial link) corresponds to the case where all
$\ell_i=1$, and Milnor's $\mu$-invariants (with non-repeating
indices) compare $\cM(L)$ with this \emph{free Milnor group}
$\cM(m)$ by examining each longitudinal element in terms of the
generators corresponding to the \emph{other} components.
Specifically, mapping $x_i^{\pm 1}$ to $\pm X_i$ induces a canonical isomorphism 
\[
\cM(m)_{(n)}/\cM(m)_{(n+1)} \cong \sRL_n(m)
\] 
from the lower central series quotients
to the \emph{reduced} free Lie algebra $\sRL(m)=\oplus_{n=1}^m \sRL_n(m)$, which is the quotient of
the free $\Z$-Lie algebra on the $X_i$ by the relations which set an iterated
Lie bracket equal to zero if it contains more than one occurrence of a generator. This isomorphism takes 
a product of length $n$ commutators in distinct $x_i$ to a sum of length $n$ Lie brackets
in distinct $X_i$. In particular, $\sRL_n(m)=0$ for $n>m$. 

Let $\cM^i(L)$ denote the quotient of $\cM(L)$ by the relation $x_i=1$. If
the element in $\cM^i(L)$ determined by the longitude $\ell_i$
lies in the $(n+1)$th lower central subgroup $\cM^i(L)_{(n+1)}$ for
each $i$, then we have isomorphisms:
$$
\cM(L)_{(n+1)}/\cM(L)_{(n+2)}\cong\cM(m)_{(n+1)}/\cM(m)_{(n+2)}\cong
\sRL_{(n+1)}(m).
$$
Via the usual identification of non-associative bracketings and binary trees, 
$\sRL_{(n+1)}(m)$ can be identified with the abelian group on order $n$ \emph{rooted} non-repeating trees
modulo IHX and antisymmetry relations as in Figure~\ref{fig:Relations-fig} (with $\pi$ trivial). 
This identification explains the subscripts in the following definition:
\begin{defn}\label{def:mu(L)}
The elements $\mu^i_n(L)\in \sRL^i_{(n+1)}(m)$
determined by the longitudes $\ell_i$ are the  \emph{non-repeating Milnor-invariants} of \emph{order} $n$. Here
$\sRL^i(m)$ is the reduced free Lie algebra on the $m-1$ generators $X_j$, for $j\neq i$.
\end{defn}

This definition of non-repeating $\mu$-invariants was originally given by Milnor \cite{M1}. He later expressed the elements $\mu^i_n(L)$ in terms of integers $\mu_L(i, k_1, \dots, k_{n+1})$, which are the coefficients of $X_{k_1}\cdots X_{k_{n+1}}$ in the Magnus expansion of $\ell_i$.
We note that our \emph{order} $n$ corresponds to the originally used \emph{length} $n+2$ (of entries in $\mu_L$). 

By construction, these non-repeating $\mu$-invariants depend only on the link-homotopy class of the link $L$. We have only defined order $n$ $\mu$-invariants assuming that the lower-order $\mu$-invariants vanish, which will turn out to be guaranteed by the existence of an order $n$ non-repeating Whitney tower.

\subsection{Mapping from trees to Lie brackets}\label{subsec:trees-to-brackets-main-Milnor-thm}
For each $i$, define a map 
\[
\eta_n^i :\Lambda_n(m)\to \sRL^i_{(n+1)}(m)
\]
by sending a tree $t$ which has an
$i$-labeled univalent vertex $v_i$ to the iterated bracketing determined
by $t$ with a root at $v_i$. Trees without an $i$-labeled vertex
are sent to zero. For example, if $t$ is an order $1$ $Y$-tree
with univalent labels $1,2,3$, and cyclic vertex orientation
$(1,2,3)$, then $\eta_1^1(t)=[X_2,X_3]$, and
$\eta_1^3(t)=[X_1,X_2]$, and $\eta_1^2(t)=[X_3,X_1]$. Note that the 
IHX and AS relations in $\Lambda_n(m)$
go to the Jacobi and skew-symmetry relations in
$\sRL^i_{(n+1)}(m)$, so the maps $\eta_n^i$ are well-defined.

\begin{lem}\label{lem:RL}
\[
\sum_{i=1}^m \eta_n^i :\Lambda_n(m)\longrightarrow \oplus_{i=1}^m \sRL^i_{(n+1)}(m)
\]
is a monomorphism. 
\end{lem}

\emph{Proof:}  
Putting an $i$-label in place of the root in a tree corresponding to a Lie bracket in
$\sRL^i_{(n+1)}(m)$ gives a left inverse to $\eta_n^i$. In fact, for the top degree $n+2=m$, this is an inverse because every index $i$ appears exactly once in a tree $t$ of order $n=m-2$. 
For arbitrary $n$, it is easy to check that composing the sum of these left inverse maps with $\sum_{i=1}^m \eta_n^i $ is just multiplication by $n+2$ on $\Lambda_n(m)$.
Since $\Lambda_n(m)$ is torsion-free by Lemma~\ref{lem:Lambda-computation}, it follows that $\sum_{i=1}^m \eta_n^i$ is injective.
$\hfill\square$

\begin{rem}\label{rem:repeating}
The monomorphism $\sum_{i=1}^m \eta_n^i$ fits into the bottom row of a commutative diagram:
\[
\xymatrix{
\cT_n(m)\ar@{>->}[r]^{\eta_n} \ar@{->>}[d] & \ar@{->>}[d]  \oplus_{i=1}^m \sL_{(n+1)}(m)\\
 \Lambda_n(m) \ar@{>->}[r] & \oplus_{i=1}^m \sRL^i_{(n+1)}(m) 
}
\]

Here the upper row is relevant for {\em repeating} Milnor invariants as explained in \cite{CST0,CST1}. The injectivity of the top horizontal map $\eta_n$, defined by Jerry Levine, is much harder to show and is the central result of \cite{CST3} (implying that $\cT_n(m)$ has at most 2-torsion). The two vertical projections simply set trees with repeating labels to zero. 
\end{rem}

The maps $\eta_n^i$ correspond to tree-preserving geometric constructions which
desingularize an order $n$ Whitney tower to a collection of
class $n+1$ \emph{gropes}, as described in detail in
\cite{S1}, and sketched in section~\ref{subsec:w-tower-to-grope-review} below. Gropes are 2-complexes built by gluing together compact
orientable surfaces, and this correspondence will be used in
the proof of the following theorem:

\begin{thm}\label{thm:mu} If a link $L\subset S^3$ bounds a
non-repeating Whitney tower $\cW$ of order $n$ on immersed disks
$D=\amalg^m D^2\to
 B^4$, then for each $i$ the longitude $\ell_i$
lies in $\cM^i(L)_{(n+1)}$, and
\[
\eta_n^i(\lambda_n(\cW))=\mu^i_n(L) \ \in \sRL^i_{(n+1)}(m)
\]
\end{thm}
Since the sum of the $\eta_n^i$ is injective, this will prove Theorem~\ref{thm:lambda-link-htpy}: The intersection invariant $\lambda_n(\cW) \in \Lambda_n(m)$ does not depend on the Whitney tower $\cW$ and is a link homotopy invariant of $L$, denoted by $\lambda_n(L)$. 

For $L$ bounding an honest order $n$ Whitney tower, one can deduce this theorem from the main result in \cite[Thm.5]{CST2} (and the diagram in Remark~\ref{rem:repeating} above); but here we only have a non-repeating order $n$ Whitney tower as an input.

\emph{Proof:}  
We start by giving an outline of the argument, introducing some notation that will be clarified during the proof: 
\begin{enumerate}
\item First the Whitney tower will be cleaned up, including the elimination of all repeating intersections of positive order and all repeating Whitney disks,
to arrive at an order $n$ non-repeating Whitney tower $\cW$ bounded by $L$ such that all unpaired intersection points of positive order have non-repeating trees (so the only repeating intersections are self-intersections in the order $0$ disks $D_j$). 

\item Then the preferred order $0$ disk $D_i$ (and all Whitney disks involving $D_i$) will be resolved to a grope $G_i$ of class $n+1$ bounded by $L_i$, such that
$G_i$ is in the complement $B^4\smallsetminus \cW^i$, where $\cW^i$ is the result of deleting $D_i$ and the Whitney disks used to construct $G_i$ from $\cW$.  
The grope $G_i$ will display the longitude $\ell_i$ in $\pi_1(B^4\smallsetminus \cW^i)$ as a product of $(n+1)$-fold commutators of meridians to the order $0$ surfaces $D^i:=\cup_{j\neq i}D_j$ of $\cW^i$ corresponding to 
putting roots at all $i$-labeled vertices of the trees representing $\lambda_n(\cW)$. This is the same formula as in the definition of the map $\eta_n^i$,
so it only remains to show that $\mu_n^i(L)$ can be computed in $\pi_1(B^4\smallsetminus \cW^i)$.

\item This last step is accomplished by using \emph{Whitney tower-grope duality} (Proposition~\ref{prop:w-tower-duality}) and Dwyer's theorem \cite{Dw} to show that 
the inclusion
$S^3\smallsetminus \partial D^i\to B^4\smallsetminus \cW^i$ induces an isomorphism
on the Milnor groups modulo the $(n+2)$th terms of the lower central series.
\end{enumerate}

Step (i):
Let $\cW$ be an order $n$ non-repeating Whitney tower on $D\to B^4$ bounded by $L\subset S^3$. As described in
\cite[Lem.3.5]{S1} (or \cite[Lem.13]{ST2}), $\cW$ can
be \emph{split}, so that each Whitney disk of $\cW$ is embedded, and the interior of each Whitney disk contains either a single unpaired intersection 
$p$ or a single
boundary arc of a higher-order Whitney disk, and no other
singularities.
This splitting process does not change the trees representing $\lambda_n(\cW)$, and results in each tree $t_p$ associated to
an order $n$ intersection $p$ being contained in a $4$-ball thickening of $t_p$, with all these $4$--balls pairwise disjoint.
Splitting simplifies combinatorics, and facilitates the use of local coordinates for describing constructions.
Also, split Whitney towers correspond to \emph{dyadic} gropes (whose upper stages are all genus one), and dyadic gropes are parametrized by trivalent (rooted) trees.
 
We continue to denote the split order $n$ non-repeating Whitney tower by $\cW$, and will
keep this notation despite future modifications. In the following, further splitting may be performed without mention. 
 
If $\cW$ contains any repeating intersections of positive order, then by following the pushing-down procedure described in the proof of 
Theorem~\ref{thm:tower-separates} given in section~\ref{sec:thm-tower-separates-proof}, all these repeating intersections
can be pushed-down until they create (many) pairs of self-intersections in the order $0$ disks. Then all repeating Whitney disks are clean,
and by doing Whitney moves guided by these clean Whitney disks it can be arranged that $\cW$ contains no repeating Whitney disks and no repeating intersections of positive order. 
 
Step (ii):
Consider now the component $L_i$ bounding $D_i$. We want to  
convert $D_i$ into a class $n+1$ grope displaying the longitude $\ell_i$ as a product of $(n+1)$-fold iterated commutators in meridians 
to the $D_{j\neq i}$ using the tree-preserving Whitney tower-to-grope construction of \cite[Thm.5]{S1}. 
This construction is sketched roughly below in section~\ref{subsec:w-tower-to-grope-review}, and a simple case is illustrated in Figure~\ref{bing-hopf-example-tower-and-grope-with-tree}.
Actually, the resulting grope $G_i$ comes with \emph{caps}, which in this setting are embedded normal disks to the other $D_j$ which are bounded by essential circles called \emph{tips}
on $G_i$. For our purposes the caps only serve to show that these tips are meridians to
the $D_j$. The trees associated to gropes are rooted trees, with the root vertex corresponding to the bottom stage surface, and the other univalent vertices corresponding to the tips (or to the caps). Since $\cW$ was split, the upper surface stages of $G_i$ will all be genus one,
so the collection of order $n$ unitrivalent trees $t(G_i)$ associated to $G_i$ will contain one tree for each dyadic branch of upper stages, with each trivalent vertex of a tree corresponding to a genus one surface in a branch. In this setting the class of $G_i$ is equal to $n+1$,
the number of non-root univalent vertices in each tree in the collection $t(G_i)$ (see e.g.~\cite[Sec.2.3]{S1}).

Applying the construction of \cite[Thm.5]{S1} to $D_i$ converts $D_i$ and all the Whitney disks of $\cW$ corresponding to trivalent vertices in trees containing an $i$-label into a class $n+1$ grope $G_i$. This grope $G_i$ (without the extra caps provided by \cite[Thm.5]{S1}) 
is disjoint from $\cW^i\subset \cW$, where the order $n$ 
non-repeating Whitney tower $\cW^i$ consists of the order $0$ immersed disks $D^i:=\cup_{j\neq i}D_j$ together with the Whitney disks 
of $\cW$ whose trees do not have an $i$-labeled vertex. In the present setting, any self-intersections of $D_i$ will give rise to self-intersections in the bottom stage surface of $G_i$ (which is bounded by $L_i$), but all higher stages of $G_i$ will be embedded.

At the level of trees, this construction of $G_i$ corresponds to replacing each $i$-labeled vertex on a tree representing
$\lambda_n(\cW)$ with a root \cite[Thm.5(v)]{S1}, which is the same formula as the map on generators defining $\eta_n^i$
(signs and orientations are checked in \cite[Lem.31]{CST2} and \cite[Sec.4.2]{CST2} in the setting of repeating Milnor invariants; see also sketch in section~\ref{subsec:w-tower-to-grope-review} below). 
So we have shown that, as an element in $\pi_1(B^4\smallsetminus \cW^i)$, the $i$th longitude $\ell_i$ is represented by 
the iterated commutators in meridians to the $D_j$ that correspond to the image of $\eta_n^i(\lambda_n(\cW))$ if 
the inclusion
$S^3\smallsetminus \partial D^i\to B^4\smallsetminus \cW^i$ induces an isomorphism
on the quotients of the Milnor groups by the $(n+2)$th terms of the lower central series.

Step (iii): To finish the proof of Theorem~\ref{thm:mu} we will use Dwyer's Theorem \cite{Dw} and a new notion of Whitney tower-grope duality to check that the inclusion
$S^3\smallsetminus \partial D^i\to B^4\smallsetminus \cW^i$ does indeed 
induce the desired isomorphism on the quotients of the Milnor groups by the $(n+2)$th terms of the lower central series.
It is easy to check that the inclusion induces an isomorphism on first homology, 
so by \cite{Dw} the kernel of the induced map on $\pi_1$
is generated by the attaching maps of the $2$-cells of surfaces generating the (integral) second homology
group $H_2(B^4\smallsetminus \cW^i)$. The order $0$ self-intersections $D_j\cap D_j$ only contribute Milnor relations, coming from the attaching maps of the $2$-cells of the Clifford tori around the self-intersections.
If the $D_j$ were pairwise disjoint, then by introducing (more) self-intersections as needed (by finger moves realizing the Milnor relations, see e.g.~\cite[XII.2]{Kirby}), it could be arranged that
$\pi_1(B^4\smallsetminus D^i)$ was in fact isomorphic to the free Milnor group.
Since the $D_j$ will generally intersect each other, we have to use the fact that $\cW^i$ is a non-repeating Whitney tower of order $n$ to show that any new relations coming from
(higher-order) intersections are trivial modulo $(n+2)$-fold commutators.
Since $H_2(B^4\smallsetminus \cW^i)$ is Alexander dual
to $H_1(\cW^i,\partial D^i)$, the proof of Theorem~\ref{thm:mu}
is completed by applying the following
general duality result to $\cW^i\subset B^4$, which shows that the other generating surfaces extend to class $n+2$ gropes. 
$\hfill\square$

\begin{prop}[Whitney tower-grope duality]\label{prop:w-tower-duality}
If $\cV$ is a split Whitney tower on 
$A:\Sigma=\amalg_j\Sigma_j^2 \to X$, where each order $0$ surface $A_j$ is a sphere $S^2\to X$ or a disk $(D^2,\partial D^2)\to (X,\partial X)$, then there exist dyadic
gropes $G_k\subset X\smallsetminus \cV$ such that the $G_k$ are geometrically dual to a generating set for the relative 
first homology group $H_1(\cV,\partial A)$. Furthermore, 
the tree $t(G_k)$ associated to each $G_k$ is obtained by attaching a rooted edge to the interior of an edge
of a tree $t_p$ associated to an unpaired intersection $p$ of $\cV$.
\end{prop}
Here \emph{geometrically dual} means that the bottom stage surface of each $G_k$ bounds a $3$--manifold which intersects 
exactly one generating curve of $H_1(\cV,\partial A)$ transversely in a single point, and is disjoint from the other generators.
In particular, there are as many gropes $G_k$ as free generators of $H_1(\cV,\partial A)$. 
Note that it follows from the last sentence of the proposition that if $\cV$ is order $n$, then each $G_k$ is class $n+2$. 

\emph{Proof:}  
Since the $A_j$ are simply-connected, the group $H_1(\cV,\partial A)$ is generated by sheet-changing curves in $\cV$ which pass once through a transverse intersection
(and avoid all other transverse intersections in $\cV$). Such curves either pass through an unpaired intersection
or a paired intersection. First we consider a sheet-changing curve through an unpaired intersection $p\in W_I\cap W_J$ 
(so $t_p=\langle I,J \rangle$).
The Clifford torus $T$ around $p$ is geometrically dual to the curve, and the dual pair of circles in $T$ represent meridians to
$W_I$ and $W_J$, respectively (recall our convention that if, say, $J=j$ is order $0$, Then $W_J=A_j$ is an order $0$ surface).
The next lemma shows that the circles on $T$ bound branches of the desired grope $G_{(I,J)}$, with $t(G_{(I,J)})=(I,J)$.
\begin{figure}[ht!]
\centerline{\includegraphics[scale=.3]{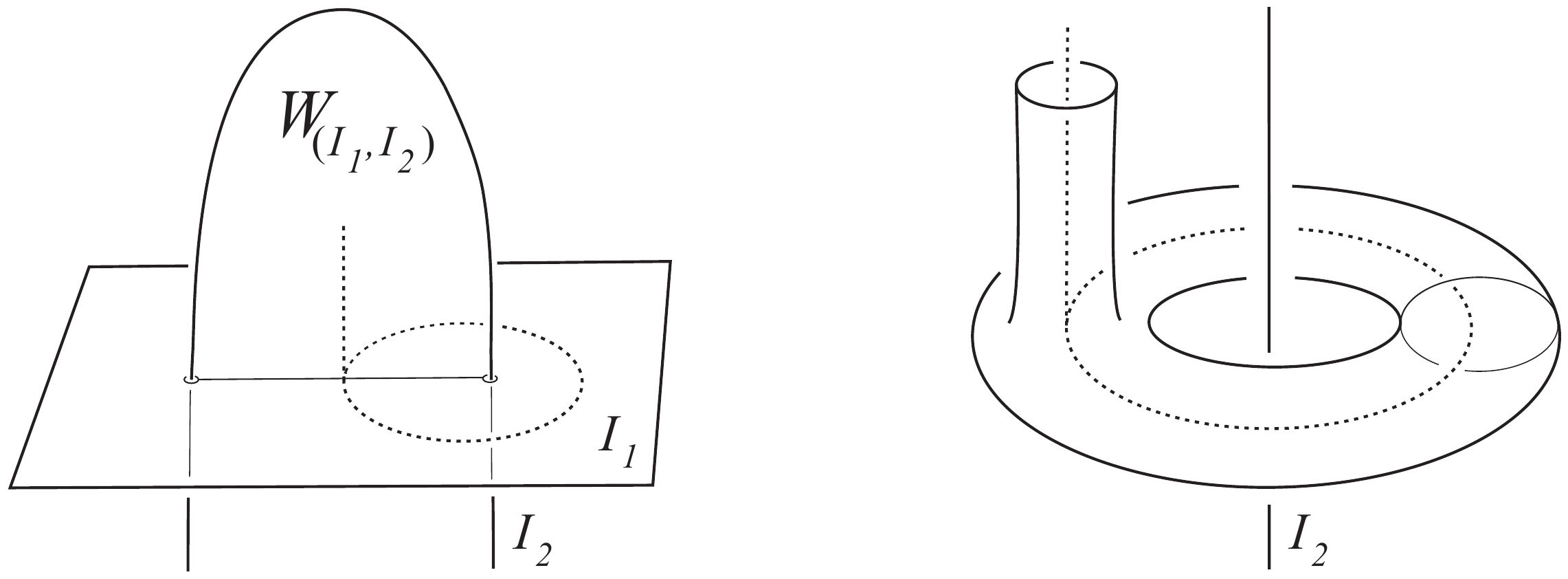}}
         \caption{The normal circle bundle $T_I$ to $W_{I_1}$ and $W_{(I_1,I_2)}$ over the dotted circle and arc on the left is shown on the right.}
         \label{fig:WI1I2-meridian-bounds-grope}
\end{figure}
\begin{lem}\label{lem:WI1I2-meridian-bounds-grope}
Any meridian to a Whitney disk $W_{(I_1,I_2)}$ in a Whitney tower $\cV\subset X$ 
bounds a grope $G_{(I_1,I_2)}\subset X\smallsetminus \cV$ such that $t(G_{(I_1,I_2)})=(I_1,I_2)$.
\end{lem}

\emph{Proof:}  
As illustrated in Figure~\ref{fig:WI1I2-meridian-bounds-grope}, such a meridian bounds a punctured Clifford
torus $T_I$ around one of the intersections paired by $W_{(I_1,I_2)}$. Each of a symplectic pair of circles
on $T_I$ is a meridian to one of the Whitney disks $W_{I_i}$ paired by $W_{(I_1,I_2)}$, so iterating
this construction until reaching meridians to order $0$ surfaces yields the desired grope $G_{(I_1,I_2)}$ with bottom stage $T_I$.
$\hfill\square$

Now we consider the sheet-changing curves through intersection points that are paired by Whitney disks.
Let $W_{(I,J)}$ be a Whitney disk, and consider the boundary $\gamma$ of a neighborhood
of a boundary arc of $W_{(I,J)}$ in one of the sheets paired by $W_{(I,J)}$, as illustrated in the left-hand side of 
Figure~\ref{fig:oval-bounds-grope}. We call such a loop $\gamma$ an \emph{oval} of the Whitney disk. Clearly, an oval intersects once with a sheet-changing curve that passes
once through one of the two intersections paired by $W_{(I,J)}$. So the normal circle bundle to the sheet over an oval is geometrically dual to such a sheet-changing curve. The following lemma completes the proof of Proposition~\ref{prop:w-tower-duality}. $\hfill\square$
\begin{figure}[ht!]
\centerline{\includegraphics[scale=.3]{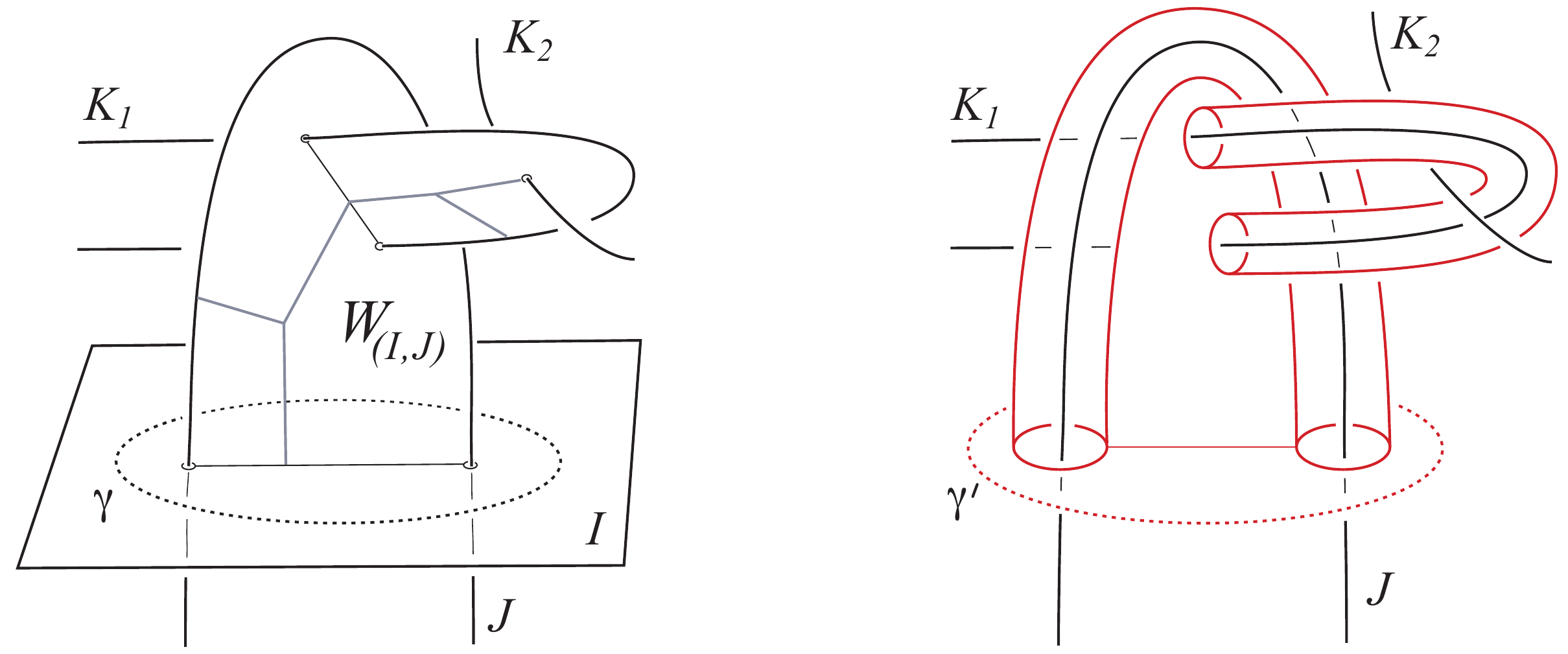}}
         \caption{Left: A (dotted) oval $\gamma\subset W_I$. Right: A (dotted) parallel $\gamma'$ to the oval bounds a grope in a nearby `time' coordinate. Not shown is the dual branch of the grope attached to the meridian of $W_J$ (as per Lemma~\ref{lem:WI1I2-meridian-bounds-grope}).}
         \label{fig:oval-bounds-grope}
\end{figure}

\begin{lem}\label{lem:ovals-bound-gropes}
Let $W_{(I,J)}$ be a Whitney disk in a split Whitney tower $\cV$ such that $W_{(I,J)}$ contains a trivalent vertex of
a tree $t_p=\langle (I,J),K \rangle$ associated to an unpaired intersection point $p\in\cV$.
If $\gamma\subset W_I$ is an oval of $W_{(I,J)}\subset\cV$; then the normal circle bundle 
$T$ to $W_I$ over $\gamma$ is the bottom stage of a dyadic grope $G\subset (X\smallsetminus \cV)$, such that 
$t(G)=( I,(J,K))$. 
\end{lem}

\emph{Proof:}  
The torus $T$ contains a symplectic pair of circles, one of which is a meridian to $W_I$, while the other is a parallel $\gamma'$ of $\gamma$.  
By Lemma~\ref{lem:WI1I2-meridian-bounds-grope}, the meridian to $W_I$ bounds a grope $G_I$ with $t(G_I)=I$, so we need to check that $\gamma'$ bounds
a grope $G_{(J,K)}$ with tree $(J,K)$.

As shown in Figure~\ref{fig:oval-bounds-grope}, $\gamma'$ bounds
a grope whose bottom stage contains a symplectic pair of circles, one of which is a meridian to $W_J$; while the other is either
parallel to an oval in $W_{(I,J)}$ around the boundary arc of a higher-order Whitney disk $W_{((I,J),K_1)}$ for $K=(K_1,K_2)$ (as shown in the figure), or is a meridian to $W_K$ if $W_{(I,J)}$ contains the unpaired intersection $p=W_{(I,J)}\cap W_K$ (since $\cV$ is split, these are the only two possible types of singularities in $W_{(I,J)}$). By Lemma~\ref{lem:WI1I2-meridian-bounds-grope},
the meridian to $W_J$ bounds a grope $G_J$; and inductively the oval-parallel circle, or again by Lemma~\ref{lem:WI1I2-meridian-bounds-grope}
the meridian to $W_K$, bounds a grope $G_K$; so the grope
bounded by $\gamma'$ does indeed have tree $(J,K)$.  
$\hfill\square$


\begin{figure}[ht!]
\centerline{\includegraphics[scale=.6]{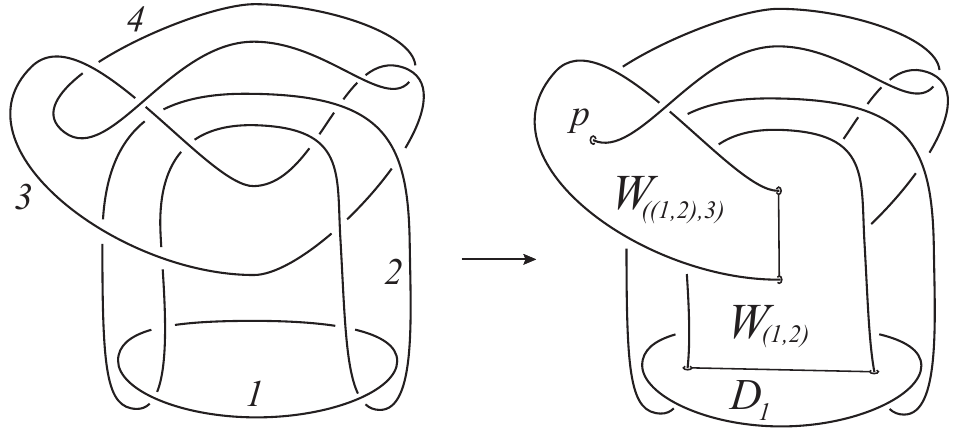}}
         \caption{Moving radially into $B^4$ from left to right, a link $L\subset S^3$ bounds an order $2$ 
         (non-repeating) Whitney tower $\cW$:
         The order $0$ disk $D_1$ consists of a collar on $L_1$ 
         together with the indicated embedded disk on the right. The other three order $0$ disks in $\cW$
         consist of collars on the other link components which extend further into $B^4$ and are capped off by disjointly embedded disks.
         The order $1$ Whitney disk $W_{(1,2)}$ pairs $D_1\cap D_2$, and the order $2$ Whitney disk $W_{((1,2),3)}$
         pairs $W_{(1,2)}\cap D_3$, with $p=W_{((1,2),3)}\cap D_4$
         the only unpaired intersection point in $\cW$. 
         See Figure~\ref{bing-hopf-example-tower-and-grope-with-tree} for the tree-preserving resolution of $\cW$ to a grope.}
         \label{bing-hopf-example-with-tower}
\end{figure}
\begin{figure}[ht!]
\centerline{\includegraphics[scale=.6]{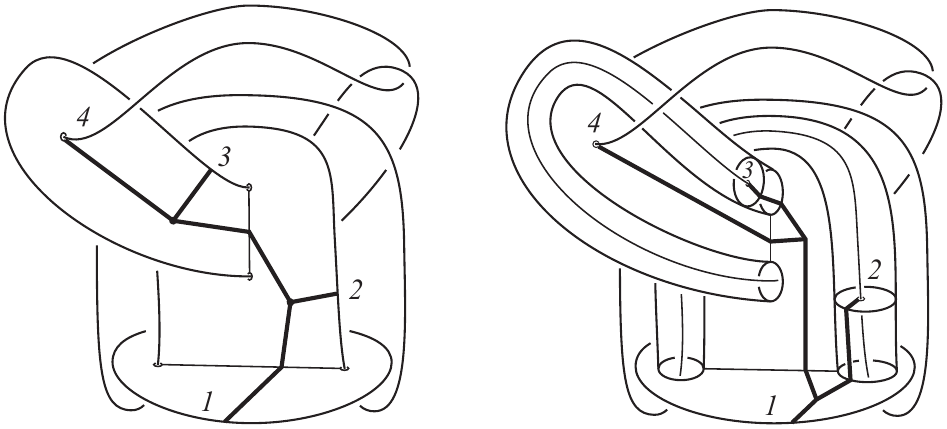}}
         \caption{Both sides of this figure correspond to the slice of 
         $B^4$ shown in the right-hand side of Figure~\ref{bing-hopf-example-with-tower}.
         The tree $t_p=\langle ((1,2),3),4 \rangle\subset \cW$ is shown on the left. 
         Replacing this left-hand side by the right-hand side illustrates the construction
         of a class $3$ (capped) grope $G_1^c$ bounded by $L_1$, shown (partly translucent) 
         on the right, gotten by surgering $D_1$ and $W_{(1,2)}$. Each of the disks $D_2$, $D_3$ and $D_4$
         has a single intersection with a cap of $G^c_1$, with $G_1$ displaying the longitude of $L_1$ as the triple
         commutator $[x_2,[x_3,x_4]]$ in $\pi_1(B^4\smallsetminus \cW^1)$, where $\cW^1=D_2\cup D_3\cup D_4$. 
         This simple case illustrates the local picture of 
         the general 
         computation of $\eta_n^i(\lambda_n(\cW))$: For a more complicated $L=\partial\cW$ this construction would be carried out in a 
         $4$--ball neighborhood of each tree containing an $i$-labeled vertex, and $\cW^i$ would consist of other 
         Whitney disks as well as the $D_{j\neq i}$. }
         \label{bing-hopf-example-tower-and-grope-with-tree}
\end{figure}



\subsection{The Whitney tower-to-grope construction}\label{subsec:w-tower-to-grope-review}
This subsection briefly sketches the Whitney tower-to-grope construction used above in 
Step (ii) of the proof of Theorem~\ref{thm:mu}.
In \cite{CST2} this procedure of converting Whitney towers to capped gropes in order to read off commutators determined by link longitudes is covered in detail in the setting of repeating Milnor invariants.
The analogous computation of repeating Milnor invariants from capped gropes described there is trickier in that meridians to a given link component $L_i$ can also contribute to the same longitude $\ell_i$. Hence the computation of $\ell_i$ uses a push-off $G_i'$
of the grope body $G_i$, and there may be intersections between the bottom stage of $G'_i$ and caps on $G_i^c$ which correspond
to repeating indices on the associated tree.

Here in the non-repeating setting, $\ell_i$ can be computed as described above directly from the body $G_i$
of the capped grope $G_i^c$, by throwing away the caps and just remembering how the tips of $G_i$ are meridians to the other components corresponding to the univalent labels on $t(G^c_i)$. See Figures~\ref{bing-hopf-example-with-tower} and~\ref{bing-hopf-example-tower-and-grope-with-tree} for the local model near a tree.

A typical 0-surgery which
converts a Whitney disk $W_{(I,J)}$ into a cap $c_{(I,J)}$ is
illustrated in Figure~\ref{fig:Wtower-to-grope1and2}, which
also shows how the signed tree is preserved. 
The sign associated
to the capped grope is the product of the 
signs coming from the
intersections of the caps with the bottom stages, which corresponds
to the sign of the un-paired intersection point in the Whitney
tower; (surgering along the other boundary arc of the Whitney
disk, and the other sign cases are checked similarly). If either of the $J$- and $K$-labeled
sheets is a Whitney disk, then the corresponding cap will be
surgered after a Whitney move which turns the single cap-Whitney
disk intersection into a cancelling pair of intersections between
the cap and a surface sheet that was paired by the Whitney disk,
as
described in Section~4.2 of \cite{S1}
(with orientations checked in Lemma~14,
Figures ~10 and 11 of \cite{ST2}).

\begin{figure}[ht!]
         \centerline{\includegraphics[scale=.4]{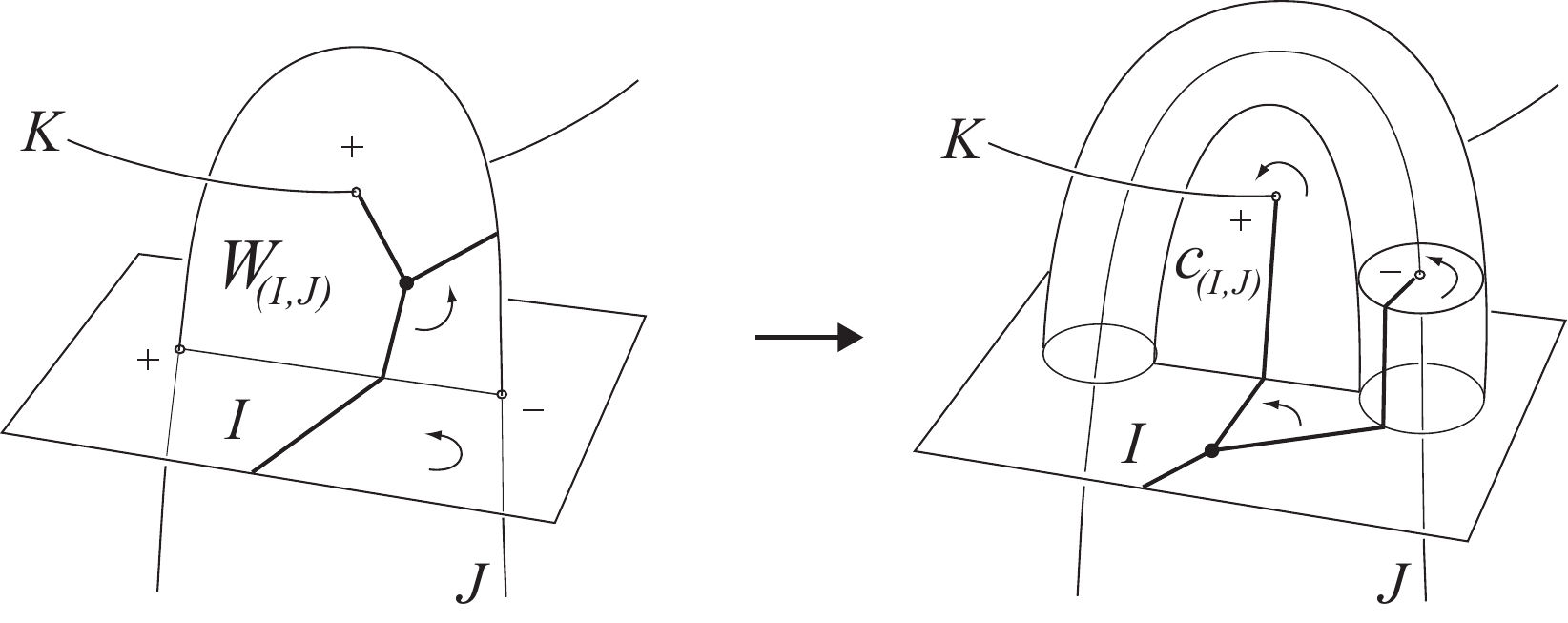}}
         \caption{Resolving a Whitney tower to a capped grope preserves
         the associated oriented trees. The boundary of the $I$-labeled sheet
         represents the commutator
         $[x_J,x_K]$, up to conjugation,
         of the meridians $x_J$ and $x_K$
         to the $J$- and $K$-labeled sheets.}
         \label{fig:Wtower-to-grope1and2}
\end{figure}

\section{Proof of Theorem~\ref{thm:in-X-L}}\label{sec:proof-thm-in-X-L}
Let $L\subset S^3$ bound an order $n$ Whitney tower in $B^4$, and let $X_L$ be the $4$--manifold gotten by attaching $0$-framed $2$--handles to $L$. We need to show:
\begin{enumerate}
\item Any map $A:\amalg^m S^2\to X_L$ of $2$--spheres into $X_L$ admits an order $n$
Whitney tower.
\item For any order $n$ non-repeating Whitney tower $\cW$ supported by $A:\amalg^m S^2\to X_L$, the non-repeating intersection invariant $\lambda_n(A):=\lambda_n(\cW)\in\Lambda_n(m)$ is independent of the choice of $\cW$. 
\end{enumerate}

The first statement of Theorem~\ref{thm:in-X-L} follows from the observations
in section~\ref{sec:properties-of-w-towers}: Any $A:\amalg^m S^2\to X_L$ is homotopic to the union of
band sums of parallel copies of cores of the $2$--handles of $X_L$ with $0$-framed immersed disks bounded by a link $L'$ formed from $L$ by the operations of adding parallel components, switching orientations, taking internal band sums and deleting components. Any order $n$ Whitney tower
on immersed disks in the $4$--ball bounded by $L$ can be modified to give an order $n$ Whitney tower on immersed disks bounded by $L'$ as described in 
subsection~\ref{sec:properties-of-w-towers}. Then the union of the Whitney tower bounded by $L'$ with the
$2$--handle cores forms an order $n$ Whitney tower supported by $A$.

To prove the second statement of Theorem~\ref{thm:in-X-L} we will use the following consequence of Theorem~\ref{thm:lambda-link-htpy}: If $\cV$ is any order $n$ non-repeating Whitney tower on a collection of $m$ immersed $2$--spheres in the $4$--sphere, then the order $n$ non-repeating intersection invariant $\lambda_n(\cV)$ must vanish in $\Lambda_n(m)$. Otherwise, the $2$--spheres supporting $\cV$ could be tubed into disjointly embedded $2$--disks in the $4$--ball bounded by an unlink $U$ in the $3$--sphere
to create an order $n$ Whitney tower $\cW_U$ in $B^4=B^4\#S^4$ with 
$\lambda_n(U)=\lambda_n(\cW_U)=\lambda_n(\cV)\neq 0\in\Lambda_n(m)$.

We start with the case where $L$ is an $m$-component link, and $A=\amalg^m_{i=1} A_i:S^2\to X_L$ is such that each $A_i$ goes geometrically once over the $2$--handle of $X_L$ attached to the $i$th component $L_i$ of $L$, and is disjoint from all other $2$--handles.
We assume that the orientations of $A$ and $L$ are compatible. In this case, the union of
an order $n$ Whitney tower $\cW_L$ bounded by $L$ 
with the cores of the $2$--handles forms an order $n$ Whitney tower $\cW$ on $A$, with
$\lambda_n(\cW)=\lambda_n(L):=\lambda_n(\cW_L)\in\Lambda_n(m)$. If $\cW'$ is any other
order $n$ non-repeating Whitney tower on $A'$, with $A'$ homotopic to $A$,
then we will show that $\lambda_n(\cW')=\lambda_n(\cW)\in\Lambda_n(m)$
by exhibiting the difference $\lambda_n(\cW)-\lambda_n(\cW')$ as $\lambda_n(\cV)$,
where $\cV$ is an order $n$ non-repeating Whitney tower on a collection of immersed $2$--spheres in the $4$--sphere.

To start the construction, let $\cW'\subset X_L=B\cup H_1\cup H_2\cup\cdots\cup H_m$ be an
order $n$ non-repeating Whitney tower on $A'$, with $A'$ homotopic to $A$. 
Here $B$ is the $4$--ball, and the $H_i$ are the $0$-framed $2$--handles.  
Any singularities of $\cW'$ which are contained in the $H_i$ can be pushed off by radial ambient isotopies, so that $\cW'$ may be assumed to only intersect the $H_i$ in disjointly embedded disks which are parallel copies of the handle cores. These embedded disks lie in the order $0$ $2$--spheres and the interiors of Whitney disks of $\cW'$. It also may be assumed that the trees representing $\lambda_n(\cW')$ are disjoint from all the $H_i$. 

The intersection $\cW'\cap \partial B$ is a link $L'$ in $S^3$, such that $L'$ is related to $L$ by adding some parallel copies of components and switching some orientations.
Note that since each $A_i$ goes over $H_i$ algebraically once, $L'$ contains $L$ as a sublink. Write $L'$ as the union $L'=L^0\cup L^1$ of two links where
the components of $L^0$ bound handle core disks in the order $0$ $2$--spheres of $\cW'$, and the
components of $L^1$ bound handle core disks in the Whitney disks (surfaces of order at least $1$) of $\cW'$. For each $i$, the components of $L^0$ which are parallel to $L_i\subset L$
must come in oppositely oriented pairs except for one component which is oriented
the same as $L_i$. The components of $L^1$ can be arbitrary parallels of components of $L$.

Now delete the $H_i$ from $X_L$, and form $S^4$ by gluing another $4$--ball $B^+$ to $B$ along their
$3$--sphere boundaries. Since $L$ bounds the order $n$ Whitney tower $\cW_L$ in $B^+$, an order $n$ Whitney tower $\cW^+\subset B^+$ bounded by $L'$ can be constructed using parallel order $0$ disks and Whitney disks of $\cW_L$ as in section~\ref{sec:properties-of-w-towers} above. The union of $\cW^+$ together with
$\mathring{\cW}':=\cW'\cap B$ is an order $n$ non-repeating Whitney tower 
$\cV:=\cW^+\cup\mathring{\cW}'$
on $m$ immersed $2$--spheres in $S^4$. Figure~\ref{fig:W-tower-on-2-spheres-in-S4-with-trees-color1} gives a schematic illustration of $\cV$.
\begin{figure}[ht!]
         \centerline{\includegraphics[scale=.5]{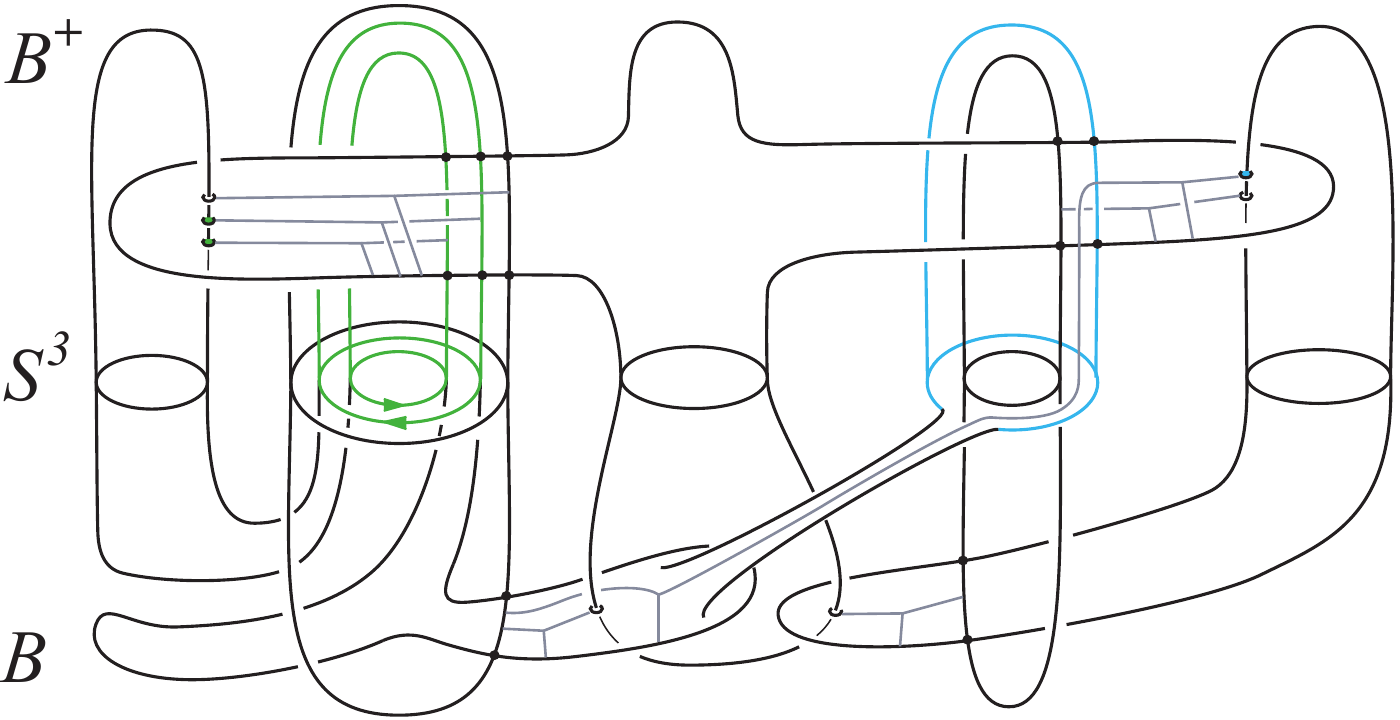}}
         \caption{The non-repeating Whitney tower $\cV=\cW^+\cup\mathring{\cW}'\subset S^4=B\cup_{S^3} B^+$: 
         The links $L\subset L'\subset S^3$ are shown in the
         horizontal middle part of the figure. The components of $L$ are black; 
         the component of $L^1\subset L'$ is blue, and an oppositely oriented
         pair of components in $L^0\subset L'$ are shown in green. The lower part of the figure shows $\mathring{\cW}'\subset B$, and 
         the upper part shows $\cW^+\subset B^+$. The tree shown involving the blue $L^1$-component 
         passes down
         through $L^1$ into a Whitney disk of $\mathring{\cW}'$ and then down into a pair of order $0$ disks in $\mathring{\cW}'$, 
         so the tree is of order greater than $n$. The pair of trees each having 
         a univalent vertex on a green order $0$ disk in $\cW^+$ have opposite signs due to the opposite orientations on the green disks.}
         \label{fig:W-tower-on-2-spheres-in-S4-with-trees-color1}
\end{figure}

We will check that $\lambda_n(\cV)=\lambda_n(\cW_L)-\lambda_n(\cW')\in\Lambda_n(m)$,
which will complete the proof in this case by the opening observation that $\lambda_n(\cV)$ vanishes.
We take the orientation of $(B^+,\partial B^+)$ in $S^4$ to be the standard orientation of $(B^4,S^3)$, and that of $(B,\partial B)$
to be the opposite. Since $\mathring{\cW}'\subset B$ contains all the trees representing $\lambda_n(\cW')$, these trees contribute the term $-\lambda_n(\cW')$
to $\lambda_n(\cV)$. 

Consider next the trees corresponding to intersections in $\cW^+$ involving components of $L^1$ (i.e.~the trees that intersect at least one order $0$ disk of $\cW^+$ bounded by a component of $L^1$). In $\cV$ these trees are subtrees of trees (for the same intersections)
which pass down through $L^1$ into the Whitney disks of $\mathring{\cW}'$ until reaching the order $0$ disks in $\mathring{\cW}'$.
Any such tree is of order strictly greater than $n$, since it contains an order $n$ proper subtree (the part of the tree in $\cW^+$) -- see Figure~\ref{fig:W-tower-on-2-spheres-in-S4-with-trees-color1}. Such higher-order trees do not contribute to $\lambda_n(\cV)$.

Consider now the remaining trees in $\cV$ which only involve the components of $L^0$. These trees represent $\lambda_n(L^0)=\lambda_n(\cW^0)$, where $\cW^0\subset \cW^+$ is the order
$n$ Whitney tower in $B^+$ bounded by $L^0\subset \partial B^+$; but we claim that in $\cV$ these trees contribute exactly $\lambda_n(L)$, which completes the proof in this case.
To see the claim, recall that $L^0$ consists of $L$ together with oppositely oriented pairs of parallel components 
of $L$. Denote by $+L_i^j$, $-L_i^j$ such a pair which is parallel to the $i$th component $L_i$ of $L$, and which bounded oppositely oriented handle-cores $+H_i$ and $-H_i$ in the $j$th component $A_j'$ of
$A'$. 
The univalent labels on trees representing $\lambda_n(L^0)=\lambda_n(\cW^0)$ which correspond to $+L_i^j$ and $-L_i^j$ when considered as trees in $\cV$ are labeled
by the same label $j$. 
Such re-labelings correspond exactly to the operations of Lemma~\ref{lem:canceling-parallels} in section~\ref{subsubsec:canceling-parallel-lemma}, which
implies that all trees involving such pairs of components contribute trivially to $\lambda_n(\cV)$,
verifying the claim.

The proof of Theorem~\ref{thm:in-X-L} in the general case follows the argument just given
with essentially only notational differences: An arbitrary $A$ is represented by the union of a linear combination of cores of the $H_i$ with immersed disks in $B^4$ bounded by a link $L_A$ formed from $L$ by the operations of adding parallel components, switching orientations, taking internal sums and deleting components. Since $L$ bounds an order $n$ Whitney tower, so does $L_A$ by section~\ref{subsubsec:canceling-parallel-lemma}. Hence $A$ supports an order $n$ Whitney tower $\cW$ with $\lambda_n(\cW)=\lambda_n(L_A)\in\Lambda_n(m)$.
One shows that $\lambda_n(\cW')=\lambda_n(L_A)$ for any non-repeating $\cW'$ on $A'$ homotopic to $A$ by proceeding as above with $L_A$ taking the place of $L$.


\section{Pulling apart parallel $2$--spheres}\label{sec:parallel-thm-proof}
In this section we prove Theorem~\ref{thm:parallel} of the
introduction, which states that for a map $A_0:S^2\to X$ of
a $2$--sphere in a simply connected $4$--manifold $X$ with vanishing normal Euler number, the homological self-intersection number $[A_0]\cdot[A_0]$ vanishes if and only if any number of parallel copies of $A_0$ can be pulled apart. 

Note that since the Euler number $e(A)$ of the normal bundle of a map $A:S^2\to X$ of a $2$-sphere in a $4$--manifold $X$ can be changed by $\pm 2$ by performing a cusp homotopy of $A$,
the condition $e(A)=0$ can be arranged if and only if the second Stiefel-Whitney class $\omega_2\in H^2(X;\Z_2)$ vanishes on $[A]$
(see e.g.~\cite[Sec.1.3A]{FQ}).
On the other hand, if $\omega_2([A])\neq 0$, then $[A]\cdot[A]$ is odd and hence not even two copies of $A$ can be pulled apart.

The proof of Theorem~\ref{thm:parallel} includes a geometric proof that boundary
links in the $3$--sphere are link-homotopically trivial
(Proposition~\ref{prop:boundary-link} below). We also give an
example (\ref{example:parallel-counter-example}) illustrating that 
the ``only if'' direction of Theorem~\ref{thm:parallel} is not generally true
in \emph{non}-simply connected $4$--manifolds.

\subsection{Proof of Theorem~\ref{thm:parallel}.}
We drop the subscript from notation and consider a map $A:S^2\to X$ with vanishing normal Euler number $e(A)=0$ and $X$ simply connected.
From the relationship $[A]\cdot[A]=e(A)+\lambda(A,A')$, and the hypothesis that $e(A)=0$, we have that $[A]\cdot[A]$ 
is equal to the Wall pairing 
$\lambda(A,A')$ which counts signed intersections between $A$ and any transverse parallel copy $A'$ (a generic normal section). (Since $X$ is simply connected, the Wall pairing is just the usual algebraic intersection number in $\Z$.) 
So the ``if'' direction of Theorem~\ref{thm:parallel} is clear, since $\lambda(A,A')$ obstructs pulling apart any two copies of $A$.

For the other direction, start by observing that the vanishing of $[A]\cdot [A]=\lambda(A,A')$ implies that $A$ supports an
order 1 Whitney tower $\cW$: The intersections between the parallels $A$ and $A'$ correspond in pairs to self-intersections of $A$,
so $\lambda(A,A')$ is equal to twice the sum of signed self-intersections of $A$.
These self-intersections must come in oppositely signed pairs, which admit Whitney disks since $X$ is simply connected.

First consider the case where $A$ also has vanishing order 1 intersection invariant (section~\ref{sec:order-one-int-trees}):
If $A$ is characteristic, then $\tau_1(A):=\tau(\cW)=0\in\cT_1(1)/\mathrm{INT}_1(A)\cong\Z_2=\Z/2\Z$; or if $A$ is not characteristic, 
then $\cT_1(1)/\mathrm{INT}_1(A)$ is trivial (see \cite[Sec.1]{ST1}).
By Theorem~2 of \cite{ST1} the vanishing of $\tau_1(A)$ implies that
$A$ admits an order 2 Whitney tower, so by Lemma~3 of \cite{S2}
for {\em any} $m\in\{3,4,5,\ldots\}$, $A$ admits a Whitney tower
$\cW$ of order $m$. (The fact that $A$ is \emph{connected} and $X$
is \emph{simply connected} is crucial here, since under these
hypotheses Lemma~3 of \cite{S2} shows that higher-order Whitney
towers can be built using a Whitney disk boundary-twisting construction.) Now,
taking parallel copies of the Whitney disks in $\cW$ yields an
order $m$ Whitney tower on $m+1$ parallel copies of $A$, as observed above in \ref{subsec:parallel-w-towers}. In
particular, we get an order $m$ non-repeating Whitney tower so, by
Theorem~\ref{thm:tower-separates}, the $m+1$ parallel copies
of $A$ can be pulled apart.


Consider now the case where $\tau_1(A)=\tau_1(\cW)$ is the
non-trivial element in $\Z_2$. We will first isolate (to a
neighborhood of a point) the obstruction to building an order 2 Whitney tower, and then combine the
previous argument away from this point with an application of
Milnor's Theorem~4 of \cite{M1} (which we will also prove geometrically in
Proposition~\ref{prop:boundary-link} below).

\begin{figure}[ht!]
         \centerline{\includegraphics[scale=.45]{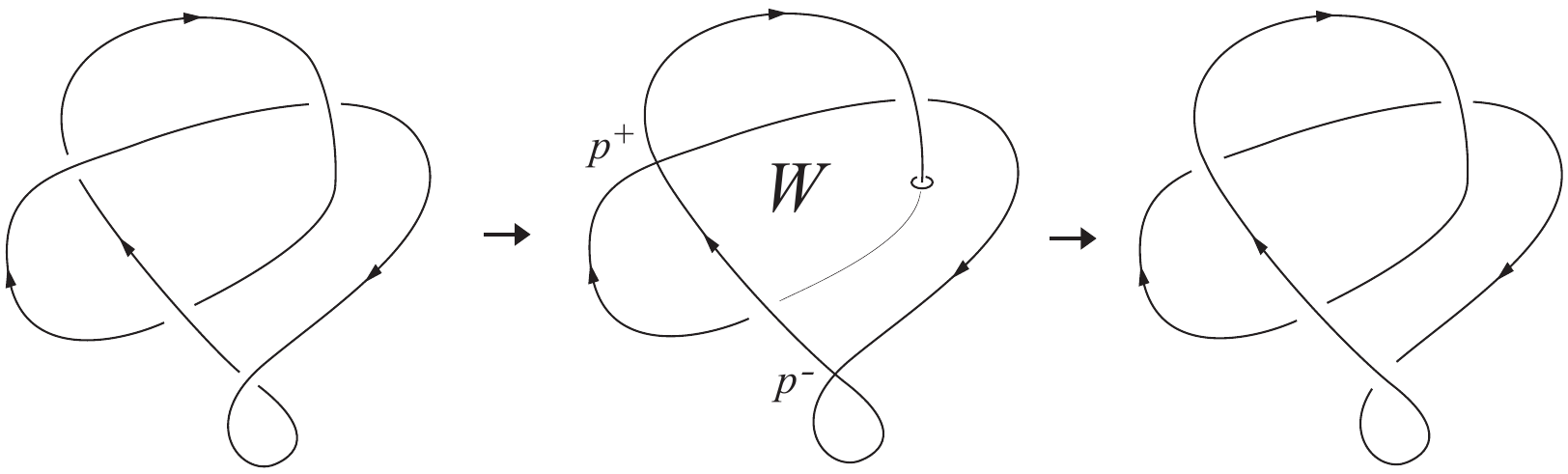}}
         \caption{Moving into $B^4$ from left to right: A trefoil knot in $S^3$ (left) bounds an immersed disk
         having a single pair of self-intersections $p^\pm$ admitting a Whitney disk $W$ containing a single 
         order $1$ intersection point (center). An unknotted `slice' of the immersed disk
         is shown on the right. The rest of the immersed disk is described by a null isotopy
         further into $B^4$ (not shown) of this unknot.}
         \label{fig:trefoil-arf}
\end{figure}
As illustrated in Figure~\ref{fig:trefoil-arf}, a trefoil knot in
the $3$--sphere bounds an immersed $2$--disk in the $4$--ball which
supports an order $1$ Whitney tower containing exactly one Whitney
disk whose interior contains a single order 1 intersection
point. It follows that the square knot, which is the connected sum
of a right- and a left-handed trefoil knot, bounds an immersed
disk $D$ in the $4$--ball which supports a Whitney tower $\cV$ containing
exactly two first order Whitney disks, each of which contains a
single order 1 intersection point with $D$. Being a well-known slice
knot, the square knot also bounds an embedded $2$--disk $D'$ in the
$4$--ball, and by gluing together two $4$--balls along their boundary
$3$--spheres we get an immersed $2$--sphere $S=D\cup D'$ in the 4--sphere
having the square knot as an ``equator'' and supporting the
obvious order 1 Whitney tower $\cW_S$ consisting of $S$ together with the two Whitney disks from $\cV$ pairing the intersections
in $D\subset S$.

Now take $\cW_S$ in a (small) $4$--ball neighborhood of a point in
$X$ (away from $A$), and tube (connected sum) $A$ into $S$. This does not change
the (regular) homotopy class of $A$, so we will still denote this
sum by $A$. Note that by construction there is a (smaller) $4$--ball
$B^4$ such that the intersection of the boundary $\partial B^4$ of
$B^4$ with $A$ is a trefoil knot (one of the trefoils 
in the connected sum decomposition of the square knot), and $B^4$ contains one of the
two Whitney disks of $\cW_S$. 
Denote by $X^\circ$ the result of removing from $X$ the interior of $B^4$,
and denote by $A^\circ$ the intersection of
$A$ with $X^\circ$ (so $A^\circ$ is just $A$ minus a small open disk). Since the order $1$ intersection
point in the Whitney disk of $\cW_S$ which is \emph{not} contained
in $B^4$ now cancels the obstruction $\tau_1(\cW)\in\Z_2$, we have
that $A^\circ$ admits an order 2 Whitney tower in $X^\circ$, and hence
again by Lemma~3 of \cite{S2}, $A^\circ$ admits a Whitney tower of any
order in $X^\circ$. As before, it follows that parallel copies of
$A^\circ$ can be pulled apart by using parallel (non-repeating) copies
of the Whitney disks in a high order Whitney tower on $A^\circ$. The
parallel copies of $A^\circ$ restrict on their boundaries to a link of
$0$-parallel trefoil knots in $\partial B^4$, and the proof of
Theorem~\ref{thm:parallel} is completed by the following lemma
which implies that these trefoil knots bound disjointly immersed
$2$--disks in $B^4$.$\hfill\square$
\begin{prop}\label{prop:boundary-link}
If the components $L_i$ of a link $L=\cup L_i\subset S^3$ are the
boundaries of disjointly embedded orientable surfaces $F_i\subset S^3$ in the
$3$--sphere, then the $L_i$ bound disjointly immersed $2$--disks in the
$4$--ball.
\end{prop}
This proposition first appeared as Milnor's Theorem~4 of \cite{M1}, and is a special case of the general results of
\cite{T} which are proved using symmetric surgery.

\emph{Proof:}  
Choose a symplectic basis of simple closed curves on each
$F_i$ bounding properly immersed $2$--disks into the $4$--ball. We shall refer to these disks as {\em caps}.
These caps may intersect each other, but the interiors of these caps lie in the interior of $B^4$ and so are disjoint from 
$\cup_i F_i\subset S^3$. The proof proceeds
inductively by using half of these caps to surger each $F_i$ to
an immersed disk $F_i^0$, while using the other half of the caps
to construct Whitney disks which guide Whitney moves to achieve
disjointness.

We start with $F_1$. Let $D_{1r}$ and $D_{1r}^*$ denote the
caps bounded by the symplectic circles in
$F_1$, with $\partial D_{1r}$ geometrically dual to $\partial
D^*_{1r}$ in $F_1$.

\textsc{Step 1:} Using finger moves, remove any interior
intersections between the $D^*_{1r}$ and any $D_{1s}$ by pushing
the $D_{1s}$ down into $F_1$
(Figure~\ref{fig:F1-surface-pushdown}).
\begin{figure}[ht!]
         \centerline{\includegraphics[scale=.45]{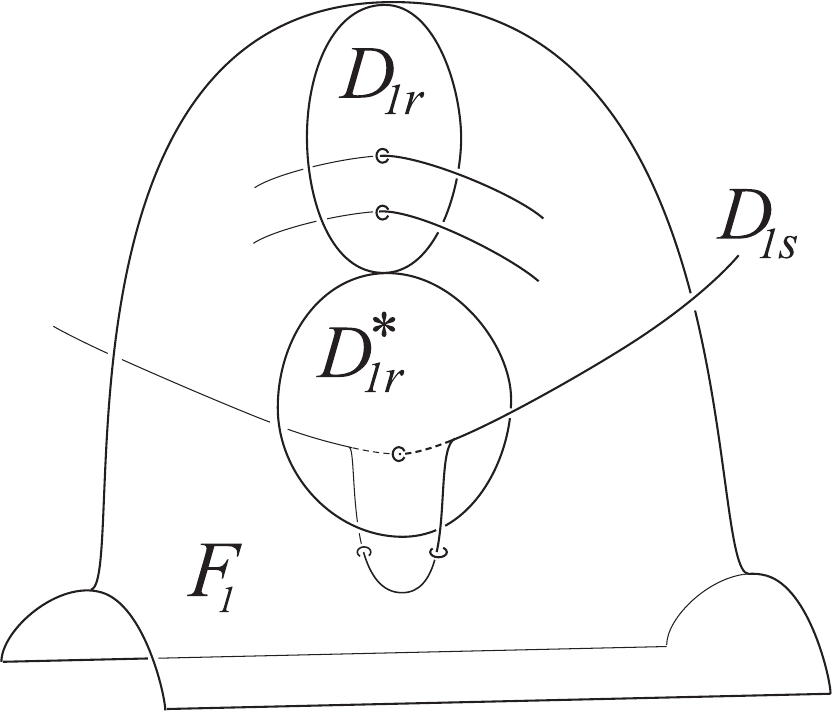}}
         \caption{}
         \label{fig:F1-surface-pushdown}
\end{figure}

\textsc{Step 2:} Surger $F_1$ along the $D_{1r}$
(Figure~\ref{fig:F1-surface-surgered}). The result is a properly
immersed $2$--disk $F_1^0$ in the $4$--ball bounded by $L_1$ in $S^3$.
The self-intersections in $F_1^0$ come from intersections and
self-intersections in the surgery disks $D_{1r}$, and any
intersections between the $D_{1r}$ and $F_1$ created in Step 1, as
well as any intersections created by taking parallel copies of the
$D_{1r}$ during surgery. We don't care about any of these
self-intersections in $F_1^0$, but we do want to eliminate all
intersections between $F_1^0$ and any of the disks $D_j$ on the
other $F_j$, $j\geq 2$. These intersections between $F_1^0$ and
the disks on the other $F_j$ all occur in cancelling pairs, with
each such pair coming from an intersection between a $D_{1r}$ and
a $D_j$. Each of these cancelling pairs admits a Whitney disk
$W_{1r}^*$ constructed by adding a thin band to (a parallel copy
of) the dual disk $D_{1r}^*$ as illustrated in
Figure~\ref{fig:F1-surface-surgered}. Note that by Step 1 the
interiors of the $D_{1r}^*$ are disjoint from $F_1^0$, hence the
interiors of the $W_{1r}^*$ are also disjoint from $F_1^0$. The
interiors of the $W_{1r}^*$ may intersect the $D_j$, but we don't
care about these intersections.
\begin{figure}[ht!]
         \centerline{\includegraphics[scale=.45]{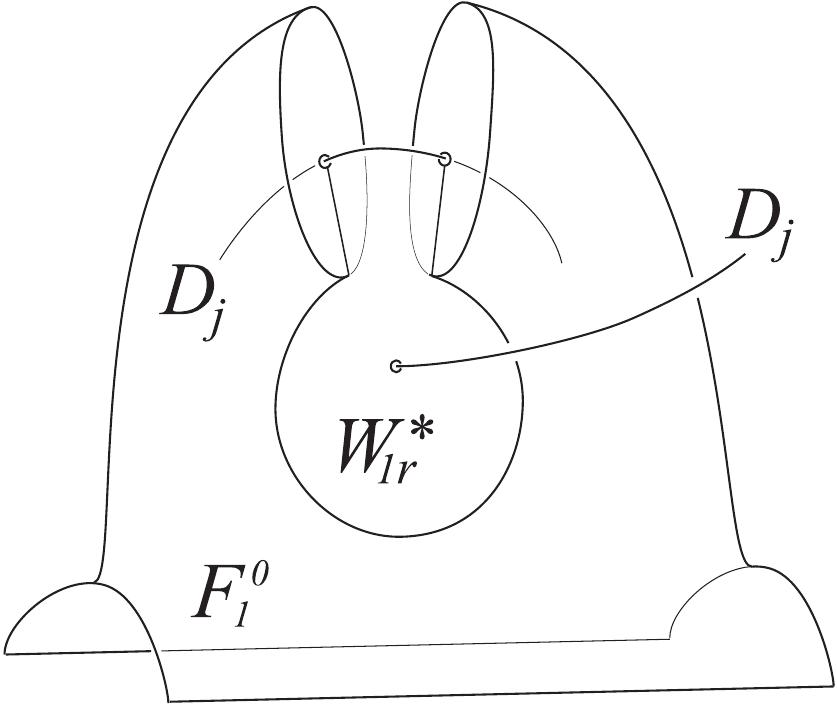}}
         \caption{}
         \label{fig:F1-surface-surgered}
\end{figure}

\textsc{Step 3:} Do the $W_{1r}^*$ Whitney moves on the $D_j$.
This eliminates all intersections between $F_1^0$ and the disks
$D_j$ on all the other $F_j$ ($j\geq 2$). Note that any interior
intersections the $W_{1r}^*$ may have had with the $D_j$ only lead
to more intersections among the $D_j$, so these three steps may be
iterated, starting next by applying Step 1 to $F_2$.
$\hfill\square$

\subsection{Example}\label{example:parallel-counter-example}
If $\pi_1X$ is non-trivial, then the conclusion of
Theorem~\ref{thm:parallel} may not hold, as we now illustrate. Let
$X$ be a $4$--manifold with $\pi_1X\cong\Z$, such that $\pi_2X$ has trivial order $0$ intersection
form; and let $A_1$ be an
immersed $0$-framed $2$--sphere admitting an order~$1$ Whitney tower
$\cW$ in $X$ with a single order~$1$ intersection point $p$ such
that $\tau_1(A_1)=t_p\in\cT_1(\Z,1)$ is represented by the single
$Y$-tree $Y(e,g,h)=t_p$ having one edge labeled by the trivial
group element $e$, and the other edges labeled by non-trivial
elements $g\neq h$, all edges oriented towards the trivalent
vertex.
Such examples are given in \cite{ST1}, and can be easily
constructed by banding together Borromean rings in the
boundary of $B^3\times S^1$ and attaching a $0$-framed two handle.

If $A_2$ and $A_3$ are parallel copies of $A_1$, then the order
1 non-repeating intersection invariant $\lambda_1(A_1,A_2,A_3)$ takes
values in $\Lambda_1(\Z,3)$ (since the vanishing of the 
order $0$ intersections means that all INT$_1$ relations are trivial).
By normalizing the group element decorating the edge adjacent to the $1$-label to the trivial
element using the HOL relations, $\Lambda_1(\Z,3)$ is
isomorphic to $\Z[\Z\times\Z]$. Using six parallel copies of the
Whitney disk in $\cW$, we can compute that
$\lambda_1(A_1,A_2,A_3)$ is represented by the sum of six
$Y$-trees $Y(e,g,h)$, where the univalent vertex labels vary over
the permutations of $\{1,2,3\}$ (see \cite[Thm.3.(iii)]{ST1}). This element corresponds to the
element
$$
(g,h)-(h,g)+(hg^{-1},g^{-1})-(g^{-1},hg^{-1})+(gh^{-1},h^{-1})-(h^{-1},gh^{-1})\in\Z[\Z\times\Z]
$$
which is non-zero if (and only if) $g$ and $h$ are distinct
non-trivial elements of $\Z$. Since $\lambda_1(A_1,A_2,A_3)\in\Lambda_1(\Z,3)$ is a well-defined homotopy 
invariant \cite[Thm.3]{ST1}, the $A_i$ can not be pulled apart whenever $g$ and $h$ are distinct and
non-trivial. 



\section{Dual spheres and stablizations}\label{sec:dual-and-stablization}
This section contains proofs of Theorem~\ref{thm:stable-separation} and Theorem~\ref{thm:duals}, both of which involve using low-order intersections to kill higher-order obstructions.
  
\subsection{Proof of Theorem~\ref{thm:duals}}\label{sec:casson-lemma-proof}
We need to show that surfaces $A_i$ with pairwise vanishing Wall intersections can be pulled apart if they have algebraic duals $B_i:S^2\to X$. 

\emph{Proof:}  
Wall's intersection pairing $\lambda(A_i,A_j)\in\Z[\pi]$ is defined when the $A_k:\Sigma_k\to X$ are maps of simply connected surfaces $\Sigma_k$, or more generally when the $A_k$ are $\pi_1$-null (Remark~\ref{rem:pi1-null}).
The pairwise vanishing of Wall's invariant gives an order 1
non-repeating Whitney tower on the $A_i$ (\ref{sec:order-zero-int-trees}). Assuming inductively for $1\leq n<
m-2$ the existence of an order $n$ non-repeating Whitney tower
$\cW$ on the $A_i$, it is enough to show that it can be arranged
that $\lambda_n(\cW)=0\in\Lambda_n(\pi,m)$, which allows us to
find an order $n+1$ non-repeating Whitney tower by 
Theorem~\ref{thm:non-rep-obstruction-theory}, and then to apply Theorem~\ref{thm:tower-separates}
when $n=m-2$.

By performing finger moves to realize the rooted product, any order
$n$ Whitney tower $\cW\subset X$ can be modified (in a neighborhood of a 1-complex) to
have an additional clean order $n$ Whitney disk $W_J$ whose
decorated tree corresponds to any given bracket $J$, with edges labeled
by any given elements of $\pi:=\pi_1X$. If $J$
is non-repeating and does not contain the label $i$, then tubing the $2$--sphere
$B_i$ into $W_J$ will change $\lambda_n(\cW)$ exactly by
adding the order $n$ generator
$\pm\langle J, i\rangle_g$, where the sign can
be chosen by the choice of orientation on $B_i$, and the element $g\in\pi$ decorating the $i$-labeled edge is determined by the choice of arc guiding the tubing (together with the whiskers on $W_J$ and $A_i$): Since $W_J$ is order $n$, 
any intersections between $B_i$
and other Whitney disks in $\cW$ will only contribute intersections of order strictly greater than $n$; and since $\lambda(A_j,B_i)=\delta_{ij} \in\Z[\pi]$, any other intersections between $A_j$ and $B_i$ contribute only canceling pairs of order $n$ intersections.  For $1\leq n$,
any generator of $\Lambda_n(\pi,m)$ can be realized as $\langle J, i\rangle_g$, 
so the just-described tubing procedure
can be used to modify $\cW$ until
$\lambda_n(\cW)=0\in\Lambda_n(\pi,m)$.
$\hfill\square$

\subsection{Proof of Theorem~\ref{thm:stable-separation}}\label{sec:stable-thm-proof}
We need to show that $\lambda_1(A)=0\in\Lambda_1(\pi,m)/\mathrm{INT}_1(A)$ if and only if $A:\amalg^m S^2\to X$ can be  pulled apart stably. 

Note that the vanishing of the homotopy invariant $\lambda_0(A)$ is implied by $\lambda_1(A)$ being defined.

\emph{Proof:}  
The ``if'' direction is immediate since $\lambda_1(A)$ only depends on the homotopy class of $A$ (by \cite{ST1}), and any $2$--spheres carried by the stabilization 
contribute trivially to $\mathrm{INT}_1(A)$. (See section~\ref{sec:order-1-INT} below for details on the $\mathrm{INT}_1(A)$
relations in the cases $m=3,4$ and $X$ simply connected.)

For the ``only if'' direction, observe first that the vanishing of $\lambda_1(A)$ gives an order 2
non-repeating tower supported by $A$ (by Theorem~\ref{thm:non-rep-obstruction-theory}).
Assuming inductively for 
$2\leq n< m-2$ the existence of an order $n$ non-repeating Whitney tower
$\cW$ on  $A$, it is enough to show that it can be arranged
that $\lambda_n(\cW)=0\in\Lambda_n(\pi,m)$, which allows us to
find an order $n+1$ non-repeating Whitney tower by 
Theorem~\ref{thm:non-rep-obstruction-theory}, and then to apply Theorem~\ref{thm:tower-separates}
when $n=m-2$.

For $n\geq 2$, any generator of $\Lambda_n(\pi,m)$ can be written as $\langle I,J \rangle_g$
where $I$ and $J$ are both of order greater than or equal to 1, and $g\in\pi$ decorates the edge where the roots of $I$ and $J$ are joined. As in the above proof of Theorem~\ref{thm:duals},
any order $n$ non-repeating Whitney tower $\cW$ on $A$ can be modified to have new clean Whitney disks $W_I$ and $W_J$, without
affecting $\lambda_n(\cW)$. Stabilizing the ambient 4--manifold by $S^2\times S^2$ and tubing $W_I$ and $W_J$ into a pair of dual 2--spheres coming from the stablization creates an intersection realizing
the generator $\langle I,J \rangle_g$, where the element $g\in\pi$ is determined by the choices of whiskers on $W_I$ and $W_J$ and the tubes into the dual spheres. By realizing generators in this way it can be arranged that $\lambda_n(\cW)=0$. 

By Poincar\'{e} duality, the same holds for any closed simply connected $4$--manifold other than $S^4$. For instance, for stablization by $\CP^2$ (or $\overline{\CP}^2$), where the dual 2--spheres
are copies of $\CP^1$, the framings on $W_I$ and $W_J$ can be recovered by boundary-twisting \cite[Sec.1.3]{FQ}, which only creates repeating intersections. 
$\hfill\square$

We remark that some control on the number of stablizations needed in Theorem~\ref{thm:stable-separation} can be obtained in terms of $m$ when $X$ is simply connected (so that $\Lambda_n(m)$ is finitely generated). 
For instance, a single stablization realizes $k$ times a generator by tubing $W_I$ or $W_J$ into $k$ copies (tubed together) of one of the dual
spheres. 


\section{Second order intersection
indeterminacies}\label{sec:order2-INT} 

It is an open problem to give necessary and sufficient algebraic conditions
for determining whether or not an arbitrary quadruple $A:\amalg^4 S^2\to X$ of 2--spheres in a $4$--manifold can be pulled apart.
The vanishing of $\lambda_0(A)$ and $\lambda_1(A)$ is of course necessary, and 
is equivalent to $A$ admitting an order 2 non-repeating Whitney
tower.
As explained in the introduction 
(\ref{sec:intro-INT-indeterminacies}), refining the sufficiency statement provided 
by Corollary~\ref{cor:tower=disjoint} requires the introduction
of \emph{intersection relations} INT$_2(A)$ in the target of $\lambda_2(A)$ which correspond to order $0$ 
and order $1$ intersections involving $2$-spheres which can be tubed into the Whitney disks 
of any Whitney tower $\cW$ supported by $A$.

With an eye towards stimulating future work, the goals of this section are to present some relevant details, describe some partial results, and introduce a
related number theoretic problem, while formulating order $2$ intersection relations  
which make the following conjecture precise:
\begin{conj}\label{conj:INTm=4}
If a quadruple of immersed $2$--spheres $A:\amalg^4 S^2\to X$ in a simply connected $4$--manifold $X$ admits an order 2 non-repeating Whitney
tower $\cW$, then  $A$ can be pulled apart if and
only if $\lambda_2(A):=\lambda_2(\cW)$ vanishes in
$\Lambda_2(4)/\mathrm{INT}_2(A)$.
\end{conj}

This section is somewhat technical, so we begin by providing an outline:
After quickly recalling in \ref{subsec:order-zero-no-indeterminacies} the lack of indeterminacies in the 
order $0$ non-repeating invariant $\lambda_0$,
the intersection relations INT$_1$ in the target of the order $1$ non-repeating invariant $\lambda_1$ are examined in detail for triples and then quadruples of $2$--spheres, including notation and examples intended to clarify and motivate the introduction of the intersection relations INT$_2(A)$ in the target of the order $2$ non-repeating invariant $\lambda_2$.
These INT$_2(A)$ relations, which are determined by $\lambda_0$ and $\lambda_1$ on $\pi_2X$, are discussed throughout section~\ref{subsec:order-2-INT}.
Section~\ref{subsec:tubing-order-2-indeterminacies} observes that if $A$ has any non-trivial order $0$ intersections with any other $2$--spheres in $X$, then the target $\Lambda_2(4)/\mathrm{INT}_2(A)$ of $\lambda_2(A)$ must be finite; and presents two related results, Proposition~\ref{dual-in-order2-tower} and Proposition~\ref{stabilization-in-order2-tower}, which give sufficient conditions for pulling apart $A$ in the setting where $\lambda_2(A)$ is defined.
Section~\ref{sec:linear-INT2} describes the INT$_2$ relations as the image in $\Lambda_2\cong\Z\oplus\Z$ of a linear map determined by $\lambda_1$ on $\pi_2X$ in the setting where $\lambda_0$ vanishes on $\pi_2X$, as motivation for the discussion in section~\ref{sec:quadratic-INT2} on how non-trivial values of $\lambda_0$ away from $A$ can affect the INT$_2$ relations.
Section~\ref{subsubsec:quadratic-INT2-maps} shows how the INT$_2$ relations can be computed as the image of a map 
whose non-linear part is determined by Diophantine quadratic equations coupled by the order $0$ intersection form 
$\lambda_0$ on $\pi_2X$, leading naturally to some relevant number theoretic questions.

Throughout the rest of this section we assume that the ambient $4$--manifold $X$
is simply connected. For brevity we suppress the domains of the components of $A$ from notation 
and consider collections $A=A_1,\ldots,A_m\imra X$ of immersed $2$--spheres.

\subsection{Order $0$ intersection invariants}\label{subsec:order-zero-no-indeterminacies}
Recall (\ref{sec:order-zero-int-trees}) that the order $0$
non-repeating intersection invariant
$\lambda_0(A_1,\ldots,A_m)=\sum\sign(p) \cdot i
\!-\!\!\!\!-\!\!\!\!-\! j\in\Lambda_0(m)$ on $2$--spheres immersed in a simply connected $4$--manifold $X$ carries exactly the same information as the
integral homological intersection form on $H_2(X)$, with the sum of
the coefficients of the $i \!-\!\!\!\!-\!\!\!\!-\! j$
corresponding to the usual homological intersection number $[A_i]\cdot [A_j]\in\Z$. There are no intersection
indeterminacies in this order $0$ setting, and
$A_1,\ldots,A_m$ admits an order $1$ non-repeating Whitney tower if
and only if $\lambda_0(A_1,\ldots,A_m)$ vanishes in
$\Lambda_0(m)$ (which is isomorphic to a direct sum of $\binom{m}{2}$ copies of
$\Z$, one for each (unordered) pair of distinct indices $i,j$).

\subsection{Order $1$ intersection relations.}\label{sec:order-1-INT}
The order $1$ intersection relations INT$_1$ are described by order $0$ intersections $\lambda_0$. These INT$_1$ relations are examined here in detail for triples
and quadruples of $2$--spheres, as notational and motivational preparation for describing
the order $2$ intersection relations. 

\subsubsection{Order $1$ triple intersections}
For a triple of immersed $2$--spheres $A_1, A_2, A_3\looparrowright X$ with $\lambda_0(A_1, A_2, A_3)=0$, the order~1
non-repeating intersection invariant $\lambda_1(A_1, A_2, A_3)$ is a
sum of order 1 $Y$-trees in $\Lambda_1(3)\cong\Z$ modulo the
INT$_1(A_1,A_2,A_3)$ intersection relations:
$$
_i^j>\!\!\!-\,\!\lambda_0(S_{(i,j)},A_k)=0
$$
where $S_{(i,j)}$ ranges over $\pi_2X$, and $(i,j)$ ranges over
the three choices of pairs from $\{1,2,3\}$. 
(Here the notation $_i^j>\!\!\!\!-\,\!\lambda_0(S_{(i,j)},A_k)$ indicates the sum of trees gotten by
attaching the root of $(i,j)$ to the $(i,j)$-labeled univalent vertices in $\lambda_0(S_{(i,j)},A_k)$ corresponding to $S_{(i,j)}$.)
Geometrically, these
relations correspond to tubing any Whitney disk $W_{(i,j)}$ into
any $2$--sphere $S_{(i,j)}$. Via the identification $\Lambda_1(3)\cong\Z$, the quotient
$\Lambda_1(3)/\mathrm{INT}_1(A_1,A_2,A_3)$ is isomorphic to $\Z_d=\Z/d\Z$, where $d$
is the greatest common divisor of all the
$\lambda_0(S_{(i,j)},A_k)$. This invariant $\lambda_1(A_1, A_2,
A_3)\in\Lambda_1(3)/\mathrm{INT}_1(A_1,A_2,A_3)$ is the \emph{Matsumoto
triple} \cite{Ma} which vanishes if and only if $A_1, A_2, A_3$ admit an order 2
non-repeating Whitney tower (and hence can be pulled apart \cite{Y}).

\textbf{Examples:} In the $4$--manifold $X_L$ gotten by attaching $0$-framed $2$-handles
to the Borromean rings $L=L_1\cup L_2\cup L_3\subset S^3=\partial B^4$, all INT$_1$ relations are trivial, and the
triple $A_1$, $A_2$, $A_3$ of $2$--spheres determined (up to homotopy) by the link components
can not be pulled apart since 
$\lambda_1(A_1, A_2, A_3)$ is equal to ($\pm$) the generator $_1^2>\!\!\!-\,\text{\small{3}}$
of $\Lambda_1(3)\cong\Z$.

If $X_L$ is changed to $X_L'$ by attaching another 2-handle along a meridional circle to
$L_3$, then INT$_1(A_1,A_2,A_3)=\Lambda_1(3)$ since 
$\,_1^2>\!\!\!-\,\!\lambda_0(S_{(1,2)},A_3)=(\pm)_1^2>\!\!\!\!-\,\text{\small{3}}$, where $S_{(1,2)}$ is the new 2--sphere
which is dual to $A_3$. Now $A_1, A_2, A_3\looparrowright X_L'$
can be pulled apart since $\lambda_1(A_1, A_2, A_3)$ takes values in the trivial group.

\subsubsection{Computing the $\mathrm{INT}_1(A_1,A_2,A_3)$
intersection relations} 
Since each element of $\pi_2X$ can affect the non-repeating order $1$ indeterminacies in three independent ways
(by tubing 2--spheres into Whitney disks $W_{(1,2)}$, $W_{(1,3)}$, and $W_{(2,3)}$) the 
$\mathrm{INT}_1(A_1,A_2,A_3)$ relations can be computed as the image
of a linear map  $\Z^r\oplus\Z^r\oplus\Z^r\longrightarrow\Z$, with $r$ the rank of
the $\Z$-module
$\pi_2X$ modulo torsion. Specifically, let $S^{\alpha}$ be a basis for
$\pi_2X$ (mod torsion), and define integers
$a^{\alpha}_{ij}:=\lambda_0(S^{\alpha}_{(i,j)},A_k)$ for
$S^{\alpha}_{(i,j)}$ ranging over the basis, and $i$, $j$, $k$ distinct.
Then,
identifying $\Lambda_1(3)\cong\Z$, the INT$_1(A_1,A_2,A_3)$
intersection relations can be described as
$$
\sum_{\alpha}(x_{12}^{\alpha}a_{12}^{\alpha}+x_{31}^{\alpha}a_{31}^{\alpha}
+x_{23}^{\alpha}a_{23}^{\alpha})=0
$$
with the coefficients $x_{ij}^{\alpha}$ ranging (independently) over $\Z$.

Using integer vector
notation, this map can be written concisely as:
$$
(x_{12},x_{31},x_{23})\longmapsto x_{12}\cdot a_{12}+x_{31}\cdot
a_{31}+x_{23}\cdot a_{23}.
$$
with ``$\,\cdot\,$'' denoting the dot product in $\Z^r$.


\subsubsection{Order $1$ quadruple intersections}\label{sec:order1-of-quadruple}

For a collection $A$ of \emph{four} immersed $2$--spheres $A=A_1, A_2,
A_3, A_4\looparrowright X$ with vanishing $\lambda_0(A)$, the order~1
non-repeating intersection invariant $\lambda_1(A)$ takes values in
$\Lambda_1(4)/\mathrm{INT}_1(A)$, with the INT$_1(A)$ relations given by
$$
\,_i^j>\!\!\!-\,\lambda_0(S_{(i,j)},A_k,A_l)=0
$$
where $S_{(i,j)}$ ranges over $\pi_2X$, and $(i,j)$ ranges over
the six choices of distinct pairs from $\{1,2,3,4\}$. Each such relation corresponds to tubing the $2$--sphere 
$S_{(i,j)}$ into a Whitney disk $W_{(i,j)}$.
Here
$\Lambda_1(4)\cong\Z\oplus\Z\oplus\Z\oplus\Z$, and each generator
of the rank $r$ $\Z$-module $\pi_2X$ modulo torsion gives six relations, so the target group
$\Lambda_1(4)/\mathrm{INT}_1(A)$ of $\lambda_1(A)$ is the
quotient of $Z^4$ by the image of a linear map from $\Z^{6r}$. The invariant
$\lambda_1(A)$ vanishes in $\Lambda_1(4)/\mathrm{INT}_1(A)$ if and only if
$A$ admits an order 2 non-repeating Whitney tower. 

\textbf{Example:} Note that each of the four copies of $\Z$ in $\Lambda_1(4)$
corresponds to a target of a Matsumoto triple (a choice of three distinct indices), but the vanishing
of the all the triples is not sufficient to get an order 2
non-repeating Whitney tower on the $A$ because of ``cross-terms''
in the INT$_1$ relations; the simplest example is the following:

Consider a five component link $L=L_1\cup \cdots\cup L_5\subset S^3=\partial B^4$ such that
$L_1\cup L_2\cup L_3$ forms a Borromean rings which is split from the component $L_4$, 
and $L_5$ is a band sum of (positive) meridians to $L_3$ and $L_4$. 
In the $4$--manifold gotten by attaching 0-framed 2-handles to $B^4$ along $L$,
let $A_i$ denote the immersed 2--sphere determined (up to homotopy) by the core of the 2-handle attached to
$L_i$.

Now any three of the quadruple $A_1,A_2,A_3,A_4$ will have vanishing first
order triple $\lambda_1(A_i,A_j,A_k)$ in $\Lambda_1(3)/\mathrm{INT}_1(A_i,A_j,A_k)$
for any choice of distinct $i$, $j$, $k$: 
Since $A_5$ is dual to $A_3$, the generator
$_1^2>\!\!\!\!-\,3$ of $\Lambda_1(3)$ is killed by INT$_1(A_1,A_2,A_3)$; and for the other
choices of $1\leq i< j< k\leq 4$ it is clear that $\lambda_1(A_i,A_j,A_k)$ vanishes since $L_i\cup L_j\cup L_k$
is a split link (so $A_i,A_j,A_k$ can be pulled apart). 

But
the first order quadruple 
$\lambda_1(A_1,A_2,A_3,A_4)=\,_1^2>\!\!\!\!\!-\,\text{\small{3}}=-\,_1^2>\!\!\!\!\!-\,\text{\small{4}}$ is non-zero
in $\Lambda_1(4)/\mathrm{INT}_1(A_1,A_2,A_3,A_4)\cong\Z^3$, where the only non-trivial 
INT$_1(A_1,A_2,A_3,A_4)$ relation is 
$$
_1^2>\!\!\!-\,\lambda_0(A_5,A_3,A_4)=\,_1^2>\!\!\!-\,\text{\small{3}}+\,_1^2>\!\!\!-\,\text{\small{4}}=0.
$$ 

Geometrically,
any order $1$ intersection in a Whitney tower on $A_1\cup A_2\cup A_3$ can be killed by tubing $A_5$ into a Whitney disk pairing intersection between $A_1$ and $A_2$ to create a canceling order $1$ intersection, but this
also creates an order $1$ intersection between the Whitney disk and $A_4$.

\subsubsection{Computing the $\mathrm{INT}_1(A_1,A_2,A_3,A_4)$ relations} 

Choose a basis $S^{\alpha}$ for $\pi_2X$
(mod torsion), and define integers
$a^{\alpha}_{ij,k}:=\lambda_0(S^{\alpha}_{(i,j)},A_k)$. Then each
element of the subgroup INT$_1(A_1,A_2,A_3,A_4)<\Lambda_1(4)$ can be written
$$
\begin{array}{rcl}
  \,_i^j>\!\!\!-\,\lambda_0(\sum_{\alpha}x^{\alpha}_{ij}S^{\alpha}_{(i,j)},A_k,A_l) & =  &\kern-2pt(\sum_{\alpha}x^{\alpha}_{ij}a^{\alpha}_{ij,k})\,_i^j>\!\!\!-\,{\scriptstyle k}+
(\sum_{\alpha}x^{\alpha}_{ij}a^{\alpha}_{ij,l})\,_i^j>\!\!\!-\,{\scriptstyle l} \\
  & =  &(x_{ij}\cdot a_{ij,k})\,_i^j>\!\!\!-\,{\scriptstyle k}+(x_{ij}\cdot
a_{ij,l})\,_i^j>\!\!\!-\,{\scriptstyle l}\\
\end{array}
$$
where the coefficients in the last expression are dot products
of vectors in $\Z^r$, with $i,j,k,l$ distinct, and $r$ the rank of $\pi_2X$ (mod torsion). Using the basis
$$
\{\,\,_2^1>\!\!\!-\,{\scriptstyle 3}\,\,, \,\,\,_2^1>\!\!\!-\,\!{\scriptstyle 4}\,\,,\,\,
\,_3^1>\!\!\!-\,\!{\scriptstyle 4}\,\,, \,\,\,_3^2>\!\!\!-\,\!{\scriptstyle 4}\,\, \}
$$
for $\Lambda_1(4)$,
the subgroup INT$_1(A_1,A_2,A_3,A_4)$ is the image of the linear map
$\Z^{6r}\longrightarrow\Z^4$:
$$
\begin{pmatrix}
  x_{12} \\
  x_{13} \\
  x_{41} \\
  x_{23} \\
  x_{24} \\
  x_{34}
\end{pmatrix}\longmapsto
\begin{pmatrix}
  a_{12,3} & -a_{13,2} & 0 & a_{23,1} & 0 & 0 \\
  a_{12,4} & 0 & a_{41,2} & 0 & a_{24,1} & 0 \\
  0 & a_{13,4} & a_{41,3} & 0 & 0 & a_{34,1} \\
  0 & 0 & 0 & a_{23,4} & -a_{24,3} & a_{34,2}
\end{pmatrix}
\begin{pmatrix}
  x_{12} \\
  x_{13} \\
  x_{41} \\
  x_{23} \\
  x_{24} \\
  x_{34}
\end{pmatrix}
$$
where the multiplication of entries is the vector dot product in $\Z^r$.

\subsection{Order 2 intersection relations.}\label{subsec:order-2-INT}

Now consider a quadruple
of immersed 2--spheres $A=A_1, A_2,
A_3, A_4\looparrowright X$ in a simply connected 4--manifold $X$, such that $\lambda_1(A)=0\in\Lambda_1(4)/\mathrm{INT}_1(A)$,
so that $A$ supports an order $2$ non-repeating Whitney tower $\cW\subset
X$.

Recall that we want to describe order $2$ intersection relations INT$_2(A)$ which account for changes in the choice of Whitney tower on $A$
and define the target of $\lambda_2(A)\in\Lambda_2(4)/\mathrm{INT}_2(A)$. Note that $\Lambda_2(4)$ is isomorphic to $\Z\oplus\Z$, generated,
for instance, by the elements
$$
t_1:=\,\,^2_1>\!\!\!-\!\!\!\!-\!\!\!\!\!-\!\!\!<^{\,3}_{\,4} \quad \mathrm{and} \quad
t_2:=\,\,^3_1>\!\!\!-\!\!\!\!-\!\!\!\!\!-\!\!\!<^{\,2}_{\,4},
$$
with the IHX relation giving:
$$
^4_1>\!\!\!-\!\!\!\!-\!\!\!\!\!-\!\!\!<^{\,2}_{\,3}\,\,=t_1+t_2.
$$

We will mostly be concerned with the case that $A$ is in the radical of $\lambda_0$ on $\pi_2X$, so that for each $i\in\{1,2,3,4\}$
the order $0$ pairing 
$\lambda_0(S,A_i)$ vanishes for
any immersed $2$--sphere $S$, but first we make some quick general observations related to Theorems~\ref{thm:stable-separation} and \ref{thm:duals} above.

\subsubsection{Tubing order $2$ Whitney disks into spheres}\label{subsec:tubing-order-2-indeterminacies}
Let $i,j,k,l$ be distinct
indices from $\{1,2,3,4\}$. 
As already observed in the proof of Theorem~\ref{thm:duals}, $\cW$ can be modified to
have an additional clean order $2$ Whitney disk $W_{((i,j),k)}$ without creating any new unpaired intersections. If $S_{((i,j),k)}$ is any immersed $2$--sphere,
then tubing 
$W_{((i,j),k)}$ into $S_{((i,j),k)}$ preserves the order of the Whitney tower and changes $\lambda(\cW)$ by $a_{ijk}\cdot\langle ((i,j),k), l \rangle$, where  $a_{ijk}= \lambda_0(S_{((i,j),k)},A_l)\in\Z$ (since any intersections between the new Whitney disk $W_{((i,j),k)}\# S_{((i,j),k)}$ and $A_i,A_j,A_k$ are repeating intersections).

Letting $S_{((i,j),k)}$ vary over a basis $S^{\alpha}$ for $\pi_2X$
(mod torsion) for distinct triples $i$, $j$, $k$ in $\{1,2,3,4\}$, 
this construction generates a subgroup of $\Lambda_2(4)$ isomorphic to
$d\Z\oplus d\Z$, where $d$
is the greatest common divisor of $\lambda_0(S^\alpha,A_i)$ over all
$S^\alpha$ and $i$. In particular, if these order $0$ intersections are relatively prime, then the target 
$\Lambda_2(4)/\mathrm{INT}_2(A)$
of 
$\lambda_2(A)$ is trivial:
\begin{prop}\label{dual-in-order2-tower}
If $A=A_1,A_2,A_3,A_4$ admits an order $2$ non-repeating Whitney tower
and if $\mathrm{gcd}(\{ \lambda_0(S^\alpha,A_i)\}_{\alpha,i})=1$, 
then $A$ can be pulled apart. $\hfill\square$
\end{prop}

\subsubsection{Tubing order $1$ Whitney disks into spheres}\label{subsec:tubing-order-1-indeterminacies}
Again as in the proof of Theorem~\ref{thm:duals}, for any choice
of distinct indices $\cW$ can be modified to have two additional clean order $1$ Whitney disks $W_{(i,j)}$ and $W_{(k,l)}$.
Tubing either of these Whitney disks into an arbitrary $2$--sphere might create unpaired order $1$ non-repeating intersections (between $A$ and the $2$--sphere) and hence not preserve the order of $\cW$, however tubing into $2$--spheres created by stabilization does indeed preserve the order (since intersections among order $1$ Whitney disks are of order $2$). In fact, a single stablization is all that is needed to kill any obstruction to pulling
apart the quadruple $A_1,A_2,A_3,A_4$:
\begin{prop}\label{stabilization-in-order2-tower}
If $A=A_1,A_2,A_3,A_4$ admit an order $2$ non-repeating Whitney tower, 
then $A$ can be pulled apart in the connected sum of 
$X$ with a \emph{single} $S^2\times S^2$ (or a \emph{single} $\CP^2$, or a \emph{single} $\overline{\CP}^2$). 
\end{prop}

 \emph{Proof:}  
 $S^2\times S^2$, or $\CP^2$, or $\overline{\CP}^2$. We will show how to change $\lambda_2(\cW)$ by any integral linear combination $a_1 t_1+a_2 t_2$
 of the above generators $t_1$, $t_2$ of $\Lambda_2(4)$: To create $a_1 t_1$, first modify $\cW$ to have two additional clean Whitney disks $W_{(1,2)}$ and $W_{(3,4)}$, then tube 
 $W_{(1,2)}$ into $S$, and tube $W_{(3,4)}$ into $|a_1|$-many copies of $S'$ (where the sign of $a_1$ corresponds to the orientations of the copies of $S'$). Note that in case of stabilization by $\CP^2$ or $\overline{\CP}^2$, the extra intersections coming from
 taking $|a_1|$ copies of $\CP^1$ are all repeating intersections, so that $\lambda_2(\cW)$ is indeed only changed by $a_1 t_1$.
 Now, to further create $a_2 t_2$ proceed in the same way starting with two additional clean Whitney disks $W_{(1,3)}$ and $W_{(2,4)}$,
which are tubed into a parallel copy of the same $S$ and $|a_2|$-many copies of $S'$. This will also create intersections with the previous copies of $S$ and $S'$,
but these extra intersections will all be repeating intersections. (Any Whitney disks tubed into copies of $\CP^1$ can be framed as in the proof of 
Theorem~\ref{thm:stable-separation}, see section~\ref{sec:stable-thm-proof}.)
$\hfill\square$
\begin{rem}
By Poincar\'{e} duality the statement of Proposition~\ref{stabilization-in-order2-tower} holds for a single stabilization by taking the connected sum of $X$ with any
simply connected closed $4$--manifold other than $S^4$.
\end{rem}

From the observations just before Proposition~\ref{dual-in-order2-tower}, the existence of any non-trivial order $0$ intersections between any $A_i$ and any $2$--spheres in $X$ implies that the obstruction to pulling apart the $A_i$ lives in a finite quotient of $\Lambda_2(4)$.
Returning to our goal of defining the INT$_2$ relations which clarify Conjecture~\ref{conj:INT},
we will consider settings where the target for $\lambda_2(A)$ is potentially infinite.
   
\subsubsection{Linear INT$_2$ relations}\label{sec:linear-INT2} 
Assume first that all order $0$ non-repeating intersections $\lambda_0$ on $\pi_2X$ vanish.

Let $i,j,k,l$
denote distinct indices in $\{1,2,3,4\}$.

Suppose that $W_{(i,j)}$ is an order $1$ Whitney disk in $\cW$, and that $W_{(i,j)}'$ is a different choice of order $1$ Whitney disk
with the same boundary as $W_{(i,j)}$ such that all intersections $W_{(i,j)}'\cap A_k$ and $W_{(i,j)}'\cap A_l$ are paired by order $2$
Whitney disks. Then replacing $W_{(i,j)}$ by $W_{(i,j)}'$, and replacing the order $2$ Whitney disks supported by $W_{(i,j)}$ with those supported by $W_{(i,j)}'$, changes $\cW$ to another order $2$ non-repeating Whitney tower $\cW'$ on $A$.
The union of $W_{(i,j)}$ with $W_{(i,j)}'$ along their common boundary is a $2$--sphere $S_{(i,j)}=W_{(i,j)}\cup W'_{(i,j)}$ with
$\lambda_0(S_{(i,j)}, A_k, A_l)=0\in\Lambda_0((i,j),k,l)$
as pictured (schematically) in Figure~\ref{INT2-ij}.
\begin{figure}[ht!]
         \centerline{\includegraphics[scale=.55]{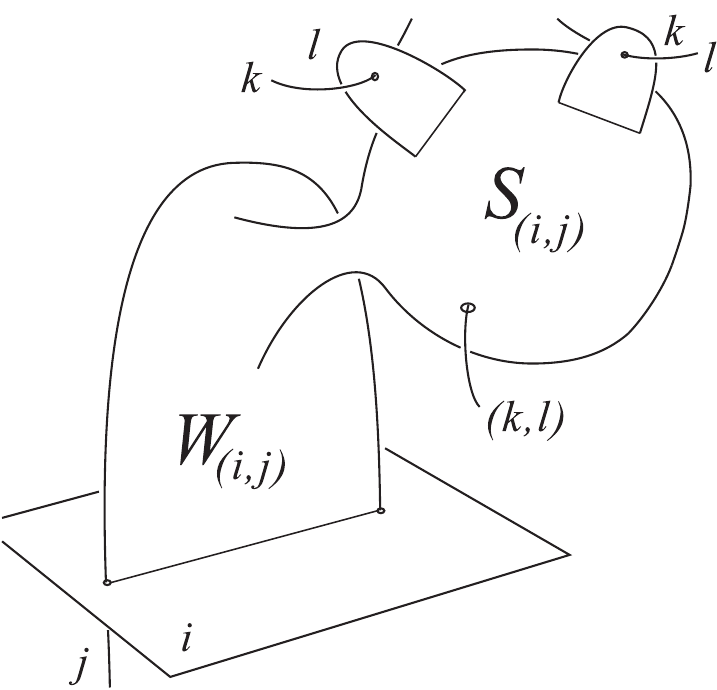}}
         \caption{Changing the interior of $W_{(i,j)}$ to $W'_{(i,j)}$
         corresponds to tubing $W_{(i,j)}$ into a $2$--sphere $S_{(i,j)}$.
         Only intersections which contribute to the difference $\lambda_2(\cW)-\lambda_2(\cW')\in\Lambda_2(4)$
         are shown.}
         \label{INT2-ij}
\end{figure}

Via the map $\Lambda_1((i,j),k,l)\to \Lambda_2(4)$ induced by sending ${(i,j)}\,-\!\!\!\!<^{\,k}_{\,l}$ to $^j_i>\!\!\!\!-\!\!\!\!-\!\!\!\!\!-\!\!\!\!<^{\,k}_{\,l}$, the corresponding change $\lambda_2(\cW)-\lambda_2(\cW')\in\Lambda_2(4)$ is equal to the image of the order $1$ non-repeating intersection invariant
$\lambda_1(S_{(i,j)},A_k,A_l)$, which is defined in
$\Lambda_1((i,j),k,l)$ since the vanishing of $\lambda_0$ means that all INT$_1$ relations are trivial.

Similarly changing the interiors of any number of the order 1 Whitney disks in $\cW$ leads to the following definition, which makes Conjecture~\ref{conj:INTm=4} precise in this setting:
\begin{defn}\label{def:INT2-vanishing-lambda0}
For a quadruple of 2--spheres $A=A_1,A_2,A_3,A_4\looparrowright X$ with $\lambda_1(A)=0$, with $X$ simply connected and $\lambda_0$ vanishing on $\pi_2X$, 
define the order $2$ intersection relations $\mathrm{INT}_2(A)<\Lambda_2(4)$ to be the subgroup generated by
$$
\,^j_i>\!\!\!-\,\lambda_1(S_{(i,j)},A_k,A_l)
$$ 
over all choices of distinct $i,j,k,l$ and all representatives $S_{(i,j)}$ of $\pi_2X$.
\end{defn}

This definition of $\mathrm{INT}_2(A)$ describes all possible changes in the order $2$ intersections due to choices of Whitney disks
for fixed choices of boundaries of order $1$ Whitney disks (up to isotopy), so by Proposition~\ref{prop:fixed-immersion} what remains to be done to confirm Conjecture~\ref{conj:INTm=4} in this case is to show that $\lambda_2(\cW)\in\Lambda_2(4)/\mathrm{INT}_2(A)$ is independent of the choice of order $1$ Whitney disk boundaries.

\subsubsection{Computing the linear INT$_2$ relations}\label{subsec:computing-linear-INT2}
In this setting (where $\lambda_0$ vanishes on $\pi_2X$), the subgroup $\mathrm{INT}_2(A)$ can be computed as follows:

For a basis $S^{\alpha}$ for the rank $r$ $\Z$-module $\pi_2X$
(mod torsion), and integers $a^{\alpha}_{ij}$
defined by
$$
\lambda_1(S^{\alpha}_{(i,j)},A_k,A_l)=a^{\alpha}_{ij}\,\,(i,j)-\!\!\!\!\!-\!\!\!<^{\,k}_{\,l},
$$
the $\mathrm{INT}_2(A)$ relations are described as the image of the linear map
$\Z^{6r}\rightarrow \Z^2$ given in the basis $\{t_1, t_2\}=\{
\,\,^2_1>\!\!\!-\!\!\!\!-\!\!\!\!\!-\!\!\!<^{\,3}_{\,4} \,\, , 
\,\,^3_1>\!\!\!-\!\!\!\!-\!\!\!\!\!-\!\!\!<^{\,2}_{\,4}\,\}
$
 by:
$$
\begin{pmatrix}
  x_{12} \\
  x_{34} \\
  x_{13} \\
  x_{24} \\
  x_{14} \\
  x_{23}
\end{pmatrix}\mapsto
\begin{pmatrix}
  a_{12} & a_{34} & 0 & 0 & a_{14} & a_{23} \\
  0 & 0 & a_{13} & a_{24} & a_{14} & a_{23}
\end{pmatrix}
\begin{pmatrix}
  x_{12} \\
  x_{34} \\
  x_{13} \\
  x_{24} \\
  x_{14} \\
  x_{23}
\end{pmatrix}
$$
where the multiplication of entries is vector inner product.

Examples in this setting realizing any coefficient matrix can be constructed by attaching $2$-handles to $B^4$ along $0$-framed links
in $S^3$
with vanishing linking matrix. 	

\subsubsection{Quadratic INT$_2$ relations}\label{sec:quadratic-INT2} 
Now assume that the $A_i$ represent elements in the radical of $\lambda_0$ on $\pi_2X$, but that $\lambda_0$ may otherwise be non-trivial.

We continue to investigate changes in order $2$ intersections due to choices of interiors of Whitney disks in $\cW$ supported by $A$.
Changing the interior of a Whitney disk 
$W_{(i,j)}$ to $W_{(i,j)}'$ along their common boundary again leads to a $2$--sphere $S_{(i,j)}=W_{(i,j)}\cup W'_{(i,j)}$ whose order $1$ intersections with $A_k,A_l$ determine $\lambda_2(\cW)-\lambda_2(\cW')\in\Lambda_2(4)$, but 
the order $1$ invariant
$\lambda_1(S_{(i,j)},A_k,A_l)$ that we want to use to measure this change may now itself have indeterminacies coming from non-trivial order $0$ intersections between $S_{(i,j)}$ and any $2$--spheres in $X$.

Specifically, $\lambda_1(S_{(i,j)},A_k,A_l)$ takes values in $\Lambda_1((i,j),k,l)$
modulo $\mathrm{INT}_1(S_{(i,j)},A_k,A_l)$, where the INT$_1(S_{(i,j)},A_k,A_l)$ relations are:
\begin{equation}
^{\;\;\;k}_{(i,j)}>\!\!\!-\,\lambda_0(S_{((i,j),k)},A_l)=0
\label{rel1}
\end{equation}
\begin{equation}
 ^{(i,j)}_{\;\;\;l}>\!\!\!-\,\lambda_0(S_{(l,(i,j))},A_k)=0 \label{rel2}
\end{equation}
\begin{equation}
_k^l>\!\!\!-\,\lambda_0(S_{(k,l)},S_{(i,j)})=0. \label{rel3}
\end{equation}
Note that the first two relations are empty by our assumption that the $A_i$ have vanishing order $0$ intersections
with all $2$--spheres. The third relation corresponds to indeterminacies in $\lambda_1(S_{(i,j)},A_k,A_l)$ due to the choice of interiors of order $1$ Whitney disks
pairing $A_k\cap A_l$, so computing with the order $1$ Whitney disks $W_{(k,l)}$ in $\cW$ determines a lift 
$\lambda^{\cW}_1(S_{(i,j)},A_k,A_l)\in \Lambda_1((i,j),k,l)$. 
Mapping ${(i,j)}\,-\!\!\!<^{\,k}_{\,l}$ to $^j_i>\!\!\!-\!\!\!\!-\!\!\!\!\!-\!\!\!<^{\,k}_{\,l}$, we have: 
$$
\lambda_2(\cW)-\lambda_2(\cW')=\,\,^j_i>\!\!\!-\,\lambda_1^{\cW}(S_{(i,j)},A_k,A_l)\in\Lambda_2(4).
$$

Now consider changing both $W_{(i,j)}$ to $W_{(i,j)}'$, and some
$W_{(k,l)}$ to $W_{(k,l)}'$ in as illustrated in
Figure~\ref{INT2-ijkl} (recall that $i,j,k,l$ are distinct). The
resulting change 
$\Delta^{\cW}(S_{(i,j)},S_{(k,l)}):=\lambda_2(\cW)-\lambda_2(\cW')\in\Lambda_2(4)$ can be expressed
as $\Delta^{\cW}(S_{(i,j)},S_{(k,l)})=$
$$ ^j_i>\!\!\!-\,\lambda_1^{\cW}(S_{(i,j)},A_k,A_l)\,+\,
{^j_i>\!\!\!-\,}\lambda_0(S_{(i,j)},S_{(k,l)})\,-\!\!\!<^{\,k}_{\,l}  \,+\,\lambda_1^{\cW}(A_i,A_j,S_{(k,l)})\,-\!\!\!<^{\,k}_{\,l}
$$
Here the $2$--sphere $S_{(k,l)}$ determined by $W_{(k,l)}$ and
$W_{(k,l)}'$ contributes the right-hand term $\lambda_1^{\cW}(A_i,A_j,S_{(k,l)})$ just as discussed above for
$S_{(i,j)}$, but now there is also a ``cross-term'' coming from
order $0$ intersections between $S_{(i,j)}$
and $S_{(k,l)}$. As in the previous paragraph
$\lambda_1^{\cW}(S_{(i,j)},A_k,A_l)$ and
$\lambda_1^{\cW}(A_i,A_k,S_{(k,l)})$ are lifts of the
corresponding order $1$ invariants. The 
three homotopy invariants $\lambda_1(S_{(i,j)},A_k,A_l)$, $\lambda_1(A_i,A_j,S_{(k,l)})$,
and $\lambda_0(S_{(i,j)},S_{(k,l)})$ are independent, so the given expression for $\Delta^{\cW}(S_{(i,j)},S_{(k,l)})$
only depends on $\cW$ and the homotopy classes of $S_{(i,j)}$ and $S_{(k,l)}$. 
\begin{figure}[ht!]
         \centerline{\includegraphics[scale=.55]{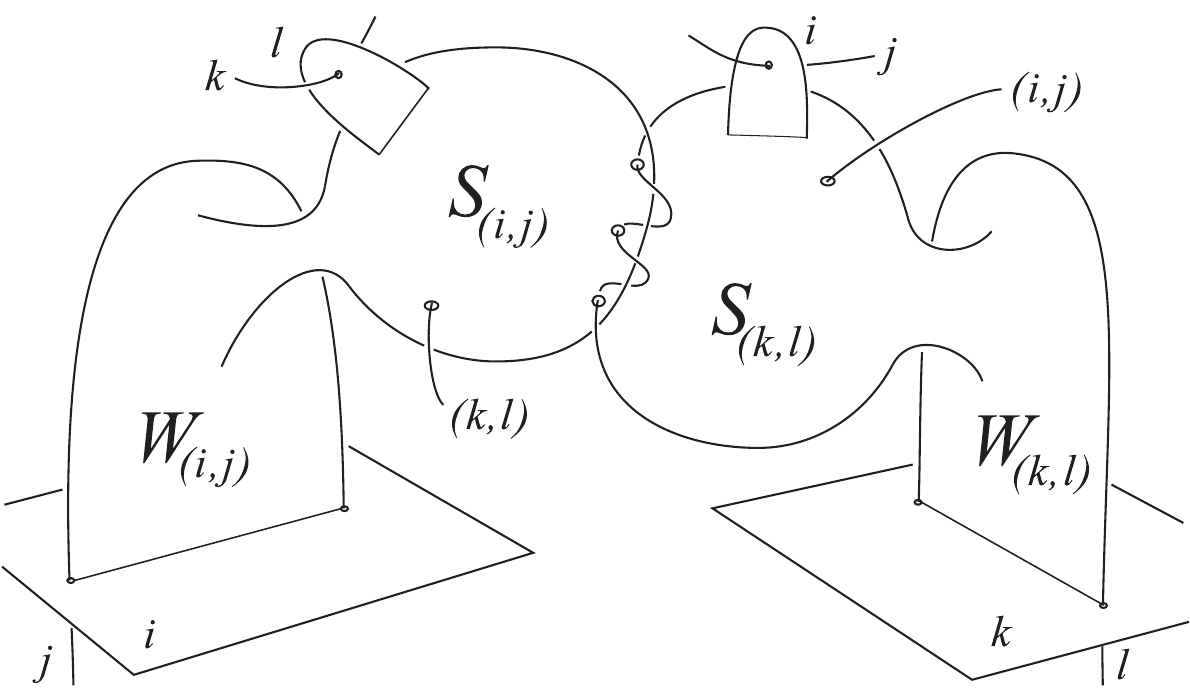}}
         \caption{A schematic illustration of how order $0$ intersections between 
         $S_{(i,j)}$ and $S_{(k,l)}$ can contribute order $2$ indeterminacies.
         (Only relevant intersections are shown.)}
         \label{INT2-ijkl}
\end{figure}

Observe that, since the intersection invariants sum over contributions from the Whitney disks, this entire
discussion applies word for word to changing \emph{all} the first
order Whitney disks $W_{(i,j)}$ on $A_i$ and $A_j$, and \emph{all}
the first order Whitney disks $W_{(k,l)}$ on $A_k$ and $A_l$;
with the $2$--spheres $S_{(i,j)}$ and $S_{(k,l)}$ interpreted as
sums (geometrically: unions) of the $2$--spheres determined
by each pair of Whitney disks.

\begin{defn}\label{def:INT2-A-in-radical}
For a quadruple of 2--spheres $A=A_1,A_2,A_3,A_4\looparrowright X$ with $X$ simply connected and  
$A$ in the radical of $\lambda_0$ on $\pi_2X$, 
define the order $2$ intersection relations $\mathrm{INT}_2^{\cW}(A)<\Lambda_2(4)$ to be the sub\emph{set}
$$\displaylines{
\mathrm{INT}_2^{\cW}:=\bigcup
\{-\Delta^{\cW}(S_{(1,2)},S_{(3,4)})-\Delta^{\cW}(S_{(1,3)},S_{(2,4)})
-\Delta^{\cW}(S_{(1,4)},S_{(2,3)})\}\hfill\break\cr
\hfill\subset\Lambda_2(4).}
$$
where $(i,j)$, $(k,l)$ vary over the pair-choices $(1,2),(3,4)$, and
$(1,3),(2,4)$ and $(1,4),(2,3)$; and where
$S_{(i,j)}$ and $S_{(k,l)}$ vary over all (homotopy
classes of) $2$--spheres in $X$.
\end{defn}

Note that, as defined, $\mathrm{INT}_2^{\cW}$ is only a subset of $\Lambda_2(4)$.

Since the above construction can be carried out simultaneously for the three pair-choices,
it follows that if $\lambda_2(\cW)\in \mathrm{INT}_2^{\cW}$, then it can be arranged that the $A_i$ 
support $\cW'$ with $\lambda_2(\cW')=0\in\Lambda_2(4)$, so the $A_i$ can be
pulled apart.

Since $\mathrm{INT}^{\cW}_2$ always contains the zero element of
$\Lambda_2(4)$,
the statement of Conjecture~\ref{conj:INTm=4} makes sense, with $\mathrm{INT}^{\cW}_2$ taking the place of 
$\mathrm{INT}_2(A)$. It would be desirable to have a formulation of the general INT$_2$ relations just in terms of $A$,
rather than $\cW$.

In the case that all order $0$ intersections vanish on $\pi_2X$,
then $\mathrm{INT}^{\cW}_2$ reduces to the subgroup $\mathrm{INT}_2(A)<\Lambda_2(4)$
of Definition~\ref{def:INT2-vanishing-lambda0}.

\subsubsection{Computing the quadratic $\mathrm{INT}_2$ relations}\label{subsubsec:quadratic-INT2-maps}
In this setting, $\mathrm{INT}^{\cW}_2$ can be computed as follows:

For a basis $S^{\alpha}$ for the rank $r$ $\Z$-module $\pi_2X$
(mod torsion), let $Q=q^{\alpha\beta}=\lambda_0(S^{\alpha},S^{\beta})$ denote the
intersection matrix. For integers $a^{\alpha}_{ij}$
defined by
$$
\lambda_1^{\cW}(S^{\alpha}_{(i,j)},A_k,A_l)=a^{\alpha}_{ij}\,\,(i,j)-\!\!\!\!\!-\!\!\!<^{\,k}_{\,l}
$$
we have the formula
$$
\begin{array}{rl}
 \Delta^{\cW}&\kern-8pt(\sum_{\alpha}x_{ij}^{\alpha}S_{(i,j)}^{\alpha},\sum_{\beta}x_{kl}^{\beta}S_{(k,l)}^{\beta})
 =\\
&= \sum_{\alpha}x_{ij}^{\alpha}a_{ij}^{\alpha}+ \sum_{\beta}x_{kl}^{\beta}a_{kl}^{\beta}+\sum_\alpha\sum_\beta x_{ij}^{\alpha}x_{kl}^{\beta}q^{\alpha\beta}\\
   &= \quad\quad x_{ij}\cdot a_{ij}+x_{kl}\cdot
   a_{kl}+x_{ij}Qx^T_{kl}
\end{array}
$$
where the $x_{uv}$ and $a_{uv}$ are vectors in $\Z^r$. 

Using the basis 
$\{t_1, t_2\}=\{
\,\,^2_1>\!\!\!-\!\!\!\!-\!\!\!\!\!-\!\!\!<^{\,3}_{\,4} \,\, , 
\,\,^3_1>\!\!\!-\!\!\!\!-\!\!\!\!\!-\!\!\!<^{\,2}_{\,4}\,\}
$
for 
$\Lambda_2(4)$,
computing $\mathrm{INT}^{\cW}_2$ amounts to determining the image
of the map $\Z^{6r}\rightarrow \Z^2$:
$$
\begin{pmatrix}
  x_{12} \\
  x_{34} \\
  x_{13} \\
  x_{24} \\
  x_{14} \\
  x_{23}
\end{pmatrix}\mapsto
\begin{pmatrix}
  a_{12} & a_{34} & 0 & 0 & a_{14} & a_{23} \\
  0 & 0 & a_{13} & a_{24} & a_{14} & a_{23}
\end{pmatrix}
\begin{pmatrix}
  x_{12} \\
  x_{34} \\
  x_{13} \\
  x_{24} \\
  x_{14} \\
  x_{23}
\end{pmatrix}
+ \begin{pmatrix}
  x_{12}Qx^T_{34}+x_{14}Qx^T_{23} \\
  x_{13}Qx^T_{24}+x_{14}Qx^T_{23}
\end{pmatrix}
$$
where the multiplication of entries is vector inner product.

For example, in the easiest case where just a single
$2$--sphere generator $S$ has non-trivial self-intersection number
$\lambda_0(S,S')=q\neq 0\in\Z$,
we have that $\mathrm{INT}^{\cW}_2$ is 
the image of the map
$\Z^6\rightarrow \Z^2$ given by:
$$
\begin{pmatrix}
  x_{12} \\
  x_{34} \\
  x_{13} \\
  x_{24} \\
  x_{14} \\
  x_{23}
\end{pmatrix}\mapsto
\begin{pmatrix}
  a_{12} & a_{34} & 0 & 0 & a_{14} & a_{23} \\
  0 & 0 & a_{13} & a_{24} & a_{14} & a_{23}
\end{pmatrix}
\begin{pmatrix}
  x_{12} \\
  x_{34} \\
  x_{13} \\
  x_{24} \\
  x_{14} \\
  x_{23}
\end{pmatrix}
+ q\begin{pmatrix}
  x_{12}x_{34}+x_{14}x_{23} \\
  x_{13}x_{24}+x_{14}x_{23}
\end{pmatrix}
$$
Examples in this setting realizing any coefficient matrix can be constructed by attaching $2$-handles to $B^4$ along links in $S^3$, and the following questions arise: Is this image always a subgroup of $\Z\oplus\Z$? Konyagin and Nathanson have shown in \cite[Thm.3]{Kon} that the image always projects to subgroups in each $\Z$-summand. And under what conditions will the image be all of
$\Z\oplus\Z$? This would imply that the $A_i$ can be pulled apart. What about
analogous questions in the general case where the equations are coupled by the intersection
matrix $Q$?


\Addresses


\begin{thebibliography}{blah}
\bibitem{B2} {\bf D Bar-Natan}, { \it Vassiliev homotopy string link invariants},
J. Knot Theory Ramifications 4 (1995) 13--32.

\bibitem{Casson}{\bf A Casson}, {\it Three Lectures on new infinite constructions in $4$--dimensional manifolds},
Notes prepared by L Guillou, Prepublications Orsay 81T06 (1974).






\bibitem{CST} {\bf J Conant}, {\bf R Schneiderman}, {\bf P\ Teichner}, {\it Jacobi
identities in low-dimensional topology}, Compositio Mathematica
143 Part 3 (2007) 780--810.

\bibitem{CST0}  {\bf J Conant}, {\bf R Schneiderman}, {\bf P\ Teichner}, {\it Higher-order intersections in low-dimensional topology}, Proc. Natl. Acad. Sci. USA vol. 108, no. 20, (2011) 8131--8138.



\bibitem{CST1} {\bf J Conant}, {\bf R Schneiderman}, {\bf P\ Teichner}, {\it Whitney tower concordance of classical links}, 
Geom. Topol. (2012) 16 (2012) 1419--1479.

\bibitem{CST2} {\bf J Conant}, {\bf R Schneiderman}, {\bf P\ Teichner}, {\it Milnor invariants and twisted Whitney towers}, 
J. Topology,  Volume 7, Issue 1,  (2013)
187--224. 

\bibitem{CST3} {\bf J Conant}, {\bf R Schneiderman}, {\bf P\ Teichner}, {\it Tree homology and a conjecture of Levine}, 
Geom. Topol. 16 (2012) 555--600.




\bibitem{CST4}  {\bf J Conant}, {\bf R Schneiderman}, {\bf P\ Teichner}, {\it Universal quadratic forms and Whitney tower intersection invariants}, Geom. Topol. Monographs 18 (2012) 35--60.
















\bibitem{Dw} {\bf W Dwyer}, {\it Homology, Massey products and maps between groups },
J. Pure Appl. Alg. 6 (1975) 177--190.


\bibitem{Edwards}  {\bf R Edwards}, {\it The solution of the four-dimensional annulus conjecture (after Frank Quinn)},
Four-manifold Theory, Contemporary Mathematics AMS Vol 35 (1982)
211--264.

\bibitem{F}  {\bf M\ Freedman}, {\it The topology of four-dimensional manifolds},
J. Diff. Geom. 17 (1982) 357--453.

\bibitem{FQ}  {\bf M\ Freedman}, {\bf F Quinn}, {\it The topology of
$4$--manifolds}, Princeton Math. Series 39 Princeton, NJ, (1990).

\bibitem{Gi} {\bf C H Giffen}
{\it Link concordance implies link homotopy}, Math. Scand. 45
(1979) 243--254.

\bibitem{Go} {\bf D Goldsmith} {\it Concordance implies homotopy
for classical links in $M^3$}, Comment. Math. Helvetici. 54 (1979)
347--355.

%
\bibitem{GL} {\bf S Garoufalidis}, {\bf J Levine}, {\it Homology surgery
and invariants of $3$--manifolds}, Geom. Topol. Vol. 5 (2001) 75--108.
%

\bibitem{HL} {\bf N\ Habegger}, {\bf X Lin}, {\it The classification of links up to homotopy},
Journal of the AMS Vol. 3 (1990) 389-419.

\bibitem{HM} {\bf N\ Habegger}, {\bf G Masbaum}, {\it The Kontsevich integral and Milnor's invariants},
Topology 39  (2000) 1253-1289.




\bibitem{Kirby} {\bf R Kirby}, {\it The Topology of $4$-manifolds},
Lecture Notes in Math. 1374 Springer--Verlag, Berlin--Heidelberg--New York (1989).



\bibitem{Ko} {\bf K Kobayashi}, {\it On a homotopy version of $4$--dimensional Whitney's lemma},
Math. Sem. Notes 5 (1977) 109--116.

\bibitem{Koj} {\bf S Kojima}, {\it Milnor's $\overline{\mu}$--invariants,
Massey products and Whitney's trick in $4$--dimensions}, Topol. Appl. 16 (1983) 43--60.

\bibitem{Kon} {\bf S Konyagin}, {\bf M Nathanson}, {\it Sums of products of congruence classes and of arithmetic progressions}, 
Int. J. Number Theory Vol. 5 Issue 4 (2009) 625--634.
 
%

%

%















\bibitem{MKS} {\bf W Magnus}, {\bf A Karass}, {\bf D Solitar}, {\it Combinatoral group theory},
Dover Publications, Inc. (1976).

\bibitem{M1} {\bf J\ Milnor}, {\it Link groups}, Annals of Math. 59 (1954) 177--195.

\bibitem{M2} {\bf J\ Milnor}, {\it Isotopy of links}, Algebraic geometry and topology
Princeton Univ. Press (1957).

\bibitem{Ma} {\bf Y Matsumoto}, {\it Secondary intersectional properties of $4$--manifolds and
Whitney's trick},  Proceedings of Symposia in Pure mathematics
Vol. 32 Part 2 (1978) 99--107.





\bibitem{Quinn}  {\bf F Quinn}, {\it Ends of maps, III: dimensions 4 and 5},
J. Diff. Geom. 17 (1982) 503--521.


\bibitem{S1} {\bf R\ Schneiderman}, {\it Whitney towers and Gropes in 4--manifolds},
Trans. Amer. Math. Soc. 358 (2006), 4251-4278.

\bibitem{S2} {\bf R\ Schneiderman}, {\em Simple Whitney towers, half-gropes and the Arf invariant of a
knot}, Pacific J. Math. Vol. 222 No. 1 Nov
(2005).

\bibitem{S3} {\bf R\ Schneiderman}, {\it Stable concordance of knots in $3$--manifolds}, 
Alg. and Geom. Topology 10 (2010) 37--432. 


\bibitem{ST1} {\bf R\ Schneiderman}, {\bf  P\ Teichner}, {\it Higher order intersection numbers of $2$--spheres in 4--manifolds},
Alg. and Geom. Topology 1 (2001) 1--29.

\bibitem{ST2} {\bf R\ Schneiderman}, {\bf P\ Teichner},
{ \em Whitney towers and the Kontsevich integral}, Proceedings of
a conference in honor of Andrew Casson, UT Austin 2003, Geometry
and Topology Monograph Series Vol. 7 (2004), 101--134.



\bibitem{T} {\bf P Teichner}, {\it Symmetric surgery and boundary link
maps}, Math. Ann. 312 (1998) 717--735.



\bibitem{W}  {\bf T\ Wall}, {\it Surgery on Compact Manifolds},
London Math.Soc.Monographs~1, Academic Press~1970 or Second
Edition, edited by A. Ranicki, Math. Surveys and Monographs 69
A.M.S.


\bibitem{Y} {\bf M Yamasaki}, {\it Whitney's trick for three $2$--dimensional
homology classes of $4$--manifolds}, Proc. Amer. Math. Soc. 75
(1979) 365--371.

\end{thebibliography}
\end{document}